\documentclass[a4paper]{article}
\usepackage{fullpage}
\usepackage[utf8]{inputenc}
\usepackage[T1]{fontenc}
\usepackage[english]{babel}
\usepackage{graphicx}
\usepackage{amsmath,amsfonts,amssymb,amsthm}
\usepackage{dsfont}
\usepackage{bm}
\usepackage{natbib}
\usepackage{cmbright}
\usepackage[table]{xcolor}
\usepackage{mdframed}
\usepackage{booktabs}
\usepackage{hyperref}
\newcommand{\verifrepo}{\href{https://github.com/ydecastro/TENSOR-KSS}{\texttt{github.com/ydecastro/TENSOR-KSS}}}
\definecolor{burgundy}{rgb}{0.5, 0.0, 0.13}
\definecolor{camel}{rgb}{0.76, 0.6, 0.42}
\definecolor{chamoisee}{rgb}{0.63, 0.47, 0.35}
\definecolor{grey1}{RGB}{128,128,128}
\hypersetup{colorlinks=true,linkcolor=burgundy,citecolor=chamoisee,urlcolor=burgundy,hypertexnames=false}
\numberwithin{equation}{section}

\newtheorem{theorem}{Theorem}
\definecolor{mastertheoremfill}{rgb}{0.985,0.97,0.93}
\newmdenv[
    backgroundcolor=mastertheoremfill,
    linecolor=burgundy,
    linewidth=0.8pt,
    topline=true,bottomline=true,rightline=true,leftline=true,
    innertopmargin=6pt,innerbottommargin=6pt,
    innerleftmargin=8pt,innerrightmargin=8pt,
    skipabove=8pt,skipbelow=6pt
]{mastertheorembox}
\newtheorem{proposition}{Proposition}
\newtheorem{corollary}[proposition]{Corollary}
\newtheorem{remark}{Remark}

\newtheorem{lemma}{Lemma}

\newcommand{\ba}{\boldsymbol{a}}

\newcommand{\R}{\mathds{R}}
\newcommand{\E}{\mathds{E}}
\renewcommand{\P}{\mathds{P}}
\newcommand{\meanTensor}{\bm{\sigma}^\star}
\newcommand{\putativemeanTensor}{\bm{\sigma}}

\newcommand{\unnormalizedTensor}{\bm{\tau}}
\newcommand{\Tensors}{\mathcal{T}(k,d)}
\newcommand{\sphereTensors}{\mathds S(k,d)}
\newcommand{\tensors}{\mathcal{T}}
\newcommand{\obs}{{\bm Y}}
\newcommand{\noise}{{\bm W}}

\newcommand{\var}{{\rm var}}

\newcommand{\1}{\mathds{1}}

\title{Non-asymptotic Tail Bounds for the Kostlan--Shub--Smale Field: Tensor PCA and Spherical $k$-Spin Complexity}

\usepackage{authblk}
\author[1]{Jean-Marc Aza\"{i}s}
\author[2]{Federico Dalmao}
\author[3,4]{Yohann De Castro}
\affil[1]{{\small Institut de Math\'ematiques de Toulouse\\ Universit\'e  de Toulouse, France.}}
\affil[2]{{\small DMEL, CENUR Litoral Norte\\  Universidad de la Rep\'ublica, Salto, Uruguay.}}
\affil[3]{{\small Institut Camille Jordan, CNRS UMR 5208\\  \'Ecole Centrale Lyon, France.}}
\affil[4]{{\small Institut Universitaire de France~(IUF)}}
\date{Preprint as of~\today}

\begin{document}
\maketitle
\begin{abstract}
This paper builds a hierarchy of explicit, non-asymptotic tail bounds for the supremum of the Kostlan--Shub--Smale (KSS) random field on the sphere, and applies it to two problems: Spiked Tensor PCA and the landscape of the spherical $k$-spin model. For Tensor PCA, we study the non-asymptotic statistical limits of estimating a rank-$R$ symmetric signal tensor of order~$k\ge 3$ and dimension~$d\ge 3$ from a single Gaussian observation at signal-to-noise ratio~$\lambda$, through the \emph{profile maximum likelihood estimator}, the MLE restricted to normalized rank-$R$ tensors of coherence at least~$\kappa$. Our analysis uses a single reduction: a deterministic geometric inequality (the Tube Method) and a rank-reduction step bound the estimation error by the supremum of the canonical KSS field, which the Kac--Rice formula turns into a Gaussian integral against the expected absolute characteristic polynomial of a shifted Gaussian Orthogonal Ensemble, controlled in turn by the four explicit tail bounds of our hierarchy (three from a Mehta--Fyodorov representation, one from a Ben Arous--Dembo--Guionnet large deviation). The same reduction yields two results, each with explicit constants. For estimation, a finite-$(k,d)$ error bound recovers the asymptotically optimal rate~$\sqrt{d\log k}$ of Perry, Wein and Bandeira, with explicit dependence on the rank~$R$ and the coherence~$\kappa$. For the landscape, a two-sided non-asymptotic bracketing of the annealed complexity of the spherical $k$-spin Hamiltonian recovers the Auffinger--Ben Arous--\v{C}ern\'y complexity function in the high-dimensional limit.
\end{abstract}

\section{Introduction}
\label{sec:introduction}

\subsection{Tensor regression model and the profile MLE}
\label{sub:model}

Tensor regression generalizes linear regression to the setting where the response is a symmetric tensor $\obs$ of order $k\ge 2$ and dimension $d\ge 2$. The model reads
\begin{equation}
    \label{eq:generic_tensor_regression}
    \obs \;=\; \lambda \meanTensor + \noise\,,
\end{equation}
\begin{subequations}
where $\meanTensor$ is the signal tensor (with $\|\meanTensor\|_F=1$), $\lambda>0$ is the signal-to-noise ratio, and $\noise$ is a standard Gaussian symmetric tensor with distribution $\P_0$ whose density is proportional to $\exp(-\|\noise\|_F^2/2)$. Throughout, $\|\cdot\|_F$ denotes the Frobenius norm and $\langle\cdot,\cdot\rangle_\tensors$ the associated inner product on the space $\Tensors$ of symmetric tensors. We denote by $\mathfrak{S}_R$ the set of \emph{normalized} symmetric tensors of rank at most~$R$. Any $\putativemeanTensor\in\mathfrak{S}_R$ admits a rank-$R$ decomposition
\begin{equation}
    \label{eq:mean_tensor_rank_r}
    \putativemeanTensor \;=\; \sum_{j=1}^R a_j\,t_j^{\otimes k}\,,\qquad a_j\neq 0\,,\quad \|\putativemeanTensor\|_F^2=1\,,
\end{equation}
where $t_j^{\otimes k}$ denotes the $k$-fold outer product of $t_j\in\mathbb{S}^{d-1}$ with itself, and the vectors $t_j$ are pairwise distinct. 

\pagebreak[3]
\medskip

\noindent
The \emph{coherence} of $\putativemeanTensor$ is defined by
\begin{equation}
    \label{def:coherence}
    \kappa^2(\putativemeanTensor) \;:=\; \max\Big\{\lambda_{\min}(\bm{G})\,\Big|\,G_{ij}=\langle t_i,t_j\rangle^k\ \text{and \eqref{eq:mean_tensor_rank_r} holds}\Big\}\,,
\end{equation}
where the maximum is taken over all rank-$R$ decompositions of $\putativemeanTensor$. The coherence satisfies $\kappa(\putativemeanTensor)\in(0,1]$, with $\kappa=1$ when the components are pairwise orthogonal. Small values of~$\kappa$ correspond to near-collinear components, which are statistically harder to disentangle. 

\paragraph{Profile Maximum Likelihood Estimator}
The \emph{profile Maximum Likelihood Estimator} (profile MLE) maximizes the inner product with the observation $\obs$ over the slice of $\mathfrak{S}_R$ on which the coherence is at least~$\kappa$:
\end{subequations}
\begin{equation}
    \label{eq:MLE_definition}
    \hat{\lambda} \;=\; \langle \obs, \hat{\putativemeanTensor} \rangle_\tensors\,,
    \qquad
    \hat{\putativemeanTensor} \;\in\; \operatorname*{arg\,max}
    \Big\{ \langle \obs, \putativemeanTensor \rangle_\tensors\,\Big|\,\putativemeanTensor\in\mathfrak{S}_R,\,\kappa(\putativemeanTensor)\ge\kappa\Big\}\,,
\end{equation}
and $\hat{\putativemeanTensor}$ is well defined: the feasible set
\[
\mathcal{C}_{R,\kappa}:=\Big\{\putativemeanTensor\in\mathfrak{S}_R\,\Big|\,\kappa(\putativemeanTensor)\ge\kappa\Big\}
\]
is \emph{compact}: bounded in the unit sphere of $\Tensors$, and closed because the coherence floor $\kappa(\putativemeanTensor)\ge\kappa>0$ bounds the coefficients ($\|\ba\|_2\le 1/\kappa$, Lemma~\ref{lem:reduce_rank_one}) while the components range over the compact sphere $\mathbb S^{d-1}$, ruling out the border-rank degenerations through which $\{\mathrm{rank}\le R\}$ fails to be closed for $k\ge 3$. The continuous Gaussian functional $\putativemeanTensor\mapsto\langle\obs,\putativemeanTensor\rangle_\tensors$ therefore attains an almost surely unique maximum on $\mathcal C_{R,\kappa}$ \citep[Theorem~3]{lifshits1983absolute,tsirel1976density}; concentrating out $\lambda$ in the Gaussian log-likelihood of~\eqref{eq:generic_tensor_regression} leaves exactly~\eqref{eq:MLE_definition}, whence the name \emph{profile MLE}.

\subsection{Geometric reduction and the Kac--Rice integral}
\label{sub:reduction}

A deterministic geometric inequality (Lemma~\ref{lem:geometric_bound}), referred to as the \emph{Tube Method}, controls the estimation error by the supremum of the empirical noise process on the feasible manifold:
\[
    \|\hat{\putativemeanTensor}-\meanTensor\|_F^2 \;\le\; \frac{4\,\Gamma_{R,\kappa}}{\lambda}\,,
    \qquad
    \Gamma_{R,\kappa}:=\sup\Big\{|\langle\noise,\putativemeanTensor\rangle_\tensors|\,\Big|\,\putativemeanTensor\in\mathcal{C}_{R,\kappa}
    \Big\}\,,
\]
\begin{subequations}
with $\Gamma_{R,\kappa}$ referred to as the \emph{noise level} and, by continuity of $\putativemeanTensor\mapsto\langle\noise,\putativemeanTensor\rangle_\tensors$ and compactness of $\mathcal{C}_{R,\kappa}$, this supremum is attained. The rank-reduction inequality (see Lemma~\ref{lem:reduce_rank_one}) then bounds
\begin{equation}\label{eq:gamma_11}
\Gamma_{R,\kappa}\le\frac{\sqrt R}\kappa\,\Gamma_{1,1}\,,
\qquad
\Gamma_{1,1} \;=\; \sup_{\bm \theta\in\mathbb S^{d-1}}\Big\{|X(\bm \theta)|\Big\}\,,
\qquad X(\bm \theta):=\langle\noise,\bm \theta^{\otimes k}\rangle_\tensors\,,
\end{equation}
where the covariance $\E[X(\bm \theta)X(\bm v)]=\langle\bm \theta,\bm v\rangle^k$ identifies $X$ as the canonical \emph{Kostlan--Shub--Smale} (KSS) random field on the sphere. This is the field named in the title: every bound below concerns the upper tail of its supremum, for which we write
\begin{equation}\label{eq:def_widetilde_Gamma}
    \widetilde\Gamma_{1,1}(u)\;:=\;\P\Big\{\sup_{\bm\theta\in\mathbb S^{d-1}}X(\bm\theta)>u\Big\}\,.
\end{equation}
The two-sided supremum $\Gamma_{1,1}$ of~\eqref{eq:gamma_11} follows from the symmetry $X\stackrel d=-X$ of the centred field:
\begin{equation}\label{eq:two_sided_intro}
    \P\{\Gamma_{1,1}>u\}\;\le\;2\,\widetilde\Gamma_{1,1}(u)\,;
\end{equation}
the factor $2$ is the only cost of passing to the two-sided sup, and appears in the failure probability of Theorem~\ref{thm:main}.

\paragraph{Kac--Rice integral.}
Throughout, the \emph{Gaussian Orthogonal Ensemble} of size~$n$, denoted $\mathrm{GOE}(n)$, is the law of a real symmetric $n\times n$ random matrix $G=(G_{ij})$ whose entries are jointly Gaussian, centred, independent up to symmetry, with the Mehta normalization $\var(G_{ii})=1$ and $\var(G_{ij})=1/2$ for $i\neq j$; equivalently, the density on the space of real symmetric matrices is proportional to $\exp(-\mathrm{tr}(G^2)/2)$ and the law is invariant under conjugation by any orthogonal matrix~\citep{mehta}. The Kac--Rice formula \citep{azais2009level}, combined with the conditional GOE law of the Riemannian Hessian \citep[Lemma~3.2(b)]{auffinger2013}, reduces $\widetilde\Gamma_{1,1}(u)$ to a Gaussian integral against the expected absolute characteristic polynomial of a shifted GOE matrix. Defining
\begin{equation}\label{eq:KR_combined_intro}
    \delta_0(u)\;:=\;C_{k,d}\int_u^\infty \E\big[|\det(G_{d-1}-\rho x\,I_{d-1})|\big]\,\varphi(x)\,\mathrm{d}x\,,
\end{equation}
the Kac--Rice formula gives
\begin{equation}\label{eq:KR_main_bound_intro}
    \widetilde\Gamma_{1,1}(u)\;\le\;\delta_0(u)\,,
\end{equation}
with constants
\begin{align}
\label{eq:rho_ckd_intro}
    \rho:=\sqrt{\frac k{2(k-1)}}\,,\qquad C_{k,d}:=\frac{2\sqrt{\pi}}{\Gamma(d/2)}\,(k-1)^{\frac{d-1}2}\,,\qquad G_{d-1}\sim\mathrm{GOE}(d-1)
\end{align}
and $\varphi$ the standard normal density.
It is proved in \cite[Theorem~8.12]{azais2009level} that the bound~\eqref{eq:KR_main_bound_intro} is super-exponentially precise as $u\to\infty$, so that $\delta_0$ is asymptotically the right scale to describe~$\widetilde\Gamma_{1,1}$.

Section~\ref{sec:tail_bounds} makes $\delta_0$ explicit: on $[u_{\mathrm{IMF}},\infty)$ with $u_{\mathrm{IMF}}:=\sqrt{2d-1}/\rho$ (Szeg\H{o}'s bound on the largest root of $H_{d-1}$, Lemma~\ref{lem:Idc_positivity}), it admits the exact closed form $\delta_{\mathrm{exact}}$ via the Mehta--Fyodorov representation (Theorem~\ref{thm:imf_tail_exact}), with $\delta_0=\delta_{\mathrm{exact}}$ pointwise there; the relaxations $\delta_{\mathrm{IMF}},\delta_{\mathrm{SMF}},\delta_{\mathrm{SM}}$ are explicit upper bounds on appropriate sub-ranges.
\end{subequations}

\paragraph{Asymptotic baseline as $u\to\infty$, with $d$ fixed.}
\label{sub:baseline}
Replacing the expected determinant in~\eqref{eq:KR_combined_intro} by its leading monomial $(\rho x)^{d-1}$ gives, by Lemma~\ref{lem:equiv_tail}, the asymptotic equivalent
\begin{equation}\label{eq:asymptotic_baseline_approx}
    \delta_{\mathrm{bl}}(u)\;:=\;\frac{\sqrt{2}}{\Gamma(d/2)}\!\left(\frac{k}{2}\right)^{\frac{d-1}2}u^{d-2}\,e^{-u^2/2}\,,
    \qquad
    \delta_0(u)\;\sim\;\delta_{\mathrm{bl}}(u)\quad\text{as $u\to\infty$ with $d$ fixed}\,.
\end{equation}
This is the optimal tail rate against which the non-asymptotic bounds are measured; $\delta_{\mathrm{bl}}$ is not itself an upper bound on $\delta_0$, only the asymptotic equivalent.

\subsection{Hierarchy of four explicit tail bounds}
\label{sub:three_bounds}

We develop a hierarchy of four explicit non-asymptotic upper bounds on $\delta_0(u)$ and, by~\eqref{eq:KR_combined_intro}, on the excursion probability $\P\big\{\sup_{\bm \theta\in\mathbb S^{d-1}}X(\bm \theta)>u\big\}$, each valid above an explicit threshold:
\begin{equation*}
    u_{\mathrm{IMF}} := \frac{\sqrt{2d-1}}{\rho}\,,\qquad
    u_{\mathrm{SMF}} := 2\sqrt{d}\,,\qquad
    u_{\mathrm{SM}}  := \frac{32\sqrt{d-1}}{\rho}\,.
\end{equation*}
For $k,d\ge 3$ one has $u_{\mathrm{IMF}}\le u_{\mathrm{SMF}}\le u_{\mathrm{SM}}\le 32\sqrt{2d}$. Table~\ref{tab:tail_bounds} collects the four bounds with their validity ranges and closed forms, each derived in Section~\ref{sec:tail_bounds}: $\delta_{\mathrm{exact}}$ is the exact closed-form evaluation of $\delta_0$ (a finite Hermite-recurrence sum, Theorem~\ref{thm:imf_tail_exact}); $\delta_{\mathrm{IMF}}$ discards the negative term $-2\,\mathcal{I}_d^c(\rho x)\,H_{d-1}(\rho x)$ of the Mehta expansion (asymptotically sharp, Theorem~\ref{thm:imf_tail}); $\delta_{\mathrm{SMF}}$ further collapses the Hermite sums to their Szeg\H{o} monomial envelope, giving an inversion-friendly closed form (Theorem~\ref{thm:smf_tail}); and $\delta_{\mathrm{SM}}$ is an independent Ben Arous--Dembo--Guionnet/layer-cake bound not using the Mehta--Fyodorov algebra (Theorem~\ref{thm:sm_tail}). The first three are strictly nested, $\delta_{\mathrm{exact}}\le\delta_{\mathrm{IMF}}\le\delta_{\mathrm{SMF}}$ on $[u_{\mathrm{IMF}},\infty)$ (Theorems~\ref{thm:imf_tail_exact} and~\ref{thm:uniform_domination}), and dominate $\delta_{\mathrm{SM}}$ on $[u_{\mathrm{SM}},\infty)$ (Remark~\ref{rem:sm_domination}).
All four decay at the leading rate $u^{d-2}e^{-u^2/2}$ as $u\to\infty$ ($d$ fixed); their $d\to\infty$ prefactors are compared below.

\begin{table}[ht]
\centering
\renewcommand{\arraystretch}{1.5}
\resizebox{\textwidth}{!}{
\begin{tabular}{lll}
\toprule
\textbf{Method} & \textbf{Validity range} & \textbf{Closed-form expression} \\
\midrule
Baseline (not a bound) & $u\to\infty$, $d$ fixed & $\delta_{\mathrm{bl}}(u)\;=\;\dfrac{\sqrt{2}}{\Gamma(d/2)}\!\left(\dfrac{k}{2}\right)^{\frac{d-1}2}u^{d-2}e^{-u^2/2}$ \\
$\delta_{\mathrm{exact}}$ (canonical) & $u\in\R$ (exact) & $\delta_{\mathrm{exact}}(u)\;=\;2(k-1)^{\frac{d-1}2}\bigl[D_1(u)+D_2(u)+D_3(u)+D_4(u)\bigr]$, finite Hermite sum (Thm.~\ref{thm:imf_tail_exact}) \\
$\delta_{\mathrm{IMF}}$ & $u\ge u_{\mathrm{IMF}}$ & $\delta_{\mathrm{IMF}}(u)\;=\;2(k-1)^{\frac{d-1}2}\,\Big(\alpha_d\,\Phi_d(\rho,u)\,e^{-u^2/2}+\Psi_d(\rho,u)\,e^{-(1+\rho^2)u^2/2}\Big)$ \\
$\delta_{\mathrm{SMF}}$ & $u\ge u_{\mathrm{SMF}}$ & $\delta_{\mathrm{SMF}}(u)\;=\;4\,\alpha_d(2k)^{\frac{d-1}2}u^{d-2}e^{-u^2/2}+2^d\beta_d(k-1)^{\frac{d-1}2}u^{2d-3}e^{-3u^2/4}$ \\
$\delta_{\mathrm{SM}}$ & $u\ge u_{\mathrm{SM}}$ & $\delta_{\mathrm{SM}}(u)\;=\;2\,\dfrac{\sqrt{2}}{\Gamma(d/2)}\!\left(\dfrac{k}{2}\right)^{\frac{d-1}2}u^{d-2}\,e^{-u^2/2}\,\bigl(1+\eta_d(\rho,u)\bigr)$ \\
\bottomrule
\end{tabular}
}
\vspace{0.5em}
\caption{The four explicit non-asymptotic tail bounds on $\delta_0(u)$, with the asymptotic baseline $\delta_{\mathrm{bl}}$ (first row), which is the equivalent $\delta_0\sim\delta_{\mathrm{bl}}$ as $u\to\infty$ ($d$ fixed), not an upper bound. The exact bound $\delta_{\mathrm{exact}}$ (components $D_1,\dots,D_4$, Theorem~\ref{thm:imf_tail_exact}) coincides with $\delta_0$ for every $u\in\R$; the threshold $u_{\mathrm{IMF}}$ marks only the nesting range $\delta_{\mathrm{exact}}\le\delta_{\mathrm{IMF}}\le\delta_{\mathrm{SMF}}$. The constants $\alpha_d,\beta_d$ and functions $\Phi_d,\Psi_d,\eta_d$ are given in~\eqref{eq:expression_constants} and Section~\ref{sec:tail_bounds}.}
\label{tab:tail_bounds}
\end{table}

\paragraph{Explicit form of the constants of Table~\ref{tab:tail_bounds}.}
The constants originate in the Mehta--Fyodorov representation~\eqref{eq:fmehta}: $\alpha_d,\beta_d$ are the dominant Hermite coefficient and remainder envelope in the splitting $Q_d=\alpha_d H_{d-1}+\mathcal{R}_d$ (Lemma~\ref{lem:dominant_term}); $\Phi_d,\Psi_d$ arise in the IMF decomposition (Proposition~\ref{prop:IMF_decomp}); and $\eta_d$ is the layer-cake correction (Proposition~\ref{prop:det_bound}). Set
\[
    c_j\;:=\;(2^j j!\sqrt\pi)^{-1/2}\,,\qquad
    \Lambda\;:=\;2\rho^2-1\;=\;\tfrac1{k-1}\,,\qquad
    \mu_m\;:=\;\int_\R H_m(y)\,e^{-y^2/2}\,\mathrm{d}y\,;
\]
by parity, $\mu_{2p+1}=0$, and direct evaluation yields the closed form $\mu_{2p}=\sqrt{2\pi}\,(2p)!/p!$. With these conventions:
\begin{subequations}
    \label{eq:expression_constants}
\begin{itemize}
    \item The dominant Hermite coefficient $\alpha_d$ (i.e.\ the coefficient of $H_{d-1}(\nu)$ in the Mehta expansion of~$Q_d(\nu)$) is
    \begin{equation*}
    \alpha_d \;=\;
    \begin{cases}
        \tfrac12\sqrt{d/2}\,c_{d-1}\,c_d\,\mu_d & \text{if $d$ is even},\\[2pt]
        1/\mu_{d-1} & \text{if $d$ is odd}.
    \end{cases}
    \end{equation*}
    \item The remainder envelope constant $\beta_d$ (from the bound $|\mathcal{R}_d(\nu)|\le\beta_d(1+\nu^2)^{d-1}e^{-\nu^2/2}$ of~\eqref{eq:remainder_bound}) has the explicit form
    \begin{equation*}
        \beta_d \;=\; \underbrace{\sum_{j=0}^{d-1}c_j^2\,2^j\!\left(\frac{(2j)!}{j!}\right)^{\!2}}_{S_d\,:\ \text{squared-Hermite contribution}}
        \;+\;\sqrt{d/2}\,c_{d-1}\,c_d\,\frac{(2d-2)!}{(d-1)!}\,2^{d-1}\,\widetilde B_d\,,
    \end{equation*}
    where $\widetilde B_d$ is the explicit constant
    \begin{equation*}
        \widetilde B_d \;=\; \max\!\left(\frac{(2d)!\,2^{d+1}}{d!}\,,\; \frac{(2d)!}{d!}\!\int_0^\infty\!(1+y)^d\,e^{-y^2/2}\,\mathrm{d}y\right)\,,
    \end{equation*}
    coming from the proof of Lemma~\ref{lem:dominant_term}: it dominates the uniform Hermite-tail envelope constant $E_d$ of~\eqref{eq:Idc_uniform_env}, for which one has $E_d\le\frac{(2d)!\,2^{d+1}}{d!}\le\widetilde B_d$, so that $|\mathcal{I}_d^c(\nu)|\le\widetilde B_d\,(1+|\nu|)^{d-1}e^{-\nu^2/2}$ holds for every $\nu\ge 0$ (see proof of Lemma~\ref{lem:dominant_term}).
    \item The polynomial--exponential functions $\Phi_d(\rho,u)$, $\Psi_d(\rho,u)$ entering the IMF decomposition~\eqref{eq:IMF_decomp} and the SM correction $\eta_d(\rho,u)$ of Proposition~\ref{prop:det_bound} read
\begin{align}
    \Phi_d(\rho,u) &\;=\;\sum_{j=0}^{\lfloor (d-2)/2\rfloor}\!2\rho\,(2\Lambda)^j\,\frac{(d-2)!!}{(d-2j-2)!!}\,(2\rho u)^{d-2j-2}
    \;+\;\frac{\1_{\{d\text{ odd}\}}}{u}\,(2\Lambda)^{\frac{d-1}2}\,(d-2)!!\,,\nonumber\\
    \Psi_d(\rho,u) &\;=\;\frac{c_0^2}{(1+\rho^2)\,u}\;+\;\sum_{j=1}^{d-1}c_j^2\,\frac{(2\rho)^{2j}\,u^{2j-1}}{1+\rho^2-(2j-1)/u^2}\,,\nonumber\\
    \eta_d(\rho,u)&\;=\;\Bigl(1+\bigl(\tfrac{\sqrt{d-1}}{\rho u}\bigr)^{1/2}\Bigr)^{d-1}\!\!-1+(d+1)\,2^{d-1}\,e^{-2\rho u\sqrt{d-1}/9}\,.\label{eq:eta_def_intro}
\end{align}
\end{itemize}
\end{subequations}

\begin{figure}[!t]
\centering
\includegraphics[width=0.99\linewidth]{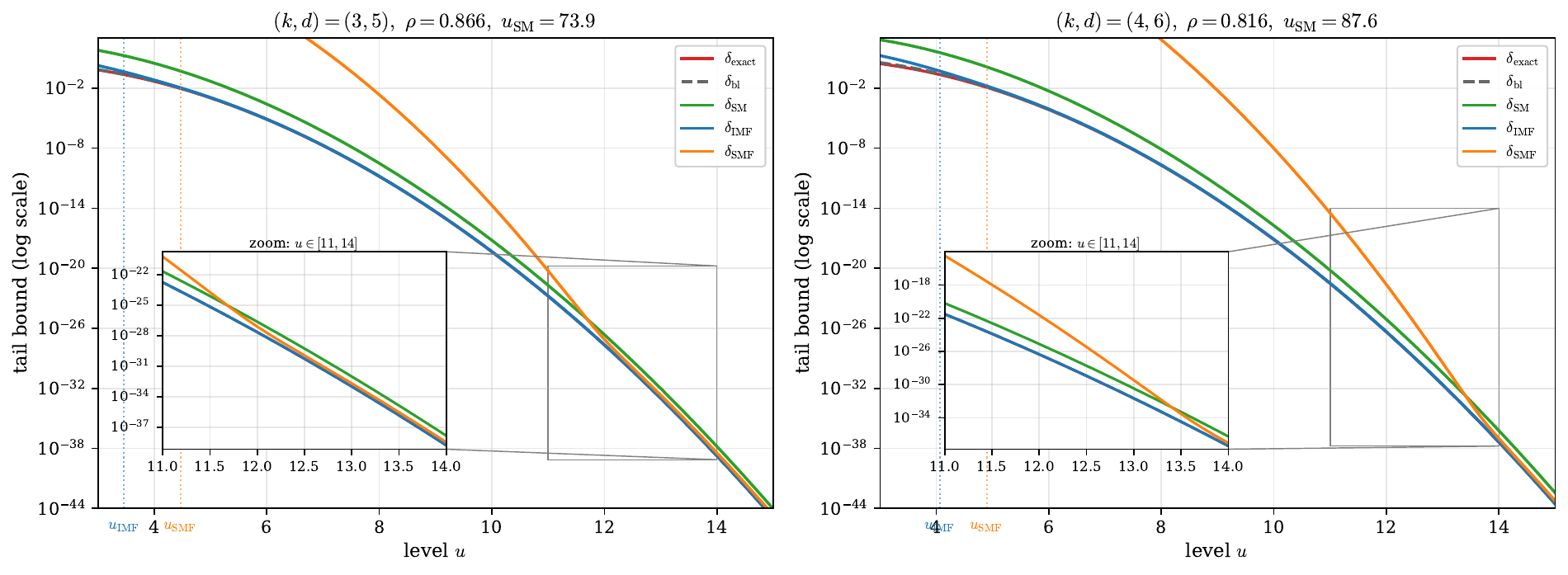}
\caption{Comparison of the tail bounds $\delta_{\mathrm{exact}}$, $\delta_{\mathrm{IMF}}$, $\delta_{\mathrm{SMF}}$, $\delta_{\mathrm{SM}}$ against the asymptotic baseline $\delta_{\mathrm{bl}}$, on logarithmic scale, plotted against the level~$u$ for $(k,d)=(3,5)$ (left, $u_{\mathrm{SM}}=73.90$) and $(k,d)=(4,6)$ (right, $u_{\mathrm{SM}}=87.64$), both well outside the displayed range. The five curves are: (i)~the strictly exact closed-form bound $\delta_{\mathrm{exact}}(u)$ of Theorem~\ref{thm:imf_tail_exact} (rigorous, sharpest closed form, and the reference for $\delta_0$ since it is the closed-form evaluation of the Kac--Rice integral~\eqref{eq:KR_combined_intro} on $[u_{\mathrm{IMF}},\infty)$); (ii)~the asymptotically sharp $\delta_{\mathrm{IMF}}(u)$; (iii)~the inversion-friendly $\delta_{\mathrm{SMF}}(u)$ (factor $2$ above baseline); (iv)~the independent spectral-radius bound $\delta_{\mathrm{SM}}(u)$; and (v)~the asymptotic baseline $\delta_{\mathrm{bl}}(u)$ of~\eqref{eq:asymptotic_baseline_approx} (not a bound). All curves start from $u=u_{\mathrm{IMF}}$; the $\delta_{\mathrm{SM}}$ expression is evaluated below its own validity threshold $u_{\mathrm{SM}}$ to make the comparison with $\delta_{\mathrm{SMF}}$ and $\delta_{\mathrm{IMF}}$ visible (it ceases to be a bound below $u_{\mathrm{SM}}$). An inset (top right) zooms onto the moderate-$u$ region where the gap between $\delta_{\mathrm{IMF}}$, $\delta_{\mathrm{exact}}$ and $\delta_{\mathrm{bl}}$ is largest. $\delta_{\mathrm{IMF}}$ is visually indistinguishable from $\delta_{\mathrm{exact}}$, consistent with the asymptotic sharpness of the Hermite-tail analysis; $\delta_{\mathrm{SM}}$ (resp.\ $\delta_{\mathrm{SMF}}$) is looser through the explicit correction factor $1+\eta_d(\rho,u)$ (resp.\ the Szeg\H{o} envelope coarsening).}
\label{fig:delta_min}
\end{figure}

\paragraph{Asymptotic optimality of the prefactors.}
As $d\to\infty$ ($k$ fixed), a Stirling approximation shows all three relaxations capture the exponential scale $(ek/d)^{d/2}$ of the Kac--Rice baseline; the IMF bound recovers the exact leading constant, while SM (the absolute-value split at the edge) and SMF (the Szeg\H{o} envelope) each incur a uniform factor $2$. The derivation is given in Appendix~\ref{app:prefactor_asymptotics}.

\paragraph{Master tail bound and main theorem.}
By direct application of Theorem~\ref{thm:imf_tail} (the IMF tail bound, Section~\ref{sub:IMF}) and the two-sided symmetry $\P\{\Gamma_{1,1}>u\}\le 2\,\P\{\sup_{\bm\theta}X(\bm\theta)>u\}$, the master tail bound on $\Gamma_{1,1}=\sup_{\bm \theta\in\mathbb S^{d-1}}|X(\bm \theta)|$ used in Theorem~\ref{thm:main} is
\begin{subequations}
\begin{equation}\label{eq:delta_min}
    \forall u\ge u_{\mathrm{IMF}}\,,\qquad
    \delta_{\min}(u)\;:=\;2\,\delta_{\mathrm{IMF}}(u)\,,
    \qquad
    \P\big\{\Gamma_{1,1}>u\big\}\;\le\;\delta_{\min}(u)\,.
\end{equation}
The choice~\eqref{eq:delta_min} is unconditionally smaller than the SMF branch $2\delta_{\mathrm{SMF}}$ on $[u_{\mathrm{SMF}},\infty)$ (Theorem~\ref{thm:uniform_domination}, Section~\ref{sub:uniform_domination}), and smaller than the SM branch $2\delta_{\mathrm{SM}}$ on $[u_{\mathrm{SM}},\infty)$ for every $(k,d)$ of practical interest (Remark~\ref{rem:sm_domination}). The SM and SMF bounds are retained throughout the paper as pedagogical and practical alternatives: $\delta_{\mathrm{SMF}}$ enables the closed-form inversion of Remark~\ref{rem:choice_u}, $\delta_{\mathrm{SM}}$ is independent of the Mehta--Fyodorov algebra and serves as a consistency check.

\begin{theorem}[Non-asymptotic estimation error bound]
\label{thm:main}
Let $k\ge 3$, $d\ge 3$, $R\ge 1$, let $\meanTensor\in\mathfrak{S}_R$ be a signal tensor with $\kappa(\meanTensor)\ge\kappa$, and let $\hat{\putativemeanTensor}$ be the corresponding profile MLE defined in~\eqref{eq:MLE_definition}. For every $u\ge u_{\mathrm{IMF}}=\sqrt{2d-1}/\rho$, with probability at least $1-\delta_{\min}(u)$,
\begin{equation}\label{eq:main_bound}
    \|\hat{\putativemeanTensor} - \meanTensor\|_F^2
    \;\le\;
    \frac{4\sqrt{R}\,u}{\kappa\,\lambda}\,.
\end{equation}
\end{theorem}
\end{subequations}
\noindent
The proof of Theorem~\ref{thm:main}, given in Section~\ref{sub:master_proof}, combines the deterministic Tube Method (Lemma~\ref{lem:geometric_bound}), the rank-reduction inequality (Lemma~\ref{lem:reduce_rank_one}), the symmetry reduction~\eqref{eq:two_sided}, and the IMF tail bound (Theorem~\ref{thm:imf_tail}). The choice $\delta_{\min}=2\,\delta_{\mathrm{IMF}}$ in~\eqref{eq:delta_min} is unconditionally tighter than $2\,\delta_{\mathrm{SMF}}$ on $[u_{\mathrm{SMF}},\infty)$ by Theorem~\ref{thm:uniform_domination}; the corresponding comparison with $2\,\delta_{\mathrm{SM}}$ on $[u_{\mathrm{SM}},\infty)$ holds for every $(k,d)$ of practical interest (Remark~\ref{rem:sm_domination}). The strictly tighter substitute $2\,\delta_{\mathrm{exact}}(u)\le\delta_{\min}(u)$ is available for numerical evaluation (Theorem~\ref{thm:imf_tail_exact}).

\paragraph{Tighter, weaker, and inverted variants}
The failure probability $\delta_{\min}(u)=2\,\delta_{\mathrm{IMF}}(u)$ in Theorem~\ref{thm:main} admits a strictly tighter substitute $2\,\delta_{\mathrm{exact}}(u)$ (Theorem~\ref{thm:imf_tail_exact}, Section~\ref{sub:exact_mf}) and a more analytically tractable, looser substitute $2\,\delta_{\mathrm{SMF}}(u)$ (Theorem~\ref{thm:smf_tail}, Section~\ref{sub:SMF}). The choice between them is purely a matter of computational convenience: $\delta_{\mathrm{exact}}$ is preferred when the bound is to be evaluated numerically; $\delta_{\mathrm{SMF}}$ when an explicit closed-form threshold $u_\alpha$ is required (Remark~\ref{rem:choice_u} and Appendix~\ref{app:u_alpha}); $\delta_{\mathrm{IMF}}$ is the natural default. The independent $\delta_{\mathrm{SM}}$ (Theorem~\ref{thm:sm_tail}, Section~\ref{sub:SM}) is dominated by $\delta_{\mathrm{IMF}}$ on $[u_{\mathrm{SM}},\infty)$ for every $(k,d)$ of practical interest (Remark~\ref{rem:sm_domination}) and is retained for its conceptual independence from the Mehta--Fyodorov algebra and as an unconditional check; in all cases the conservative bound $\min(2\delta_{\mathrm{IMF}},2\delta_{\mathrm{SM}})$ is unconditional on $[u_{\mathrm{SM}},\infty)$.

\begin{remark}[Choice of the level $u$]
\label{rem:choice_u}
For a target confidence level $\alpha\in(0,1)$, set
\begin{equation}\label{eq:u_alpha}
    u \;=\; \sqrt{\,2\log(1/\alpha)\,+\,n\log(2k)\,+\,2n\log\log(1/\alpha)\,}\,,\qquad n=d-1.
\end{equation}
There exists an explicit threshold $\alpha_0(k,d)\in(0,1)$ such that the choice~\eqref{eq:u_alpha} satisfies $\delta_{\min}(u)\le\alpha$ for every $\alpha\in(0,\alpha_0(k,d)]$. The dependence on $(k,d)$ enters only through the constants $\alpha_d$, $\beta_d$ of Lemma~\ref{lem:dominant_term} and $u^\star_d$ (the level beyond which the SMF remainder is negligible; Corollary~\ref{cor:single_term}), together with the absolute constant of the Gaussian tail bound; no further hidden quantities are involved. This threshold is exponentially small in the dimension: the closed form~\eqref{eq:u_alpha} is calibrated against the SMF main term, itself a bound only on $[u_{\mathrm{SMF}},\infty)$ with $u_{\mathrm{SMF}}=2\sqrt d$, so $\alpha_0(k,d)=O(e^{-2d})$ up to a polynomial-in-$d$ prefactor and the SMF bound is vacuous ($>1$) over the moderate-$\alpha$ range. For moderate confidence levels one therefore inverts the non-vacuous master bound $\delta_{\min}=2\delta_{\mathrm{IMF}}$ of Theorem~\ref{thm:main} numerically; this affects only the availability of the closed form~\eqref{eq:u_alpha}, not the validity of Theorem~\ref{thm:main}. Substituting~\eqref{eq:u_alpha} into~\eqref{eq:main_bound} yields the non-asymptotic rate
\[
    \|\hat{\putativemeanTensor} - \meanTensor\|_F^2 \;\lesssim\; \frac{\sqrt{R}}{\kappa\,\lambda}\,\sqrt{d\log k - \log\alpha}\,,
\]
which recovers, with explicit constants and an explicit dependence on $(R,\kappa)$, the asymptotic optimal rate of \citet{perry2020statistical} established for $d\to\infty$ in the rank-one case. The full proof, given in Appendix~\ref{app:u_alpha}, performs the inversion term-by-term against the SMF main contribution and absorbs the polynomial $\log u$ correction arising from $u^{n-1}$ into the $2n\log\log(1/\alpha)$ slack of~\eqref{eq:u_alpha}.
\end{remark}
\pagebreak[3]

\subsection{Related work and contributions}
\label{sub:related}

Tensor methods are widely used for learning latent variable models: community detection in stochastic block models, parameter estimation in mixtures of Gaussians, latent Dirichlet allocation, and Independent Component Analysis can all be recast as the low-rank decomposition of an empirical moment tensor of order $k\ge 3$, which resolves the identifiability obstructions that, for the matrix case $k=2$, prevent unique recovery of the low-rank components from the moment tensor alone \citep{anandkumar2014tensor}. Spiked Tensor PCA, introduced by \citet{montanari2014statistical}, isolates the low-rank-plus-noise structure of the estimation problem, abstracting away problem-specific details.
\citet[Theorem~1.3]{perry2020statistical} establish that detecting a rank-one signal below an explicit eigenvalue threshold is information-theoretically impossible as $d\to\infty$, and identify the optimal estimation rate $O(\sqrt{d\log k})$ in our normalization (which differs from theirs by a factor of $\sqrt{d/2}$, the noise-normalisation conversion detailed in \citealp{azais2024second}). Theorem~\ref{thm:main} recovers this rate non-asymptotically, with explicit constants and an explicit dependence on the rank~$R$ and the coherence~$\kappa$ that does not appear in the existing asymptotic literature (see Remark~\ref{rem:choice_u}). The matrix-analogue spectral signature of the detection threshold is the Baik--Ben Arous--P\'ech\'e (BBP) phase transition \citep{baik2005phase}, describing how the extreme eigenvalues of a deformed random matrix detach from the bulk only once the signal-to-noise ratio crosses a precise threshold; this transition is closely connected to the spectral edge of the Gaussian Orthogonal Ensemble whose conditional law is used throughout our Kac--Rice analysis.

\medskip

The geometric landscape of the likelihood function for tensor PCA is a smooth random function on the high-dimensional sphere, and its critical-point structure has been studied in the spin-glass literature. \citet{auffinger2013} establish that the number of critical points of the same Kostlan--Shub--Smale field vanishes above the spectral edge and relate this complexity to GOE matrix theory; we exploit in particular their conditional Hessian law, \citet[Lemma~3.2(b)]{auffinger2013}, as a main input to the Kac--Rice computation. The present paper differs from this body of work in focus: rather than counting critical points, we control the tail of the global maximum, and our bounds can be read as a non-asymptotic refinement of the large-deviation complexity theory in the region above the spectral edge. The connection to spin glasses extends further: aging phenomena and energy-landscape complexity in the $p$-spin spherical glass model studied by \citet{ben-guionnet}, \citet{fyodorov2004complexity} and \citet{arous2019landscape} are direct probabilistic analogues of the optimization problem we analyze, and the topology of local maxima governs the detectability threshold in both settings. On the algorithmic side, the high-dimensional dynamics of efficient solvers for the multi-spiked tensor model (Langevin dynamics and online stochastic gradient descent) were recently characterised by \citet{benarous2024langevin,benarous2024sgd}, who determine the sample complexity and signal-to-noise separation required for recovery; these address the \emph{computational} side of the statistical-to-computational gap, complementary to the \emph{statistical} profile-MLE guarantee of Theorem~\ref{thm:main}.

\medskip

The canonical Kostlan--Shub--Smale random field, originally introduced by \citet{shub1993complexity} and \citet{kostlan1993distribution} to study the distribution of roots of random polynomial systems, is at the algebraic core of our analysis. Its covariance structure $\langle\bm \theta,\bm v\rangle^k$ matches that of our tensor noise exactly, and underlies the structured conditioning of field value and Hessian. Our methodology is the Kac--Rice framework for level sets and extrema of random fields, treated in detail by \citet{azais2009level}. The present paper is a companion to \citet{azais2024second}, where the same conditional Kac--Rice analysis of the Kostlan--Shub--Smale field underlies an exact \emph{spacing test} for detecting sparse alternatives in Gaussian symmetric tensors, including an exact, non-asymptotic test for the spiked tensor model that needs no prior knowledge of the noise level, thus providing the \emph{detection} counterpart to the estimation and complexity results obtained here. Sharpening the resulting integrals depends on orthogonal-polynomial expansions of the GOE eigenvalue density formalised in \citet{mehta} and on Szeg\H{o}'s bounds on Hermite polynomial roots \citep{szego}.

\medskip

\noindent
{\bf Contributions.} \emph{(i)} A geometric framework for rank-$R$ tensor estimation via the profile MLE on $\mathfrak{S}_R$, with a deterministic Tube Method bound and a rank-reduction inequality controlling the rank-$R$ noise level by the rank-one Kostlan--Shub--Smale supremum (Lemmas~\ref{lem:geometric_bound},~\ref{lem:reduce_rank_one}). \emph{(ii)} A four-tier hierarchy of explicit non-asymptotic tail bounds on that supremum: $\delta_{\mathrm{exact}}$ (Theorem~\ref{thm:imf_tail_exact}), the asymptotically sharp $\delta_{\mathrm{IMF}}$ (Theorem~\ref{thm:imf_tail}), the inversion-friendly $\delta_{\mathrm{SMF}}$ (Theorem~\ref{thm:smf_tail}), and the independent $\delta_{\mathrm{SM}}$ (Theorem~\ref{thm:sm_tail}), with the domination established in Theorem~\ref{thm:uniform_domination}. \emph{(iii)} A unified main theorem (Theorem~\ref{thm:main}) with explicit failure probability $\delta_{\min}=2\delta_{\mathrm{IMF}}$, recovering the Perry--Wein--Bandeira rate $\sqrt{d\log k}$ with explicit $(R,\kappa)$-dependence. \emph{(iv)} A non-asymptotic two-sided bracketing of the annealed complexity of the spherical $k$-spin Hamiltonian (Theorem~\ref{thm:complexity}): writing $N^{\mathrm{lm}}_{[E,\infty)}$ for the number of local maxima of $X$ with $X(\bm\theta)\ge E$, for every finite $(k,d)$ and every $E\ge E_{\mathrm{BDG}}=8\sqrt{2(d-1)}/\rho$,
\[
\big(1-C_{\mathrm{amp}}^{d-1}\delta_{\mathrm{BDG}}^{1/2}(\rho E)\big)\,\delta_{\mathrm{exact}}(E)
\;\le\;\E\!\big[N^{\mathrm{lm}}_{[E,\infty)}\big]
\;\le\;\delta_{\mathrm{exact}}(E)\,,
\]
where $\delta_{\mathrm{BDG}}(\rho E)=e^{-2(\rho E)^2/9}$, $C_{\mathrm{amp}}=8\sqrt 2$, and the factor $C_{\mathrm{amp}}^{d-1}\delta_{\mathrm{BDG}}^{1/2}$ vanishes super-exponentially in $d$. This recovers the Auffinger--Ben Arous--\v{C}ern\'y complexity function~\citep[Theorem~2.4]{auffinger2013} as $d\to\infty$ (Corollary~\ref{cor:abc_match}, with the dual very-deep-minima statement in Corollary~\ref{cor:minima_dual}; at reduced energy $e=\rho E/\sqrt{2(d-1)}\ge 1$, the finite-$d$ bracket holding for $e\ge 8$), and connects to the spiked landscape of \citet{arous2019landscape}.

\subsection{Notation}
\label{sub:notation}

A comprehensive table of all mathematical notations, grouped by topic, is provided in Appendix~\ref{app:notations}; we record here only the conventions used most frequently in the body. We denote by $\Tensors$ the space of symmetric tensors of order $k$ and dimension $d$, equipped with the inner product $\langle\cdot,\cdot\rangle_\tensors$ and the Frobenius norm $\|\cdot\|_F$; $\sphereTensors$ is the unit sphere in~$\Tensors$. The unit sphere of $\R^d$ is $\mathbb{S}^{d-1}$, with surface area $|\mathbb{S}^{d-1}|=2\pi^{d/2}/\Gamma(d/2)$.  The subset $\mathfrak{S}_R\subset\sphereTensors$ collects tensors of rank at most~$R$. The standard Gaussian tail is $\bar\Phi(u):=\P(Z>u)$ with $Z\sim\mathcal{N}(0,1)$, and $\varphi(z):=(2\pi)^{-1/2}e^{-z^2/2}$. The integer $n:=d-1$ denotes the size of the GOE matrices arising in the Kac--Rice analysis (dimension of the tangent space to~$\mathbb{S}^{d-1}$); we use $n$ exclusively for the GOE matrix size and $d$ for the ambient tensor dimension.

\section{Expected GOE characteristic polynomial and Kac--Rice tail bounds}
\label{sec:tail_bounds}
\begin{subequations}
This section assembles a hierarchy of four explicit non-asymptotic bounds on the Kac--Rice integral~\eqref{eq:KR_combined_intro}. Throughout, the object of interest is the supremum probability $\P\{\sup_{\bm \theta\in\mathbb S^{d-1}} X(\bm \theta)>u\}$; the passage to $\P\{\Gamma_{1,1}>u\}$ relevant for Theorem~\ref{thm:main} is the symmetry reduction~\eqref{eq:two_sided} of Section~\ref{sub:master_proof}, applied once at the end of the argument.
The Kac--Rice formula \citep[Theorem~6.4]{azais2009level} bounds this supremum probability by an integral against the intensity of \emph{critical points} of $X$ at level $x$, which by isotropy reduces to the Gaussian integral against the expected absolute determinant of the conditional Hessian, namely
\begin{align*}
        \P\Big\{\sup_{\bm \theta\in\mathbb S^{d-1}}X(\bm \theta)>u\Big\}
        &\le\int_u^\infty \bar p(x)\,\mathrm{d}x\,,
        \qquad
        \bar{p}(x) := |\mathbb{S}^{d-1}| \, p_{\nabla X}(0) \, \E\!\big[ |\det \nabla^2 X(\bm{\theta})| \,\big|\, X(\bm{\theta})=x \big] \varphi(x)\,,
        \end{align*}
where $\bar p(x)$ is the intensity of critical points of~$X$ at level~$x$; $p_{\nabla X}(0)=(2\pi k)^{-\frac{d-1}2}$; and for any fixed $\bm \theta\in\mathbb S^{d-1}$,  $\nabla^2 X(\bm \theta)$ denotes the Riemannian Hessian at~$\bm \theta$. The bound 
\[
    \P\Big\{\sup_{\bm \theta\in\mathbb S^{d-1}}X(\bm \theta)>u\Big\}\le\E\Big[\#\big\{\bm\theta\in\mathbb S^{d-1}:\nabla X(\bm\theta)=0,\,X(\bm\theta)>u\big\}\Big]=\int_u^\infty\bar p(x)\,\mathrm{d}x\,,
\] 
holds because, on the event $\{\sup_{\bm\theta}X(\bm\theta)>u\}$, the smooth field $X$ attains its supremum at an interior critical point of the compact manifold $\mathbb S^{d-1}$, so the indicator of this event is dominated by the count of critical points above $u$. The bound is loose only by the (positive) contribution of saddle points above $u$; this slack is asymptotically negligible above the spectral edge and is the cost of a closed-form Kac--Rice integrand.
The conditioning on the level $X(\bm\theta)=x$ alone (rather than jointly on $\{\nabla X(\bm\theta)=0,\,X(\bm\theta)=x\}$, as the Kac--Rice intensity formally requires) is justified by isotropy: at any fixed~$\bm\theta$ the cross-covariances $\E[X(\bm\theta)\,\partial_i X(\bm\theta)]$ and $\E[\partial_i X(\bm\theta)\,\nabla^2 X(\bm\theta)]$ vanish, so the Gaussian gradient $\nabla X(\bm\theta)$ is independent of the pair $(X(\bm\theta),\nabla^2 X(\bm\theta))$ and the event $\{\nabla X(\bm\theta)=0\}$ may be dropped from the conditioning of the determinant.
The conditional Hessian is itself, by \citet[Lemma~3.2(b)]{auffinger2013}, a shifted GOE matrix:
\begin{equation}\label{eq:hessian_law}
    \E\!\big[ |\det \nabla^2 X(\bm{\theta})| \,\big|\, X(\bm{\theta})=x \big]
    \;=\;(2k(k-1))^{\frac{d-1}2}\, \E\!\big[ |\det (G_{d-1} - \rho x \, I_{d-1})| \big]\,,
\end{equation}
and combining these inputs yields the central reduction
\begin{equation}\label{eq:KR_combined}
    \P\Big\{\sup_{\bm \theta\in\mathbb S^{d-1}}X(\bm \theta)>u\Big\}
    \;\le\;
    C_{k,d}\int_u^\infty \E\big[|\det(G_{d-1}-\rho x\,I_{d-1})|\big]\,\varphi(x)\,\mathrm{d}x\,,
\end{equation}
where $\rho$ and $C_{k,d}$ are given in \eqref{eq:rho_ckd_intro}. We develop, in turn, the four explicit closed-form bounds on the right-hand side.
\end{subequations}

The remainder of this section derives the bounds, tightest first, with the sharpened variant $\delta_{\mathrm{IMF}}^\star$ interpolating between $\delta_{\mathrm{exact}}$ and $\delta_{\mathrm{IMF}}$: they are nested as $\delta_{\mathrm{exact}}\le\delta_{\mathrm{IMF}}^\star\le\delta_{\mathrm{IMF}}\le\delta_{\mathrm{SMF}}$ on $[u_{\mathrm{IMF}},\infty)$ (Theorems~\ref{thm:imf_tail_exact},~\ref{thm:imf_tail_sharp},~\ref{thm:uniform_domination}), with the independent $\delta_{\mathrm{SM}}$ dominated on $[u_{\mathrm{SM}},\infty)$ (Remark~\ref{rem:sm_domination}); $\delta_{\mathrm{IMF}}$ is the bound used in Theorem~\ref{thm:main}.

\subsection{Mehta--Fyodorov representation (MF)}
\label{sub:MF}

The Mehta--Fyodorov representation uses two identities: Fyodorov's formula (Lemma~\ref{lfyod}), which writes the expected absolute characteristic polynomial of a $(d-1)$-GOE at level $\nu$ through the one-point eigenvalue density of a one-size-larger $d$-GOE, and Mehta's Hermite expansion of that density (Lemma~\ref{lem:mehta}); together they give an exact closed algebraic form of the Kac--Rice integrand.

\begin{lemma}[Fyodorov's formula~\citep{fyodorov2004complexity}]
\label{lfyod}
For every $\nu\in\R$,
\begin{equation}\label{eq:fyo}
    \E\big[|\det(G_{d-1}-\nu I_{d-1})|\big]
    \;=\;\frac{2^{\frac{3}2}\,\Gamma((d+2)/2)}{d}\;Q_{d}(\nu)\,,
\end{equation}
where $Q_d(\nu):=e^{\nu^2/2}q_d(\nu)$ and $q_d$ is the eigenvalue intensity (one-point correlation function) of a $d\times d$ GOE matrix in the Mehta normalization, so that $\int_\R q_d(\nu)\,\mathrm{d}\nu = d$, the expected number of eigenvalues. The identity expresses the expected absolute characteristic polynomial of the $(d-1)\times(d-1)$ GOE through the one-point function of the \emph{one-size-larger} $d\times d$ GOE; see \citet{fyodorov2004complexity}.
\end{lemma}

\begin{lemma}[Mehta expansion~{\citep[p.~221]{mehta}}]
\label{lem:mehta}
The GOE one-point function admits the following Hermite-function expansion. With $c_j:=(2^j j!\sqrt{\pi})^{-1/2}$ and $H_j$ the $j$-th physicist Hermite polynomial,
\begin{align}
    Q_d(\nu) &\;=\;
    \frac{1}{2}\sqrt{\frac d2}\,c_{d-1}c_d\,H_{d-1}(\nu)\big[\mu_d-2\mathcal{I}_d^c(\nu)\big]
    \;+\;\1_{\{d\text{ odd}\}}\,\frac{H_{d-1}(\nu)}{\mu_{d-1}}
    +e^{-\nu^2/2}\sum_{j=0}^{d-1} c_j^2\,H_j^2(\nu)\,, \label{eq:fmehta}
\end{align}
where
\[
    \mu_m:=\int_\R H_m(y)\,e^{-y^2/2}\,\mathrm{d}y\,,\qquad
    \mathcal{I}_d^c(\nu):=\int_\nu^\infty H_d(y)\,e^{-y^2/2}\,\mathrm{d}y\,.
\]
By parity, $\mu_m=0$ whenever $m$ is odd; in particular, for $d$ odd only the $H_{d-1}/\mu_{d-1}$ term contributes, whereas for $d$ even only the $\mu_d$ bracket does.
\end{lemma}

The combination of Lemmas~\ref{lfyod} and~\ref{lem:mehta} produces an exact integral representation of the Kac--Rice probability bound. Fyodorov's formula~\eqref{eq:fyo} gives the expected absolute characteristic polynomial as a constant multiple of $Q_d(\rho x)$, and Mehta's expansion~\eqref{eq:fmehta} supplies an explicit formula for $Q_d$. Substituting into the Kac--Rice reduction~\eqref{eq:KR_combined} therefore yields an integrand of the form $Q_d(\rho x)\,e^{-x^2/2}$: a polynomial-and-Hermite quantity weighted by the original Kac--Rice Gaussian, with no further exponential factor entering the integrand. Fyodorov's identity is already stated in the natural variable $Q_d$, so the eigenvalue intensity $q_d=e^{-\nu^2/2}Q_d$ never appears in the substitution: the Gaussian envelope that would otherwise be carried by $q_d$ is absorbed into the definition of $Q_d$. This algebraic identity, unavailable to the spectral-splitting route, is why the IMF and SMF bounds of Sections~\ref{sub:IMF} and~\ref{sub:SMF} close the residual factor-$2$ gap of Theorem~\ref{thm:sm_tail}. We state this in the following proposition.

\begin{proposition}[Exact MF representation]
\label{prop:ortho_exact}
For every $u\in\R$,
\begin{equation}\label{eq:ortho_exact}
    \P\Big\{\sup_{\bm \theta\in\mathbb S^{d-1}}X(\bm \theta)>u\Big\}
    \;\le\;
    2(k-1)^{\frac{d-1}2}\int_u^\infty Q_{d}(\rho x)\,e^{-x^2/2}\,\mathrm{d}x\,.
\end{equation}
\end{proposition}

\begin{proof}[Proof of Proposition~\ref{prop:ortho_exact}]

\noindent
$\bullet$
Starting from the Kac--Rice reduction in~\eqref{eq:KR_combined}, Lemma~\ref{lfyod} evaluated at the shifted argument $\nu=\rho x$ gives an exact expression for the expected absolute determinant directly in terms of $Q_d$:
\begin{equation}\label{eq:fyo_substitute_app}
    \E\big[|\det(G_{d-1}-\rho x\,I_{d-1})|\big]
    \;=\;\frac{2^{\frac{3}2}\Gamma((d+2)/2)}{d}\,Q_d(\rho x)\,,
\end{equation}
where $Q_d(\nu)=e^{\nu^2/2}q_d(\nu)$ is the unweighted polynomial part of the eigenvalue intensity (one-point correlation function) of a $d\times d$ GOE matrix in the Mehta normalization. No exponential factor in $\rho x$ is introduced at this step: the Gaussian envelope that would otherwise be carried by $q_d$ is precisely what the definition of $Q_d$ absorbs.

\noindent
$\bullet$ 
We multiply this expression by the standard Gaussian density $\varphi(x)=(2\pi)^{-1/2}e^{-x^2/2}$ and integrate over the interval $[u,\infty)$ to match the Kac--Rice integral:
\[
    \int_u^\infty\E\big[|\det(G_{d-1}-\rho x\,I_{d-1})|\big]\,\varphi(x)\,\mathrm{d}x
    \;=\;\frac{2^{\frac{3}2}\Gamma((d+2)/2)}{d\sqrt{2\pi}}\int_u^\infty Q_d(\rho x)\,e^{-x^2/2}\,\mathrm{d}x\,.
\]
We can simplify the constant prefactor using the Gamma function identity $\Gamma((d+2)/2)=(d/2)\Gamma(d/2)$:
\[
    \frac{2^{\frac{3}2}\Gamma((d+2)/2)}{d\sqrt{2\pi}} 
    \;=\; \frac{2\sqrt{2} \cdot (d/2)\Gamma(d/2)}{d \sqrt{2} \sqrt{\pi}} 
    \;=\; \frac{\Gamma(d/2)}{\sqrt{\pi}}\,.
\]
Finally, substituting this result back into the original Kac--Rice bound~\eqref{eq:KR_combined} alongside the geometric volume prefactor $C_{k,d}=2\sqrt\pi(k-1)^{\frac{d-1}2}/\Gamma(d/2)$, the $\Gamma(d/2)$ and $\sqrt{\pi}$ terms cancel:
\[
    C_{k,d} \frac{\Gamma(d/2)}{\sqrt{\pi}} \;=\; \left( \frac{2\sqrt\pi(k-1)^{\frac{d-1}2}}{\Gamma(d/2)} \right) \frac{\Gamma(d/2)}{\sqrt{\pi}} \;=\; 2(k-1)^{\frac{d-1}2}\,.
\]
This yields the exact representation~\eqref{eq:ortho_exact}.
\end{proof}

\noindent
The identity~\eqref{eq:ortho_exact} is an \emph{exact representation}, not yet a bound: it expresses the Kac--Rice probability as an explicit integral against the Mehta function $Q_d(\rho x)$, which itself comprises a linear Hermite term $H_{d-1}(\rho x)$, a squared Hermite sum $\sum_j c_j^2\,H_j^2(\rho x)$, and a signed tail integral $\mathcal{I}_d^c(\rho x)$. The IMF (Section~\ref{sub:IMF}) and SMF (Section~\ref{sub:SMF}) bounds both take this identity as their starting point; they differ only in how aggressively they coarsen these three components. The IMF bound preserves the full partial Hermite sum structure and is asymptotically sharp; the SMF bound trades a small constant loss for an explicit, closed-form polynomial--exponential envelope.
\subsection{Strictly exact Mehta--Fyodorov tail bound}
\label{sub:exact_mf}

Substituting the full Mehta expansion~\eqref{eq:fmehta} into Proposition~\ref{prop:ortho_exact} produces the four-piece decomposition
\begin{equation}\label{eq:Q_d_decomp}
    \int_u^\infty Q_d(\rho x)\,e^{-x^2/2}\,\mathrm{d}x \;=\; D_1(u)\;+\;D_2(u)\;+\;D_3(u)\;+\;D_4(u)\,,
\end{equation}
where, with $\nu=\rho x$,
\begin{align*}
    D_1(u) &\;:=\;\tfrac12\sqrt{d/2}\,c_{d-1}c_d\,\mu_d\,\mathcal{L}_d(u)\,,\quad &\text{(zero for $d$ odd, since $\mu_d=0$),}\\
    D_2(u) &\;:=\;-\sqrt{d/2}\,c_{d-1}c_d\,\mathcal{C}_d(u)\,,\quad &\text{(cross integral $\mathcal{C}_d$ defined in~\eqref{eq:Cd_def}),}\\
    D_3(u) &\;:=\;\1_{\{d\,\mathrm{odd}\}}\,\mu_{d-1}^{-1}\,\mathcal{L}_d(u)\,,\quad &\text{(zero for $d$ even),}\\
    D_4(u) &\;:=\;T_d^{\mathrm{exact}}(u)\quad &\text{(squared-Hermite tail of \eqref{eq:Td_exact_intro}),}
\end{align*}
with $\mathcal{L}_d(u):=\int_u^\infty H_{d-1}(\rho x)\,e^{-x^2/2}\,\mathrm{d}x$ and the cross integral
\begin{equation}\label{eq:Cd_def}
    \mathcal{C}_d(u)\;:=\;\int_u^\infty H_{d-1}(\rho x)\,\mathcal{I}_d^c(\rho x)\,e^{-x^2/2}\,\mathrm{d}x\,.
\end{equation}

\begin{theorem}[Strictly exact MF tail bound]
\label{thm:imf_tail_exact}
\begin{subequations}
For every $u\in\R$,
\begin{equation}\label{eq:exact_bound}
    \P\Big\{\sup_{\bm \theta\in\mathbb S^{d-1}}X(\bm \theta)>u\Big\}
    \;\le\;
    \delta_{\mathrm{exact}}(u)
    \;:=\;
    2(k-1)^{\frac{d-1}2}\bigl[D_1(u)+D_2(u)+D_3(u)+D_4(u)\bigr]\,,
\end{equation}
where each $D_i$ is given in fully closed form by:
\begin{enumerate}
    \item[$(D_1,D_3)$] the linear Hermite tail $\mathcal{L}_d(u)$ admits the closed form
        \begin{equation}\label{eq:Ld_closed}
            \mathcal{L}_d(u)\;=\;\frac1\rho\,J_{d-1}^{1/\rho^2}(\rho u)\,,
        \end{equation}
        with $J_m^\beta$ given by Lemma~\ref{lem:Jm_beta};
    \item[$(D_4)$] the squared Hermite tail $T_d^{\mathrm{exact}}(u)=\rho^{-1}\sum_{j=0}^{d-1}c_j^2\,K_j^\beta(\rho u)$ of~\eqref{eq:Td_exact_intro}, with $\beta=(3k-2)/k$ and $K_j^\beta$ given by Corollary~\ref{cor:Kj_beta};
    \item[$(D_2)$] the cross integral $\mathcal{C}_d(u)$ admits the two-piece closed form
        \begin{align}
            \mathcal{C}_d(u) &\;=\;\frac1\rho\sum_{\ell=0}^{\lfloor (d-1)/2\rfloor}\frac{2^{\ell+1}\,(d-1)!!}{(d-2\ell-1)!!}\sum_{p=0}^{d-2\ell-1}2^p\,p!\,\binom{d-1}{p}\binom{d-2\ell-1}{p}\,J_{2d-2\ell-2-2p}^\beta(\rho u)\notag\\
            &\quad\;+\;\1_{\{d\,\mathrm{even}\}}\,2^{d/2}\,(d-1)!!\,\mathcal{F}_d(u)\,,\label{eq:Cd_closed}
        \end{align}
    where the Fubini remainder $\mathcal{F}_d(u)$ (only present for $d$ even) reads
        \begin{equation}\label{eq:Fubini_F}
            \mathcal{F}_d(u)\;=\;\sqrt{2\pi}\,\bar\Phi(\rho u)\,\frac{J_{d-1}^{1/\rho^2}(\rho u)}{\rho}
            \;-\;2\rho\sum_{\ell=0}^{(d-2)/2}(2\Lambda)^\ell\,\frac{(d-2)!!}{(d-2\ell-2)!!}\,J_{d-2\ell-2}^\beta(\rho u)\,.
        \end{equation}
\end{enumerate}
Moreover,
\[
    \delta_{\mathrm{exact}}(u)\;\le\;\delta_{\mathrm{IMF}}^\star(u)\;\le\;\delta_{\mathrm{IMF}}(u)
    \qquad\text{for every $u\ge u_{\mathrm{IMF}}$,}
\]
with strict inequality at every finite $u$.
\end{subequations}
\end{theorem}

\noindent
$\delta_{\mathrm{exact}}$ has two roles. The quantity $2(k-1)^{(d-1)/2}\sum_{i=1}^4 D_i(u)$ is the closed-form \emph{evaluation} of the Kac--Rice integral $2(k-1)^{(d-1)/2}\int_u^\infty Q_d(\rho x)e^{-x^2/2}\,\mathrm{d}x$ of Proposition~\ref{prop:ortho_exact}; by the same Kac--Rice identity it equals, \emph{exactly}, the expected number of critical points of $X$ above level $u$ (cf.\ Remark~\ref{rem:kr_representation} and Section~\ref{sec:refinements}). The inequality~\eqref{eq:exact_bound} for $\P\{\sup X>u\}$ is then the first-moment (Markov) step
\[
    \P\{\sup X>u\}\;\le\;\E\big[\#\{\text{critical points above }u\}\big]\,.
\]
Thus $\delta_{\mathrm{exact}}$ is an equality for the critical-point count and an upper bound for the supremum probability; both readings are used in Sections~\ref{sec:statistical} and~\ref{sec:refinements}.

\begin{proof}
Deferred to Appendix~\ref{app:deferred}.
\end{proof}
\noindent
The identity $\delta_{\mathrm{exact}}(u)=2(k-1)^{(d-1)/2}\sum_{i=1}^4 D_i(u)$ matches high-precision quadrature of the integral~\eqref{eq:ortho_exact} to relative error $10^{-13}$--$10^{-8}$ on the grid $(k,d)\in\{(3,3),(3,5),(4,6),(5,7),(3,10)\}$; the full symbolic and numerical certification is in the companion repository \verifrepo.

\subsection{Improved Mehta--Fyodorov bound (IMF)}
\label{sub:IMF}

Concretely, inspecting the Mehta expansion~\eqref{eq:fmehta} one sees that the coefficient of the Hermite polynomial $H_{d-1}(\nu)$ is the bracket $[\mu_d-2\mathcal{I}_d^c(\nu)]$, where $\mathcal{I}_d^c(\nu):=\int_\nu^\infty H_d(y)e^{-y^2/2}\,\mathrm{d}y$ is a Hermite tail integral. If a regime can be identified in which $\mathcal{I}_d^c(\nu)$ is \emph{strictly positive}, the subtractive contribution $-2\mathcal{I}_d^c(\nu)H_{d-1}(\nu)$ to $Q_d$ is strictly negative (since $H_{d-1}(\nu)>0$ in the same regime), and discarding it yields a strict upper bound on the Kac--Rice integral.

Pinpointing this regime reduces to locating the largest root of $H_{d-1}$: beyond it, $H_{d-1}(\nu)$ is positive, and by a root-interlacing argument $\mathcal{I}_d^c(\nu)$ is too. Hence the threshold is set by the largest root of $H_{d-1}$. Szeg\H{o}'s classical estimate~\citep[Theorem~6.32]{szego} locates this largest root strictly below $\sqrt{2d-1}$, so that the relevant regime is $\nu\ge\sqrt{2d-1}$, which translates into the threshold $\rho u\ge\sqrt{2d-1}$ in the Kac--Rice integral. This threshold is smaller than the Spectral Method threshold $32\sqrt{d-1}/\rho$ and incurs no corrective factor $1+\eta_d(\rho,u)$.

The remainder of this subsection assembles the three ingredients required to turn the discarding argument into an explicit bound: a positivity-and-envelope statement for the Hermite tail integral $\mathcal{I}_d^c$ (Lemma~\ref{lem:Idc_positivity}); a recurrence-based evaluation of the resulting linear and squared Hermite tail integrals; and the explicit polynomial-rational functions $\Phi_d(\rho,u)$ and $\Psi_d(\rho,u)$ that encode, respectively, the linear Hermite tail and the squared Hermite tail.

\begin{lemma}[Positivity and upper bound of the Hermite tail integral]
\label{lem:Idc_positivity}
Let $d\ge 3$ and let $x_1^{(m)}$ denote the largest root of the (physicist) Hermite polynomial $H_m$. For every $\nu\ge x_1^{(d-1)}$, the tail integral $\mathcal{I}_d^c(\nu)$ is strictly positive. Moreover there exists a dimension-dependent constant $C_d>0$ such that, for all such $\nu$,
\[
    0 \;<\; \mathcal{I}_d^c(\nu) \;\le\; C_d\,\nu^{d-1}\,e^{-\nu^2/2}\,.
\]
The assumption $d\ge 3$ ensures $x_1^{(d-1)}\ge x_1^{(2)}=1/\sqrt 2>0$, which is used implicitly in the polynomial-envelope step of the proof.
\end{lemma}

\begin{proof}[Proof of Lemma~\ref{lem:Idc_positivity}]

\noindent
$\bullet$
We apply the explicit finite series expansion from Lemma~\ref{lem:hermite_tail_explicit} (specifically Equation~\eqref{eq:hermite_tail_series_1}) with degree $m=d$. This yields:
\[
    \mathcal{I}_d^c(\nu) = e^{-\nu^2/2} \sum_{k=0}^{\lfloor (d-1)/2 \rfloor} 2^{k+1} \frac{(d-1)!!}{(d-2k-1)!!} H_{d-2k-1}(\nu) + R_d(\nu)\,,
\]
where the remainder term guarantees $R_d(\nu) \ge 0$.

\noindent
$\bullet$
The roots of consecutive orthogonal polynomials interlace~\citep[Theorem~3.3.2]{szego}; in particular, the largest root $x_1^{(m)}$ of $H_m$ is strictly increasing in the degree $m$. Therefore, for every $k \ge 1$, we have the strict inequality:
\[
    x_1^{(d-2k-1)} < x_1^{(d-1)}\,.
\]
When we evaluate the sum at any $\nu \ge x_1^{(d-1)}$, the leading polynomial ($k=0$) satisfies $H_{d-1}(\nu) \ge 0$. For all subsequent lower-degree polynomials ($k \ge 1$), since $\nu \ge x_1^{(d-1)} > x_1^{(d-2k-1)}$, they evaluate to strictly positive values: $H_{d-2k-1}(\nu) > 0$. Because all the constant combinatorial coefficients in the sum are strictly positive and $R_d(\nu) \ge 0$, the entire expression evaluates to a strictly positive number: $\mathcal{I}_d^c(\nu) > 0$.

\noindent
$\bullet$
To upper bound the integral for $\nu \ge x_1^{(d-1)}$, we apply the polynomial envelope from Lemma~\ref{lem:hermite_bound}, which bounds each polynomial by $H_{d-2k-1}(\nu) \le (2\nu)^{d-2k-1}$.
The dominant term in the sum corresponds to $k=0$, which is bounded by:
\[
    2 H_{d-1}(\nu) e^{-\nu^2/2} \;\le\; 2(2\nu)^{d-1} e^{-\nu^2/2}\,.
\]
For the lower-order terms ($k\ge 1$) the same envelope gives $H_{d-2k-1}(\nu)\le(2\nu)^{d-2k-1}$, and we promote the asymptotic gain into a uniform pointwise bound exactly as for the remainder term $R_d$: writing $\nu^{d-2k-1}=\nu^{d-1}\,\nu^{-2k}$, the factor $\nu^{-2k}$ is decreasing on $[x_1^{(d-1)},\infty)$ and attains its supremum $\big(x_1^{(d-1)}\big)^{-2k}<\infty$ at the left endpoint (here $x_1^{(d-1)}>0$ is used). Hence each lower-order term is bounded by a constant multiple of the leading envelope $\nu^{d-1}e^{-\nu^2/2}$ on the whole interval $[x_1^{(d-1)},\infty)$.

It remains to control the remainder $R_d(\nu)$, which is non-zero only when $d$ is even. In that case, Lemma~\ref{lem:hermite_tail_explicit} gives the explicit Gaussian-tail expression
\[
    R_d(\nu) \;=\; 2^{d/2}(d-1)!!\int_\nu^\infty e^{-x^2/2}\,\mathrm{d}x\,.
\]
Applying Mills' ratio (equivalently, Lemma~\ref{lem:gauss_tail} with $m=0$, $a=1$, i.e.~\eqref{eq:gauss_tail_mills}) yields $\int_\nu^\infty e^{-x^2/2}\,\mathrm{d}x \le \nu^{-1}e^{-\nu^2/2}$ for $\nu>0$, whence
\[
    R_d(\nu) \;\le\; 2^{d/2}(d-1)!!\,\nu^{-1}\,e^{-\nu^2/2} \;=\; O\!\big(\nu^{-1}\,e^{-\nu^2/2}\big)\,.
\]
To promote this asymptotic decay into a uniform pointwise bound on the full domain: the ratio of the remainder envelope to the leading envelope is $\nu^{-1}/\nu^{d-1} = \nu^{-d}$. Since $x_1^{(d-1)}>0$, this ratio is a positive continuous function on $[x_1^{(d-1)},\infty)$ which attains its supremum at the left endpoint:
\[
    \sup_{\nu\ge x_1^{(d-1)}}\nu^{-d} \;=\; \big(x_1^{(d-1)}\big)^{-d}\;<\;\infty\,.
\]
Consequently $R_d(\nu)\le 2^{d/2}(d-1)!!\,\big(x_1^{(d-1)}\big)^{-d}\,\nu^{d-1}\,e^{-\nu^2/2}$ for every $\nu\ge x_1^{(d-1)}$, a uniform pointwise bound on the whole semi-infinite interval (not only for large $\nu$). Absorbing the leading constant, the lower-order polynomial terms, and the remainder $R_d(\nu)$ into a single dimension-dependent constant $C_d$ yields the final bound:
\[
    \mathcal{I}_d^c(\nu) \;\le\; C_d\,\nu^{d-1}e^{-\nu^2/2}\,,
\]
\end{proof}

\medskip

With the positivity of $\mathcal{I}_d^c$ secured on $\{\nu\ge x_1^{(d-1)}\}$, the subtractive contribution in the Mehta expansion may then be discarded. The two remaining pieces (the linear Hermite tail $\int_u^\infty H_{d-1}(\rho x)e^{-x^2/2}\,\mathrm{d}x$ and the squared Hermite tail $\sum_j c_j^2\int_u^\infty H_j^2(\rho x)e^{-(1+\rho^2)x^2/2}\,\mathrm{d}x$) decay at different Gaussian rates, $e^{-u^2/2}$ and $e^{-(1+\rho^2)u^2/2}$ respectively. This two-scale structure is preserved in the proposition below; the SMF bound of the next subsection collapses it to a single dominant scale. Analytically, we evaluate the linear tail by iterating the three-term Hermite recurrence, and the squared tail by substituting the Szeg\H{o} envelope $H_j(\rho x)\le(2\rho x)^j$ into each squared term and invoking Lemma~\ref{lem:gauss_tail}.

\begin{proposition}[Explicit IMF upper bound on $Q_d$]
\label{prop:IMF_decomp}
\begin{subequations}
For every $u>0$ such that $\rho u\ge\sqrt{2d-1}$,
\begin{equation}\label{eq:IMF_decomp}
    \int_u^\infty Q_d(\rho x)\,e^{-x^2/2}\,\mathrm{d}x
    \;\le\;
    \alpha_d\,\Phi_d(\rho,u)\,e^{-u^2/2}\;+\;\Psi_d(\rho,u)\,e^{-(1+\rho^2)u^2/2}\,,
\end{equation}
where $\alpha_d$ is the dominant Mehta coefficient of Lemma~\ref{lem:dominant_term}, and the explicit polynomial--rational functions are given, with $\Lambda:=2\rho^2-1$, by
\begin{align}
    \Phi_d(\rho,u) &\;:=\;\sum_{k=0}^{\lfloor (d-2)/2\rfloor}\!2\rho\,(2\Lambda)^k\,\frac{(d-2)!!}{(d-2k-2)!!}\,(2\rho u)^{d-2k-2}
    \;+\;\frac{\1_{\{d\text{ odd}\}}}{u}\,(2\Lambda)^{\frac{d-1}2}\,(d-2)!!\,,\label{eq:phi_def}\\
    \Psi_d(\rho,u) &\;:=\;\frac{c_0^2}{(1+\rho^2)\,u}\;+\;\sum_{j=1}^{d-1}c_j^2\,\frac{(2\rho)^{2j}\,u^{2j-1}}{1+\rho^2-(2j-1)/u^2}\,.\label{eq:psi_def}
\end{align}
\end{subequations}
\end{proposition}

\begin{proof}[Proof of Proposition~\ref{prop:IMF_decomp}]

\noindent
$\bullet$
By Szeg\H{o}'s bound, the largest root of $H_{d-1}$ strictly satisfies $x_1^{(d-1)}<\sqrt{2d-1}$. The proposition assumes $\rho u\ge\sqrt{2d-1}$, ensuring that $\rho x > x_1^{(d-1)}$ for all integration variables $x\ge u$. By Lemma~\ref{lem:Idc_positivity}, this guarantees $\mathcal{I}_d^c(\rho x)>0$ over the entire integration domain. 
Returning to the exact Mehta expansion~\eqref{eq:fmehta}, the contribution of $-2\mathcal{I}_d^c(\rho x) H_{d-1}(\rho x)$ is therefore strictly negative. Discarding it yields a strict upper bound:
\begin{equation}\label{eq:IMF_split}
    \int_u^\infty Q_d(\rho x)e^{-x^2/2}\,\mathrm{d}x \;\le\; \alpha_d\,\mathcal{L}_d(u)\;+\;\mathcal{S}_d(u)\,,
\end{equation}
where $\alpha_d$ is the dominant Mehta coefficient given by~\eqref{eq:alpha_d}, and the components are defined as:
\begin{align*}
    \mathcal{L}_d(u) &:= \int_u^\infty H_{d-1}(\rho x)e^{-x^2/2}\,\mathrm{d}x\,, \\
    \mathcal{S}_d(u) &:= \sum_{j=0}^{d-1}c_j^2\int_u^\infty H_j^2(\rho x)e^{-(1+\rho^2)x^2/2}\,\mathrm{d}x\,.
\end{align*}

\noindent
$\bullet$
We derive a generalized recurrence for $J_m := \int_u^\infty H_m(\rho x)e^{-x^2/2}\,\mathrm{d}x$. Applying the identity $H_m(\rho x)=2\rho x H_{m-1}(\rho x)-2(m-1)H_{m-2}(\rho x)$ and evaluating the first term via integration by parts, we obtain:
\begin{equation*}
    J_m = 2\rho\,H_{m-1}(\rho u)\,e^{-u^2/2} + 2\Lambda(m-1)\,J_{m-2}\,, \qquad \text{where } \Lambda := 2\rho^2-1\,,
\end{equation*}
where we used $\frac{\mathrm{d}}{\mathrm{d}x}[H_{m-1}(\rho x)e^{-x^2/2}]=2\rho(m-1)H_{m-2}(\rho x)e^{-x^2/2}-xH_{m-1}(\rho x)e^{-x^2/2}$. Unrolling this recurrence for $m=d-1$ produces a finite series of boundary evaluations. We bound each boundary evaluation using the Szeg\H{o} envelope $H_{d-2k-2}(\rho u) \le (2\rho u)^{d-2k-2}$. If $d-1$ is even (i.e., $d$ is odd), the recurrence bottoms out with a Gaussian tail, which we bound using Mills' ratio $\int_u^\infty e^{-x^2/2}\,\mathrm{d}x \le u^{-1}e^{-u^2/2}$. Factoring out $e^{-u^2/2}$ yields exactly:
\begin{equation*}
    \mathcal{L}_d(u) \;\le\; \Phi_d(\rho,u)\,e^{-u^2/2}\,,
\end{equation*}
where $\Phi_d$ is the polynomial-rational function defined in~\eqref{eq:phi_def}.

\noindent
$\bullet$
For the squared component, we substitute the Szeg\H{o} bound $H_j(\rho x) \le (2\rho x)^j$ directly into the integral:
\begin{equation*}
    \mathcal{S}_d(u) \;\le\; \sum_{j=0}^{d-1} c_j^2 \int_u^\infty (2\rho x)^{2j} e^{-(1+\rho^2)x^2/2}\,\mathrm{d}x\,.
\end{equation*}
We evaluate each integral using Lemma~\ref{lem:gauss_tail} with parameters $a=1+\rho^2$ and $m=2j$. For the base case $j=0$, we use the Mills bound $I_0(u;a) \le (au)^{-1} e^{-au^2/2}$ of~\eqref{eq:gauss_tail_mills}. Factoring out the shared exponential weight $e^{-(1+\rho^2)u^2/2}$ leaves the rational function $\Psi_d(\rho,u)$ defined in~\eqref{eq:psi_def}, yielding:
\begin{equation*}
    \mathcal{S}_d(u) \;\le\; \Psi_d(\rho,u)\,e^{-(1+\rho^2)u^2/2}\,.
\end{equation*}

\noindent
$\bullet$
Substituting the bounds for $\mathcal{L}_d(u)$ and $\mathcal{S}_d(u)$ back into~\eqref{eq:IMF_split} completes the proof:
\begin{equation*}
    \int_u^\infty Q_d(\rho x)e^{-x^2/2}\,\mathrm{d}x \;\le\; \alpha_d\,\Phi_d(\rho,u)\,e^{-u^2/2} \;+\; \Psi_d(\rho,u)\,e^{-(1+\rho^2)u^2/2}\,. \qedhere
\end{equation*}
\end{proof}

\medskip

\begin{theorem}[IMF tail bound]
\label{thm:imf_tail}
Let $u_{\mathrm{IMF}}=\sqrt{2d-1}/\rho$. For all $u\ge u_{\mathrm{IMF}}$,
\begin{equation*}
    \P\Big\{\sup_{\bm \theta\in\mathbb S^{d-1}}X(\bm \theta)>u\Big\}
    \;\le\;
    \delta_{\mathrm{IMF}}(u)
    \;:=\;
    2(k-1)^{\frac{d-1}2}\Big(\alpha_d\,\Phi_d(\rho,u)\,e^{-u^2/2}+\Psi_d(\rho,u)\,e^{-(1+\rho^2)u^2/2}\Big)\,.
\end{equation*}
\end{theorem}

\begin{proof}[Proof of Theorem~\ref{thm:imf_tail}]
Apply Proposition~\ref{prop:IMF_decomp} to the right-hand side of~\eqref{eq:ortho_exact} and substitute the constant prefactor $2(k-1)^{\frac{d-1}2}$ from Proposition~\ref{prop:ortho_exact}.
\end{proof}

\paragraph{A sharper bound $\delta_{\mathrm{IMF}}^\star$.} The IMF bound of Theorem~\ref{thm:imf_tail} performs two relaxations beyond $\delta_{\mathrm{exact}}$: the positivity-driven discarding of $-2\,\mathcal{I}_d^c\,H_{d-1}$ (Lemma~\ref{lem:Idc_positivity}) and the Szeg\H{o} envelope $H_j(\rho x)\le(2\rho x)^j$ on the squared-Hermite tail $\sum_j c_j^2\int_u^\infty H_j(\rho x)^2\,e^{-(1+\rho^2)x^2/2}\,\mathrm{d}x$. We introduce a sharper bound $\delta_{\mathrm{IMF}}^\star$ that performs only the first relaxation: the squared-Hermite tail is evaluated \emph{exactly} via the closed-form recurrence of Lemma~\ref{lem:Jm_beta}. The resulting $\delta_{\mathrm{IMF}}^\star$ is strictly sandwiched between $\delta_{\mathrm{exact}}$ and $\delta_{\mathrm{IMF}}$ on $[u_{\mathrm{IMF}},\infty)$ (Theorem~\ref{thm:imf_tail_sharp}) and can replace $\delta_{\mathrm{IMF}}$ everywhere without additional analytical cost.
\begin{theorem}[Sharpened IMF tail bound]
\label{thm:imf_tail_sharp}
For every $u\ge u_{\mathrm{IMF}}=\sqrt{2d-1}/\rho$,
\begin{equation}\label{eq:imf_sharp}
    \P\Big\{\sup_{\bm \theta\in\mathbb S^{d-1}}X(\bm \theta)>u\Big\}
    \;\le\;
    \delta_{\mathrm{IMF}}^{\star}(u)
    \;:=\;
    2(k-1)^{\frac{d-1}2}\Big[\alpha_d\,\Phi_d(\rho,u)\,e^{-u^2/2}\;+\;T_d^{\mathrm{exact}}(u)\Big]\,,
\end{equation}
where the squared-Hermite tail is evaluated exactly as
\begin{equation}\label{eq:Td_exact_intro}
    T_d^{\mathrm{exact}}(u)\;:=\;\frac1\rho\sum_{j=0}^{d-1} c_j^2\,K_j^\beta(\rho u)\,,
    \qquad
    \beta\;:=\;\frac{1+\rho^2}{\rho^2}\;=\;\frac{3k-2}{k}\,,
\end{equation}
and each $K_j^\beta$ is given in closed form by Corollary~\ref{cor:Kj_beta}. Moreover,
\[
    \delta_{\mathrm{IMF}}^\star(u)\;\le\;\delta_{\mathrm{IMF}}(u)\qquad\text{for every $u\ge u_{\mathrm{IMF}}$,}
\]
with strict inequality at every finite $u$.
\end{theorem}

\begin{proof}[Proof of Theorem~\ref{thm:imf_tail_sharp}]
We follow the proof of Theorem~\ref{thm:imf_tail}, replacing only the squared-Hermite step.

\noindent
$\bullet$ Inserting the Mehta expansion~\eqref{eq:fmehta} into the right-hand side of~\eqref{eq:ortho_exact} and discarding the strictly negative contribution $-2\,\mathcal{I}_d^c(\rho x)\,H_{d-1}(\rho x)$ (which is justified by Lemma~\ref{lem:Idc_positivity} on $\{\rho x\ge\sqrt{2d-1}\}$) reproduces the split~\eqref{eq:IMF_split}:
\[
    \int_u^\infty Q_d(\rho x)\,e^{-x^2/2}\,\mathrm{d}x \;\le\; \alpha_d\,\mathcal{L}_d(u)\;+\;\mathcal{S}_d(u)\,,
\]
with the linear-Hermite component $\mathcal{L}_d(u)=\int_u^\infty H_{d-1}(\rho x)\,e^{-x^2/2}\,\mathrm{d}x\le\Phi_d(\rho,u)\,e^{-u^2/2}$ bounded as in Proposition~\ref{prop:IMF_decomp} (the recurrence-based bound on $\mathcal{L}_d$ is unchanged).

\noindent
$\bullet$ For the squared-Hermite component $\mathcal{S}_d(u)=\sum_{j=0}^{d-1}c_j^2\int_u^\infty H_j(\rho x)^2\,e^{-(1+\rho^2)x^2/2}\,\mathrm{d}x$, the change of variable $y=\rho x$ converts each integral to
\[
    \int_u^\infty H_j(\rho x)^2\,e^{-(1+\rho^2)x^2/2}\,\mathrm{d}x
    \;=\;\frac1\rho\int_{\rho u}^\infty H_j(y)^2\,e^{-\beta y^2/2}\,\mathrm{d}y
    \;=\;\frac1\rho\,K_j^\beta(\rho u)\,,
\]
with $\beta=(1+\rho^2)/\rho^2=(3k-2)/k$, and Corollary~\ref{cor:Kj_beta} gives $K_j^\beta(\rho u)$ in closed form. Summing over $j$ yields $\mathcal{S}_d(u)=T_d^{\mathrm{exact}}(u)$ as defined in~\eqref{eq:Td_exact_intro}, with no inequality used in this step.

\noindent
$\bullet$ Substituting the constant prefactor $2(k-1)^{(d-1)/2}$ of Proposition~\ref{prop:ortho_exact} produces~\eqref{eq:imf_sharp}. Comparing $T_d^{\mathrm{exact}}(u)$ to the Szeg\H{o}-relaxed bound $\Psi_d(\rho,u)\,e^{-(1+\rho^2)u^2/2}$ of Proposition~\ref{prop:IMF_decomp}: by Lemma~\ref{lem:hermite_bound}, $H_j(\rho x)<(2\rho x)^j$ strictly for every $j\ge 2$ and $\rho x>x_1^{(j)}$ (for $j=0,1$ the comparison is equality, $H_0=1=(2\rho x)^0$ and $H_1(\rho x)=2\rho x=(2\rho x)^1$). On $[u_{\mathrm{IMF}},\infty)$, $\rho u\ge\sqrt{2d-1}>x_1^{(j)}$ for every $j\le d-1$, so every term satisfies
\[
    \int_u^\infty H_j(\rho x)^2\,e^{-(1+\rho^2)x^2/2}\,\mathrm{d}x\;\le\;\int_u^\infty(2\rho x)^{2j}\,e^{-(1+\rho^2)x^2/2}\,\mathrm{d}x\,,
\]
with strict inequality whenever $j\ge 2$. Under the standing assumption $d\ge 3$, the sum~\eqref{eq:Td_exact_intro} contains the index $j=2$, contributing strictly; the remaining $j\in\{0,1\}$ contribute equality but with non-negative integrands. Therefore $T_d^{\mathrm{exact}}(u)<\Psi_d(\rho,u)\,e^{-(1+\rho^2)u^2/2}$ on $[u_{\mathrm{IMF}},\infty)$, hence $\delta_{\mathrm{IMF}}^\star(u)<\delta_{\mathrm{IMF}}(u)$.
\end{proof}

\noindent
For $k\ge 3$, $\beta=(3k-2)/k\ge 7/3>2$, so $\theta=(2-\beta)/\beta<0$ in Lemma~\ref{lem:Jm_beta}, and the iterated form~\eqref{eq:Jm_beta_iter} carries alternating signs. The identity~\eqref{eq:Td_exact_intro} remains exact pointwise, but no individual $J_{2j-2p}^\beta$ should be replaced by its absolute value during subsequent estimates; this differentiates the exact evaluation from the Szeg\H{o}-envelope relaxation.

\subsection{Simplified Mehta--Fyodorov bound (SMF)}
\label{sub:SMF}

The SMF bound applies a second relaxation beyond IMF: the Szeg\H{o} envelopes $H_{d-1}(\nu)<(2\nu)^{d-1}$ and $(\nu^2+1)^{d-1}$ collapse $\Phi_d$ and the squared Hermite sum to a single monomial main term $u^{d-2}e^{-u^2/2}$ plus a faster $e^{-3u^2/4}$ remainder, at the cost of a uniform factor $2$ in the leading constant (Theorem~\ref{thm:uniform_domination}). Beyond the level $u^\star_d$ the remainder is dominated and the bound collapses to a single monomial (Corollary~\ref{cor:single_term}); this closed form is used in the inversion of Remark~\ref{rem:choice_u}.

\begin{lemma}[Hermite envelope]
\label{lem:hermite_bound}
For every $m\ge 2$ and every $x>x_1^{(m)}$, $0<H_m(x)<(2x)^m$.
\end{lemma}

\begin{proof}[Proof of Lemma~\ref{lem:hermite_bound}]
Write $H_m(x)=2^m\prod_{j=1}^m(x-x_j^{(m)})$. For $x>x_1^{(m)}$, every factor is strictly positive, so $H_m(x)>0$. For the upper bound, the inequality $H_m(x)<(2x)^m$ is equivalent to $\prod_{j=1}^m(1-x_j^{(m)}/x)<1$. Taking logarithms,
\[
    \sum_{j=1}^m\log(1-x_j^{(m)}/x)\;<\;-\sum_{j=1}^m x_j^{(m)}/x\;=\;0\,,
\]
where the strict inequality uses $\log(1-t)<-t$ for $t\neq 0$ and the equality uses the vanishing sum of Hermite roots, which follows from the absence of the $x^{m-1}$ term in $H_m(x)$.
\end{proof}

\medskip

The Hermite envelope of Lemma~\ref{lem:hermite_bound} is the only tool needed for the SMF relaxation: past the largest root, every Hermite polynomial is majorized by the monomial $(2\nu)^m$, with a strict inequality and no logarithmic corrections. We use it twice, once to collapse the linear Hermite term $\alpha_d H_{d-1}(\nu)$ to a monomial and once inside the squared Hermite sum $\sum_j c_j^2 H_j^2(\nu)$, yielding a single-monomial dominant contribution to $Q_d$ together with a controlled remainder. The next lemma states this.

\begin{lemma}[Non-asymptotic decomposition of $Q_d$]
\label{lem:dominant_term}
\begin{subequations}
There exist explicit constants $\alpha_d,\beta_d>0$ such that
\begin{equation}\label{eq:Qd_decomp}
    Q_d(\nu) \;=\; \alpha_d\,H_{d-1}(\nu)\;+\;\mathcal{R}_d(\nu)\,,
\end{equation}
with the dominant coefficient given by
\begin{equation}\label{eq:alpha_d}
    \alpha_d \;=\;
    \begin{cases}
        \tfrac12\sqrt{d/2}\,c_{d-1}c_d\,\mu_d & \text{if $d$ even},\\[2pt]
        1/\mu_{d-1} & \text{if $d$ odd},
    \end{cases}
\end{equation}
and the remainder bound, for every $\nu\ge 0$,
\begin{equation}\label{eq:remainder_bound}
    |\mathcal{R}_d(\nu)|\;\le\;\beta_d\,(\nu^2+1)^{d-1}\,e^{-\nu^2/2}\,.
\end{equation}
\end{subequations}
\end{lemma}

\begin{proof}
Deferred to Appendix~\ref{app:deferred}.
\end{proof}

The decomposition~\eqref{eq:Qd_decomp} translates into a non-asymptotic upper bound on the Kac--Rice integrand:
\begin{equation}
\notag
    \E\big[|\det(G_{d-1}-\nu I_{d-1})|\big]
    \;\le\;\alpha_d\,\frac{2^{\frac{3}2}\Gamma((d+2)/2)}{d}\,(2\nu)^{d-1}
    \;+\;\frac{2^{\frac{3}2}\Gamma((d+2)/2)}{d}\,\beta_d\,(\nu^2+1)^{d-1}\,e^{-\nu^2/2}\,.
\end{equation}

\medskip

\noindent
Inserting this inequality into the Kac--Rice integral~\eqref{eq:KR_combined} and applying the Gaussian-type tail integrals of Lemma~\ref{lem:gauss_tail} (Appendix~\ref{app:gaussian_integrals}) yields the SMF bound.

\begin{theorem}[SMF tail bound]
\label{thm:smf_tail}
Let $u_{\mathrm{SMF}}=2\sqrt{d}$. For all $u\ge u_{\mathrm{SMF}}$,
\begin{equation}\label{eq:smf_bound}
    \P\Big\{\sup_{\bm \theta\in\mathbb S^{d-1}}X(\bm \theta)>u\Big\}
    \;\le\;
    \delta_{\mathrm{SMF}}(u)
    \;:=\;
    4\,\alpha_d\,(2k)^{\frac{d-1}2}\,u^{d-2}\,e^{-u^2/2}\;+\;2^{d}\,\beta_d\,(k-1)^{\frac{d-1}2}\,u^{2d-3}\,e^{-3u^2/4}\,.
\end{equation}
\end{theorem}

\begin{proof}[Proof of Theorem~\ref{thm:smf_tail}]

\noindent
$\bullet$
By Proposition~\ref{prop:ortho_exact}, the excursion probability is bounded by the integral of the Mehta function:
\[
    \P\Big\{\sup_{\bm \theta}X(\bm \theta)>u\Big\} \;\le\; 2(k-1)^{\frac{d-1}2}\int_u^\infty Q_d(\rho x)\,e^{-x^2/2}\,\mathrm{d}x\,.
\]
Using the non-asymptotic decomposition from Lemma~\ref{lem:dominant_term}, we write $Q_d = \alpha_d H_{d-1} + \mathcal{R}_d$. This splits the corresponding integral into two parts: $I_{\mathrm{main}} + I_{\mathrm{rem}}$.

\noindent
$\bullet$
The main term integral evaluates the dominant Hermite polynomial. By hypothesis, $u \ge u_{\mathrm{SMF}} = 2\sqrt{d}$ and $\rho^2 \ge 1/2$ (since $k \ge 3$). This ensures that $\rho u \ge \sqrt{2d} > \sqrt{2d-1}$, meaning $\rho x > x_1^{(d-1)}$ for all integration variables $x \ge u$. Applying the polynomial envelope from Lemma~\ref{lem:hermite_bound}, we have $0 < H_{d-1}(\rho x) < (2\rho x)^{d-1}$. 
Substituting this strictly positive bound into the main integral gives:
\[
    I_{\mathrm{main}} \;<\; \alpha_d(2\rho)^{d-1}\int_u^\infty x^{d-1}e^{-x^2/2}\,\mathrm{d}x\,.
\]
We evaluate this using Lemma~\ref{lem:gauss_tail} with $a=1$ and $m=d-1$. The required tail hypothesis $(d-2)/u^2 \le 1/2$ holds safely for $u \ge 2\sqrt{d}$, yielding:
\[
    I_{\mathrm{main}} \;\le\; 2\alpha_d(2\rho)^{d-1}u^{d-2}e^{-u^2/2}\,.
\]

\noindent
$\bullet$
Using the remainder envelope~\eqref{eq:remainder_bound}, the absolute value of the remainder integral is bounded by:
\[
    |I_{\mathrm{rem}}| \;\le\; \beta_d\int_u^\infty (\rho^2 x^2+1)^{d-1}e^{-(1+\rho^2)x^2/2}\,\mathrm{d}x\,.
\]
For the polynomial factor, we factor $\rho^2 x^2$ out exactly:
\[
    (\rho^2 x^2+1)^{d-1} \;=\; (\rho^2 x^2)^{d-1}\,\big(1+1/(\rho^2 x^2)\big)^{d-1}
    \;=\; \rho^{2(d-1)}\,x^{2(d-1)}\,\big(1+1/(\rho^2 x^2)\big)^{d-1}\,.
\]
Since $\rho^2 \ge 1/2$ (because $k \ge 3$) and $x \ge u \ge 2\sqrt{d}$, we have $\rho^2 x^2 \ge 2d \ge 1$, so the correction factor satisfies $(1+1/(\rho^2 x^2))^{d-1} \le 2^{d-1}$. Combined with $\rho \le 1$, which gives $\rho^{2(d-1)} \le 1$, this yields $(\rho^2 x^2+1)^{d-1} \le 2^{d-1}\,x^{2(d-1)}$.
For the exponential factor, since $k \ge 3$, we know $1+\rho^2 \ge 3/2$. Substituting these simplifications gives:
\[
    |I_{\mathrm{rem}}| \;\le\; 2^{d-1}\beta_d\int_u^\infty x^{2(d-1)}e^{-3x^2/4}\,\mathrm{d}x\,.
\]
Applying Lemma~\ref{lem:gauss_tail} with $a=3/2$ and $m=2(d-1)$ (where the condition $(2d-3)/(3u^2/2) < 1$ holds for $u \ge 2\sqrt{d}$), we obtain:
\[
    |I_{\mathrm{rem}}| \;\le\; 2^{d-1}\beta_d\,u^{2d-3}e^{-3u^2/4}\,.
\]

\noindent
$\bullet$
We combine the bounded integrals and multiply by the global prefactor $2(k-1)^{\frac{d-1}2}$. 
For the main term, the identity $(2\rho)^{d-1}(k-1)^{\frac{d-1}2} = (2k)^{\frac{d-1}2}$ consolidates the constants.
Summing the two appropriately scaled components yields the final bound:
\[
    \P\Big\{\sup_{\bm \theta}X(\bm \theta)>u\Big\} \;\le\; 4\alpha_d(2k)^{\frac{d-1}2}u^{d-2}e^{-u^2/2} \;+\; 2^d\beta_d(k-1)^{\frac{d-1}2}u^{2d-3}e^{-3u^2/4}\,. \qedhere
\]
\end{proof}

\medskip

The two-term structure of $\delta_{\mathrm{SMF}}(u)$ reflects the two exponential scales inherited from the Mehta expansion: the $e^{-u^2/2}$ scale of the linear Hermite tail, which matches the optimal Kac--Rice decay rate, and the $e^{-3u^2/4}$ scale of the squared Hermite tail relaxed through the $(\nu^2+1)^{d-1}$ envelope. Since the latter is strictly faster, the remainder is asymptotically negligible and the bound collapses to its main term for sufficiently large $u$. The following corollary pinpoints the explicit threshold $u^\star_d\ge u_{\mathrm{SMF}}$ beyond which this collapse occurs, at the mild price of a factor of~$2$ in the leading prefactor.

\begin{corollary}[Asymptotic single-term form]
\label{cor:single_term}
There exists an explicit threshold $u^\star_d\ge u_{\mathrm{SMF}}$, depending only on $(k,d)$, such that for every $u\ge u^\star_d$ the remainder term in~\eqref{eq:smf_bound} is dominated by the main term, in which case
\begin{equation*}
    \P\Big\{\sup_{\bm \theta\in\mathbb S^{d-1}}X(\bm \theta)>u\Big\}
    \;\le\;
    8\,\alpha_d\,(2k)^{\frac{d-1}2}\,u^{d-2}\,e^{-u^2/2}\,.
\end{equation*}
Explicitly, $u^\star_d$ is the smallest $u\ge 2\sqrt d$ satisfying $2^{d-2}\beta_d\,u^{d-1}\,e^{-u^2/4}\le\alpha_d(2\rho)^{d-1}$.
\end{corollary}

\begin{proof}[Proof of Corollary~\ref{cor:single_term}]

\noindent
$\bullet$
Let $R(u)$ denote the ratio of the remainder term to the main term in the SMF bound:
\[
    R(u) \;=\; \frac{2^d\beta_d(k-1)^{\frac{d-1}2}u^{2d-3}e^{-3u^2/4}}{4\alpha_d(2k)^{\frac{d-1}2}u^{d-2}e^{-u^2/2}}\,.
\]
By grouping the constants and using the scaling parameter $\rho = \sqrt{k/(2(k-1))}$, which provides the algebraic relation $(k-1)^{\frac{d-1}2} / (2k)^{\frac{d-1}2} = (2\rho)^{-(d-1)}$, the ratio simplifies to:
\[
    R(u) \;=\; \frac{2^{d-2}\beta_d}{\alpha_d(2\rho)^{d-1}}\,u^{d-1}e^{-u^2/4}\,.
\]

\noindent
$\bullet$
Because the exponential decay $e^{-u^2/4}$ dominates the polynomial growth $u^{d-1}$ for large $u$, the ratio $R(u)$ decreases monotonically toward $0$ as $u \to \infty$ (specifically for $u \ge \sqrt{2(d-1)}$). We require the remainder to be less than or equal to the main term, which corresponds to the condition $R(u) \le 1$. 

We explicitly define the threshold $u^\star_d$ as the smallest level $u \ge u_{\mathrm{SMF}} = 2\sqrt{d}$ that satisfies this bounding condition:
\[
    u^\star_d \;:=\; \inf \left\{ u \ge 2\sqrt{d} \;\mathrel{\Bigg|}\; u^{d-1}e^{-u^2/4} \le \frac{\alpha_d(2\rho)^{d-1}}{2^{d-2}\beta_d} \right\}\,.
\]
Because this threshold is defined by a transcendental equation involving polynomial-exponential terms, $u^\star_d$ is the unique solution to the corresponding equality on the decaying tail, and can be evaluated directly for any fixed pair $(k,d)$. 

\noindent
$\bullet$
By definition of $u^\star_d$, for any $u \ge u^\star_d$ we have $R(u) \le 1$, so the remainder term is dominated by the main term, and the two-term bound of Theorem~\ref{thm:smf_tail} is at most twice the main term:
\[
    \P\Big\{\sup_{\bm \theta}X(\bm \theta)>u\Big\} \;\le\; \text{Main}(u) + \text{Remainder}(u) \;\le\; 2 \times \text{Main}(u)\,.
\]
Substituting the explicit main term yields the single-term envelope:
\[
    \P\Big\{\sup_{\bm \theta}X(\bm \theta)>u\Big\} \;\le\; 8\alpha_d(2k)^{\frac{d-1}2}u^{d-2}e^{-u^2/2}\,. \qedhere
\]
\end{proof}

\subsection{Spectral Method (SM)}
\label{sub:SM}

The Spectral Method does not use the Mehta--Fyodorov algebra: it combines the Ben Arous--Dembo--Guionnet large-deviation bound for the GOE spectral radius (Lemma~\ref{lem:bdg}) with a calibrated layer-cake split. It is an independent cross-check, incurring the same factor-$2$ penalty as the SMF bound and a larger validity threshold (Theorem~\ref{thm:uniform_domination}).

Write $M_{d-1}:=\max_i|\mu_i|$ for the spectral radius of $G_{d-1}$. A two-scale layer-cake split at $R=\sqrt{|\nu|\sqrt{d-1}}$ (well below $|\nu|$, well above the edge $\sim2\sqrt{d-1}$) replaces the crude $|\mu_i-\nu|\le 2|\nu|$ by $|\mu_i-\nu|\le(1+R/|\nu|)|\nu|$ on the bulk event $\{M_{d-1}\le R\}$, the tail $\{M_{d-1}>R\}$ being controlled by the Ben Arous--Dembo--Guionnet bound (Lemma~\ref{lem:bdg}); this removes the $2^{d-1}$ penalty of the naive argument.
\begin{subequations}
\begin{proposition}[SM bound on the expected absolute characteristic polynomial]
\label{prop:det_bound}
Let $G_{d-1}\sim\mathrm{GOE}(d-1)$ with eigenvalues $\mu_1,\dots,\mu_{d-1}$. For every threshold $R=R(d,\nu)$ satisfying $4\sqrt{2(d-1)}\le R\le|\nu|$,
\begin{equation}\label{eq:det_bound}
\E[|\det(G_{d-1}-\nu\,I_{d-1})|] \;=\;
    \E\Big[\prod_{i=1}^{d-1}|\mu_i-\nu|\Big]
    \;\le\;
    |\nu|^{d-1}\Big[\bigl(1+\tfrac{R}{|\nu|}\bigr)^{d-1}+(d+1)\,2^{d-1}\,e^{-2R^2/9}\Big]\,.
\end{equation}
In particular, for $|\nu|\ge 32\sqrt{d-1}$ and the canonical choice $R=\sqrt{|\nu|\sqrt{d-1}}$,
\begin{equation}\label{eq:det_bound_canonical}
    \E\Big[\prod_{i=1}^{d-1}|\mu_i-\nu|\Big]
    \;\le\;|\nu|^{d-1}\bigl(1+\eta_d(\nu)\bigr)\,,\quad
    \eta_d(\nu)
    :=\Bigl(1+\bigl(\tfrac{\sqrt{d-1}}{|\nu|}\bigr)^{\frac12}\Bigr)^{d-1}\!\!\!-1+(d+1)\,2^{d-1}\,e^{-\frac{2|\nu|\sqrt{d-1}}9}\,,
\end{equation}
with
$\eta_d(\nu)\to 0$ as $|\nu|/\sqrt{d-1}\to\infty$.
\end{proposition}
\end{subequations}

\begin{subequations}
\begin{proof}[Proof of Proposition~\ref{prop:det_bound}]
Fix any admissible threshold $R\in[4\sqrt{2(d-1)},|\nu|]$ and decompose the expectation according to $\{M_{d-1}\le R\}$ versus $\{M_{d-1}>R\}$:
\begin{equation*}
    \E\Big[\prod_{i=1}^{d-1}|\mu_i-\nu|\Big]
    \;=\;\E\Big[\prod_{i=1}^{d-1}|\mu_i-\nu|\,\1_{\{M_{d-1}\le R\}}\Big]
    +\E\Big[\prod_{i=1}^{d-1}|\mu_i-\nu|\,\1_{\{M_{d-1}>R\}}\Big]\,.
\end{equation*}

\medskip

\noindent
$\bullet$ 
On $\{M_{d-1}\le R\}$, $|\mu_i-\nu|\le|\mu_i|+|\nu|\le R+|\nu|$, hence the bulk contribution is at most $(R+|\nu|)^{d-1}=|\nu|^{d-1}(1+R/|\nu|)^{d-1}$.

\medskip

\noindent
$\bullet$ 
On $\{M_{d-1}>R\}$, the triangle inequality yields $|\mu_i-\nu|\le M_{d-1}+|\nu|$, and we distinguish two subcases:
\[
    \bigl(|\nu|+M_{d-1}\bigr)^{d-1}\,\1_{\{M_{d-1}>R\}}\;\le\;(2|\nu|)^{d-1}\,\1_{\{R<M_{d-1}\le|\nu|\}}+(2M_{d-1})^{d-1}\,\1_{\{M_{d-1}>|\nu|\}}\,.
\]
By Lemma~\ref{lem:bdg} (whose hypothesis holds since $R\ge 4\sqrt{2(d-1)}$), the first piece contributes at most by a term $(2|\nu|)^{d-1}e^{-2R^2/9}$. For the second piece, the layer-cake formula applied to $Z=(2M_{d-1})^{d-1}\1_{\{M_{d-1}>|\nu|\}}$ gives
\begin{equation*}
    \E[Z]\;=\;2^{d-1}|\nu|^{d-1}\P\{M_{d-1}>|\nu|\}+(d-1)\,2^{d-1}\!\int_{|\nu|}^\infty s^{d-2}\,\P\{M_{d-1}>s\}\,\mathrm{d}s\,,
\end{equation*}
and, since $|\nu|\ge R\ge 4\sqrt{2(d-1)}$, Lemma~\ref{lem:bdg} again applies on $[|\nu|,\infty)$. 

By Lemma~\ref{lem:gauss_tail} with $a=4/9$, $m=d-2$, the integral is at most $|\nu|^{d-1}e^{-2|\nu|^2/9}$ (check that its hypothesis $(d-3)/(4|\nu|^2/9)<1/14$ holds since $|\nu|^2\ge 32(d-1)$), yielding 
\[
    \E[Z]\le d\,2^{d-1}\,|\nu|^{d-1}\,e^{-2|\nu|^2/9}\,.
\]
Combining the two subcases and bounding $e^{-2|\nu|^2/9}\le e^{-2R^2/9}$ (since $R\le|\nu|$),
\begin{equation*}
    \E\bigl[(|\nu|+M_{d-1})^{d-1}\1_{\{M_{d-1}>R\}}\bigr]\;\le\;(d+1)\,2^{d-1}\,|\nu|^{d-1}\,e^{-2R^2/9}\,.
\end{equation*}

\medskip

\noindent
$\bullet$ 
Adding both previous terms yields~\eqref{eq:det_bound}. The canonical choice $R=\sqrt{|\nu|\sqrt{d-1}}$ is admissible precisely when $|\nu|\ge 32\sqrt{d-1}$ (equivalently $R\ge 4\sqrt{2(d-1)}$), in which case $R/|\nu|=(\sqrt{d-1}/|\nu|)^{1/2}$ and $2R^2/9=2|\nu|\sqrt{d-1}/9$, giving~\eqref{eq:det_bound_canonical}.
\end{proof}
\end{subequations}

Proposition~\ref{prop:det_bound} gives a pointwise bound on the expected absolute determinant inside the Kac--Rice integral, with the remaining $x$-dependence concentrated in the leading monomial $(\rho x)^{d-1}$. Integrating against the standard Gaussian weight $\varphi(x)$ then reduces to the Gaussian moment integrals of Lemma~\ref{lem:gauss_tail}. The resulting tail bound inherits the polynomial--exponential envelope $u^{d-2}e^{-u^2/2}$, which already has the optimal asymptotic decay rate; the multiplicative factor $1+\eta_d(\rho,u)$ measures the cost of the layer-cake split and vanishes whenever $\rho u/\sqrt{d-1}\to\infty$.

\begin{theorem}[SM tail bound]
\label{thm:sm_tail}
Let $u_{\mathrm{SM}}=32\frac{\sqrt{d-1}}\rho=32\sqrt{2(d-1)\frac{k-1}k}$. For all $u\ge u_{\mathrm{SM}}$,
\begin{equation}\label{eq:spectral_bound}
    \P\Big\{\sup_{\bm \theta\in\mathbb S^{d-1}}X(\bm \theta)>u\Big\}
    \;\le\;
    \delta_{\mathrm{SM}}(u)
    \;:=\;
    2
    \frac{\sqrt{2}}{\Gamma(d/2)}\,\Big(\frac k2\Big)^{\frac{d-1}2}\!u^{d-2}\,e^{-u^2/2}\,\bigl(1+\eta_d(\rho,u)\bigr)\,,
\end{equation}
where $\eta_d(\rho,u)$ is given by \eqref{eq:eta_def_intro} and is such that $\eta_d(\rho,u)\to 0$ as $\rho u/\sqrt{d-1}\to\infty$.
\end{theorem}

\begin{subequations}
\begin{proof}[Proof of Theorem~\ref{thm:sm_tail}]

\noindent
$\bullet$ 
We apply Proposition~\ref{prop:det_bound} pointwise to the Kac--Rice reduction~\eqref{eq:KR_combined}, setting $\nu = \rho x$ and $R(x) = \sqrt{\rho x\sqrt{d-1}}$.
For any level $u \ge u_{\mathrm{SM}} = 32\sqrt{d-1}/\rho$ and every integration variable $x \ge u$, we have $\rho x \ge 32\sqrt{d-1}$. This guarantees that the admissibility condition $R(x) \ge 4\sqrt{2(d-1)}$ is satisfied. Therefore, the expected absolute determinant is bounded by:
\[
    \E[|\det(G_{d-1}-\rho x\,I_{d-1})|] \;\le\; (\rho x)^{d-1}\bigl(1+\eta_d(\rho x)\bigr)\,,
\]
where $\eta_d(\rho x)$ is defined in~\eqref{eq:det_bound_canonical}. 

\noindent
$\bullet$ 
Since $\eta_d(\cdot)$ is non-increasing, $\eta_d(\rho x) \le \eta_d(\rho u)$ for all $x \ge u$; write $\eta(u) := \eta_d(\rho u)$. Substituting into the Kac--Rice integral yields:
\begin{equation}\label{eq:after_det_bound_app}
    \P\Big\{\sup_{\bm \theta}X(\bm \theta)>u\Big\} \;\le\; C_{k,d}\,\rho^{d-1}\bigl(1+\eta(u)\bigr)\int_u^\infty x^{d-1}\varphi(x)\,\mathrm{d}x\,.
\end{equation}

\noindent
$\bullet$ 
To evaluate the remaining integral, we apply Lemma~\ref{lem:gauss_tail} with $a=1$ and $m=d-1$. The required hypothesis $(d-2)/u^2 \le 1/2$ holds safely since $u \ge u_{\mathrm{SM}} \ge 2\sqrt{d-1}$. This gives:
\[
    \int_u^\infty x^{d-1}\varphi(x)\,\mathrm{d}x \;\le\; \frac{2}{\sqrt{2\pi}} u^{d-2}e^{-u^2/2}\,.
\]
Finally, we substitute the geometric volume prefactor $C_{k,d}=2\sqrt\pi(k-1)^{\frac{d-1}2}/\Gamma(d/2)$ into our bound. Using the identity $\rho^{d-1}(k-1)^{\frac{d-1}2}=(k/2)^{\frac{d-1}2}$, the constants consolidate:
\begin{equation*}
    \P\Big\{\sup_{\bm \theta}X(\bm \theta)>u\Big\} \;\le\; \frac{2\sqrt 2\,(k/2)^{\frac{d-1}2}}{\Gamma(d/2)}\,u^{d-2}\,e^{-u^2/2}\,\bigl(1+\eta(u)\bigr)\,. \qedhere
\end{equation*}
\end{proof}
\end{subequations}

\subsection{Pointwise domination of the IMF bound}
\label{sub:uniform_domination}

\begin{theorem}[Domination of the SMF bound by the IMF bound]
\label{thm:uniform_domination}
For every $k\ge 3$ and $d\ge 3$,
\begin{equation}\label{eq:imf_le_smf}
    \delta_{\mathrm{IMF}}(u)\;\le\;\delta_{\mathrm{SMF}}(u)\qquad\text{for every }u\ge u_{\mathrm{SMF}}=2\sqrt d\,.
\end{equation}
\end{theorem}

The proof uses the geometric envelope of Lemma~\ref{lem:Phi_d_envelope}; the prefactor identity of Lemma~\ref{lem:psm_eq_psmf}, also recorded here, enters the secondary SM comparison of Remark~\ref{rem:sm_domination}.

\begin{lemma}[Asymptotic-prefactor identity $P_{\mathrm{SMF}}=P_{\mathrm{SM}}$]
\label{lem:psm_eq_psmf}
For every integer $d\ge 1$ and every $k\ge 2$,
\begin{equation}\label{eq:psm_eq_psmf}
    4\,\alpha_d\,(2k)^{\frac{d-1}2}\;=\;2\,\frac{\sqrt 2}{\Gamma(d/2)}\,\Big(\frac k2\Big)^{\frac{d-1}2}\,,
\end{equation}
or equivalently $\alpha_d=2^{-(d-1/2)}/\Gamma(d/2)$.
\end{lemma}

\begin{proof}
We verify $\alpha_d\cdot 2^{d-1/2}\Gamma(d/2)=1$ for both parities of $d$.

For $d$ odd, write $d=2m+1$ ($m\ge 0$). Then by Lemma~\ref{lem:dominant_term} and the closed form $\mu_{2m}=\sqrt{2\pi}(2m)!/m!$ recalled in~\eqref{eq:expression_constants},
\[
    \alpha_d\;=\;\frac{1}{\mu_{2m}}\;=\;\frac{m!}{\sqrt{2\pi}\,(2m)!}\,,\qquad
    \Gamma(d/2)\;=\;\Gamma\!\big(m+\tfrac12\big)\;=\;\frac{(2m-1)!!\,\sqrt\pi}{2^m}\,.
\]
Substituting and using the factorisation $(2m)!=2^m\,m!\,(2m-1)!!$,
\[
    \alpha_d\cdot 2^{d-1/2}\Gamma(d/2)
    \;=\;\frac{m!}{\sqrt{2\pi}\,(2m)!}\cdot 2^{2m+1/2}\cdot\frac{(2m-1)!!\sqrt\pi}{2^m}
    \;=\;\frac{m!\,2^m\,(2m-1)!!}{(2m)!}\;=\;1\,.
\]

For $d$ even, write $d=2m$ ($m\ge 1$). Then $\alpha_d=\tfrac12\sqrt{m}\,c_{2m-1}c_{2m}\,\mu_{2m}$ with $c_j=(2^j j!\sqrt\pi)^{-1/2}$ and $\Gamma(d/2)=(m-1)!$. A direct computation yields $c_{2m-1}c_{2m}=1/[\sqrt{\pi}\,2^{2m-1/2}\sqrt{(2m-1)!\,(2m)!}]$ and combining with $\mu_{2m}=\sqrt{2\pi}\,(2m)!/m!$ produces
\[
    \alpha_d\;=\;\frac{\sqrt m\,\sqrt{2\pi}\,(2m)!}{2\sqrt\pi\,2^{2m-1/2}\sqrt{(2m-1)!(2m)!}\,m!}
    \;=\;\frac{\sqrt{m\,(2m)!\,/(2m-1)!}}{2^{2m}\,m!}
    \;=\;\frac{\sqrt{2m^2}}{2^{2m}\,m!}\;=\;\frac{1}{2^{2m-1/2}\,(m-1)!}\,,
\]
using $(2m)!/(2m-1)!=2m$. Hence $\alpha_d\cdot 2^{d-1/2}\Gamma(d/2)=2^{-(2m-1/2)}\,2^{2m-1/2}\,(m-1)!/(m-1)!=1$.
\end{proof}

\begin{lemma}[Geometric envelope on $\Phi_d$]
\label{lem:Phi_d_envelope}
For every $k\ge 3$, $d\ge 3$, and every $u\ge u_{\mathrm{SMF}}=2\sqrt d$,
\begin{equation}\label{eq:Phi_d_envelope}
    \Phi_d(\rho,u)\;\le\;2\,(2\rho)^{d-1}\,u^{d-2}\,.
\end{equation}
\end{lemma}

\begin{proof}
Since $\Lambda=2\rho^2-1=1/(k-1)$ (recalled below~\eqref{eq:expression_constants}), we have $2\Lambda=2/(k-1)$, and we set
\[
    q\;:=\;\frac{2\Lambda}{(2\rho u)^2}\;=\;\frac{2}{(k-1)\,(2\rho u)^2}\,.
\]
On $\{u\ge 2\sqrt d\}$, $\rho^2\ge 1/2$ for $k\ge 3$, hence $(2\rho u)^2\ge 4\rho^2\cdot 4d\ge 8d$, and therefore
\[
    (d-2)\,q\;\le\;\frac{2(d-2)}{8d\,(k-1)}\;=\;\frac{d-2}{4d\,(k-1)}\;\le\;\frac1{4(k-1)}\;\le\;\frac1{8}\,.
\]
The leading term of $\Phi_d$ (the $\ell=0$ summand of~\eqref{eq:phi_def}) equals $2\rho\,(2\rho u)^{d-2}=(2\rho)^{d-1}u^{d-2}$. The $\ell$-th summand for $\ell\ge 1$, divided by the leading term, equals $(d-2)(d-4)\cdots(d-2\ell)\,q^\ell$, a product of $\ell$ factors each at most $(d-2)q\le 1/8$; hence
\[
    \sum_{\ell\ge 1}\frac{\text{$\ell$-th summand}}{\text{leading}}\;\le\;\sum_{\ell\ge 1}(1/8)^\ell\;=\;\frac{1}{7}\,.
\]
Similarly, for $d$ odd, the trailing term $\1_{\{d\,\mathrm{odd}\}}\,u^{-1}\,(2\Lambda)^{(d-1)/2}\,(d-2)!!$ divided by the leading is $q^{(d-1)/2}\,(d-2)!!=\prod_{j=1}^{(d-1)/2}(d-2j)\,q\le(1/8)^{(d-1)/2}\le 1/8$ (the last bound uses $(d-1)/2\ge 1$ for $d\ge 3$). Summing,
\[
    \Phi_d(\rho,u)/\bigl((2\rho)^{d-1}u^{d-2}\bigr)\;\le\;1+\tfrac17+\tfrac18\;=\;\tfrac{71}{56}\;\le\;2\,,
\]
which is~\eqref{eq:Phi_d_envelope}.
\end{proof}

\begin{proof}[Proof of Theorem~\ref{thm:uniform_domination}]
We bound the IMF main and residual terms separately on $[u_{\mathrm{SMF}},\infty)$.

\medskip

\noindent
$\bullet$ \emph{Main term.} On $[u_{\mathrm{SMF}},\infty)$, Lemma~\ref{lem:Phi_d_envelope} gives $\Phi_d(\rho,u)\le 2(2\rho)^{d-1}u^{d-2}$. Multiply by $2(k-1)^{(d-1)/2}\alpha_d e^{-u^2/2}$ and use $\rho^{d-1}(k-1)^{(d-1)/2}=(k/2)^{(d-1)/2}$:
\begin{align*}
    \underbrace{2(k-1)^{\frac{d-1}2}\,\alpha_d\,\Phi_d(\rho,u)\,e^{-u^2/2}}_{\text{IMF main}}
    &\;\le\;4\,\alpha_d\,(2\rho)^{d-1}\,(k-1)^{\frac{d-1}2}\,u^{d-2}\,e^{-u^2/2}\\
    &\;=\;4\,\alpha_d\,(2k)^{\frac{d-1}2}\,u^{d-2}\,e^{-u^2/2}
    \;=\;\underbrace{\delta_{\mathrm{SMF}}^{\mathrm{main}}(u)}_{\text{first term of \eqref{eq:smf_bound}}}\,.
\end{align*}

\noindent
$\bullet$ \emph{Residual.} On $[u_{\mathrm{SMF}},\infty)$, $u^2\ge 4d$ so $(2j-1)/u^2\le(2d-3)/(4d)\le 1/2$ for every $j\le d-1$; combined with $1+\rho^2\ge 3/2$ (since $k\ge 3$), the rational factor in~\eqref{eq:psi_def} satisfies
\[
    \frac{1}{1+\rho^2-(2j-1)/u^2}\;\le\;\frac{1}{1+\rho^2-1/2}\;\le\;1\,.
\]
Hence
\[
    \Psi_d(\rho,u)\;\le\;\frac{c_0^2}{(1+\rho^2)\,u}+\sum_{j=1}^{d-1}c_j^2\,(2\rho)^{2j}\,u^{2j-1}
    \;\le\;u^{2d-3}\sum_{j=0}^{d-1}c_j^2\,(2\rho)^{2j}\,,
\]
using $u^{2j-1}\le u^{2d-3}$ for $j\le d-1$ and $u\ge 1$, the trivial bound $1/[(1+\rho^2)u]\le u^{2d-3}$, and $1/(1+\rho^2)\le 1\le(2\rho)^0=1$ to consolidate the $j=0$ term. Comparing the per-$j$ Szeg\H{o}-style coefficient $c_j^2(2\rho)^{2j}=c_j^2\,4^j\rho^{2j}$ to the entry $c_j^2\,2^j\,((2j)!/j!)^2$ of $S_d$ in~\eqref{eq:expression_constants}: since $(2j)!/j!=(j+1)(j+2)\cdots(2j)\ge 2^j$ for $j\ge 1$, $((2j)!/j!)^2\ge 4^j$, and using $\rho\le 1$, $(2\rho)^{2j}=4^j\rho^{2j}\le 4^j\le 2^j\cdot 4^j\le 2^j((2j)!/j!)^2$, hence $c_j^2(2\rho)^{2j}\le c_j^2\,2^j((2j)!/j!)^2$ termwise (for $j=0$ the comparison reads $c_0^2\le c_0^2$ with equality), giving $\sum_j c_j^2(2\rho)^{2j}\le S_d\le\beta_d$. Therefore
\[
    \underbrace{2(k-1)^{\frac{d-1}2}\,\Psi_d(\rho,u)\,e^{-(1+\rho^2)u^2/2}}_{\text{IMF residual}}
    \;\le\;2(k-1)^{\frac{d-1}2}\beta_d\,u^{2d-3}\,e^{-(1+\rho^2)u^2/2}\,.
\]
Comparing to $\delta_{\mathrm{SMF}}^{\mathrm{rem}}(u):=2^d\beta_d(k-1)^{(d-1)/2}u^{2d-3}e^{-3u^2/4}$, the ratio of exponentials is $e^{-(1+\rho^2)u^2/2+3u^2/4}=e^{-\Lambda u^2/4}$, which lies in $(0,1]$ for $\Lambda=2\rho^2-1=1/(k-1)>0$ (i.e., $k\ge 3$). The ratio of polynomial prefactors is $2/2^d=2^{1-d}\le 1$ for $d\ge 1$. Hence the IMF residual is bounded by $\delta_{\mathrm{SMF}}^{\mathrm{rem}}$ on $[u_{\mathrm{SMF}},\infty)$.

\medskip

\noindent
Adding the two pieces, $\delta_{\mathrm{IMF}}(u)\le\delta_{\mathrm{SMF}}^{\mathrm{main}}(u)+\delta_{\mathrm{SMF}}^{\mathrm{rem}}(u)=\delta_{\mathrm{SMF}}(u)$ on $[u_{\mathrm{SMF}},\infty)$, which is~\eqref{eq:imf_le_smf}.
\end{proof}

\begin{remark}[SM-branch domination: a secondary observation]
\label{rem:sm_domination}
On its own validity range $[u_{\mathrm{SM}},\infty)$ the independent SM bound is also dominated by $\delta_{\mathrm{SMF}}$ (hence, by Theorem~\ref{thm:uniform_domination}, by $\delta_{\mathrm{IMF}}$) for every $(k,d)$ of practical interest, although, unlike Theorem~\ref{thm:uniform_domination}, this is \emph{not} claimed uniformly in $(k,d)$. Indeed, by Lemma~\ref{lem:psm_eq_psmf} the SM and SMF main coefficients are equal:
\[
    \delta_{\mathrm{SMF}}^{\mathrm{main}}(u)\;=\;2\,\frac{\sqrt 2}{\Gamma(d/2)}\,(k/2)^{(d-1)/2}\,u^{d-2}\,e^{-u^2/2}\;=\;\delta_{\mathrm{SM}}(u)/(1+\eta_d(\rho,u))\,.
\]
Therefore
\[
    \delta_{\mathrm{SM}}(u)-\delta_{\mathrm{SMF}}(u)
    \;=\;\delta_{\mathrm{SMF}}^{\mathrm{main}}(u)\,\eta_d(\rho,u)\;-\;\delta_{\mathrm{SMF}}^{\mathrm{rem}}(u)\,.
\]
It suffices to show $\delta_{\mathrm{SMF}}^{\mathrm{rem}}(u)\le\delta_{\mathrm{SMF}}^{\mathrm{main}}(u)\,\eta_d(\rho,u)$ on $[u_{\mathrm{SM}},\infty)$. Dividing by $\delta_{\mathrm{SMF}}^{\mathrm{main}}(u)>0$, this reads
\begin{equation}\label{eq:dom_ii_target}
    h(u)\;:=\;\frac{C\,u^{d-1}\,e^{-u^2/4}}{\eta_d(\rho,u)}\;\le\;1\qquad\text{on $[u_{\mathrm{SM}},\infty)$,}
    \qquad
    C\;:=\;\frac{2^d\,\beta_d\,(k-1)^{(d-1)/2}}{4\,\alpha_d\,(2k)^{(d-1)/2}}\,,
\end{equation}
where the polynomial-prefactor reduction uses $u^{2d-3}/u^{d-2}=u^{d-1}$ and the exponential reduction uses $e^{-3u^2/4}/e^{-u^2/2}=e^{-u^2/4}$.

\medskip

\noindent\emph{Lower bound on $\eta_d(\rho,u)$.} Bernoulli's inequality $(1+\varepsilon)^{d-1}\ge 1+(d-1)\varepsilon$ for $\varepsilon\ge 0$ and $d\ge 2$, applied with $\varepsilon=\sqrt{\sqrt{d-1}/(\rho u)}$ to the first term of~\eqref{eq:eta_def_intro}, yields the $u$-dependent lower bound
\begin{equation}\label{eq:eta_lower}
    \eta_d(\rho,u)\;\ge\;(d-1)\,\bigl(\sqrt{d-1}/(\rho u)\bigr)^{1/2}\;=\;\frac{(d-1)^{5/4}}{(\rho u)^{1/2}}\,,
    \qquad\text{for every $u>0$ and $d\ge 2$.}
\end{equation}

\medskip

\noindent\emph{Monotonicity of $h$.} Substituting~\eqref{eq:eta_lower} into~\eqref{eq:dom_ii_target},
\[
    h(u)\;\le\;\frac{C\,(\rho u)^{1/2}}{(d-1)^{5/4}}\,u^{d-1}\,e^{-u^2/4}
    \;=\;\frac{C\,\rho^{1/2}}{(d-1)^{5/4}}\,u^{d-1/2}\,e^{-u^2/4}\,.
\]
The function $u\mapsto u^{d-1/2}\,e^{-u^2/4}$ has derivative $u^{d-3/2}e^{-u^2/4}\bigl[(d-1/2)-u^2/2\bigr]$ which is non-positive for $u^2\ge 2d-1$. Since $u_{\mathrm{SM}}^2=1024(d-1)/\rho^2\ge 1024(d-1)\ge 2d-1$ for $d\ge 2$, the function is decreasing on $[u_{\mathrm{SM}},\infty)$. Therefore
\begin{equation}\label{eq:h_sup}
    \sup_{u\ge u_{\mathrm{SM}}} h(u)\;\le\;h(u_{\mathrm{SM}})_{\text{upper}}\;:=\;\frac{C\,\rho^{1/2}}{(d-1)^{5/4}}\,u_{\mathrm{SM}}^{d-1/2}\,e^{-u_{\mathrm{SM}}^2/4}\,.
\end{equation}

\medskip

\noindent\emph{Doubly-exponential bound at $u_{\mathrm{SM}}$.} Using $1/\rho\le\sqrt 2$, $1/\rho^2\ge 1$:
\[
    u_{\mathrm{SM}}^{d-1/2}\;=\;\bigl(32\sqrt{d-1}/\rho\bigr)^{d-1/2}\;\le\;\bigl(32\sqrt{2(d-1)}\bigr)^{d-1/2}
    \;=\;(2048)^{(d-1/2)/2}(d-1)^{(d-1/2)/2}\,,
\]
\[
    e^{-u_{\mathrm{SM}}^2/4}\;=\;e^{-256(d-1)/\rho^2}\;\le\;e^{-256(d-1)}\,.
\]
Hence
\begin{equation}\label{eq:uSM_dbl_exp}
    u_{\mathrm{SM}}^{d-1/2}\,e^{-u_{\mathrm{SM}}^2/4}\;\le\;(2048)^{(d-1/2)/2}\,(d-1)^{(d-1/2)/2}\,e^{-256(d-1)}\,.
\end{equation}

\medskip

\noindent\emph{Stirling control of $\beta_d/\alpha_d$.} By Lemma~\ref{lem:psm_eq_psmf}, $\alpha_d=2^{-(d-1/2)}/\Gamma(d/2)$. Stirling's upper bound at $z=d/2$ gives $\Gamma(d/2)\le\sqrt{2\pi}\,(d/2)^{(d-1)/2}e^{-d/2+1/(6d)}$. Combined with $2^{d-1/2}\,(d/2)^{(d-1)/2}\le d^{d/2}\cdot 2^{(d-1)/2}$, this yields $1/\alpha_d\le d^{d/2}\cdot 2^{(d-1)/2}$, so
\begin{equation}\label{eq:alpha_d_lower}
    \alpha_d\;\ge\;d^{-d/2}\cdot 2^{-(d-1)/2}\,.
\end{equation}
For the numerator, by the definition of $\beta_d$ in the proof of Lemma~\ref{lem:dominant_term},
\[
    \beta_d \;\le\; S_d \;+\; \sqrt{d/2}\,c_{d-1}c_d\,\frac{(2d-2)!}{(d-1)!}\,2^{d-1}\,\tilde B_d\,,
\]
with $c_j=(2^j j!\sqrt\pi)^{-1/2}$, $S_d=\sum_{j=0}^{d-1}c_j^2\,2^j\,((2j)!/j!)^2$, and $\tilde B_d=\max((2d)!\,2^{d+1}/d!,\,B_d')$. Each ingredient is bounded by a super-exponential of $d\log d$: using $(2j)!/j!\le(2j)!\le(2j)^{2j}$ and $c_j^2\,2^j\le 1/(j!\sqrt\pi)$, one obtains $S_d\le e^{c\,d\log d}$ for some absolute $c$; similarly $(2d-2)!/(d-1)!\le(2d-2)^{d-1}$ and $\tilde B_d\le(2d)^{2d+1}$, and $\sqrt{d/2}\,c_{d-1}c_d\,2^{d-1}\le d^{d}$ for $d$ large. Multiplying, there is an absolute constant $A>0$ such that
\begin{equation}\label{eq:beta_d_upper}
    \beta_d \;\le\; e^{A\,d\log d}\qquad\text{for every $d\ge 3$.}
\end{equation}
Combining~\eqref{eq:alpha_d_lower} and~\eqref{eq:beta_d_upper},
\begin{equation}\label{eq:beta_over_alpha}
    \beta_d/\alpha_d \;\le\; e^{A\,d\log d}\cdot d^{d/2}\cdot 2^{(d-1)/2} \;\le\; e^{A'\,d\log d}
\end{equation}
for some absolute $A'>0$ and every $d\ge 3$.

\medskip

\noindent\emph{Closing the estimate.} The constant $C$ in~\eqref{eq:dom_ii_target} satisfies $C\le c_1\,(2(k-1))^{(d-1)/2}\,\beta_d/\alpha_d\le c_1\,(2(k-1))^{(d-1)/2}\,e^{A'd\log d}$ for an absolute $c_1>0$. Combining with~\eqref{eq:uSM_dbl_exp}, the bound $u_{\mathrm{SM}}^{d-1/2}\le\bigl(2048(d-1)\bigr)^{(d-1/2)/2}$, and $1/\rho^2\ge 1$,
\[
    h(u_{\mathrm{SM}})_{\text{upper}}\;\le\;\exp\!\Bigl[\,\underbrace{\bigl(A'+\tfrac12\bigr)\,d\log d+\tfrac{d-1}{2}\log\!\bigl(2(k-1)\bigr)+O(d)}_{\text{prefactor growth}}\;-\;\underbrace{256(d-1)/\rho^2}_{\ge\,256(d-1)}\,\Bigr]\,.
\]
The subtracted term is \emph{linear} in $d$, whereas the prefactor growth is of order $d\log d$. Hence the bracket is negative (so that $h(u_{\mathrm{SM}})\le 1$, and by the monotonicity~\eqref{eq:h_sup} $h(u)\le 1$ on all of $[u_{\mathrm{SM}},\infty)$, i.e.\ $\delta_{\mathrm{SMF}}(u)\le\delta_{\mathrm{SM}}(u)$ there) for every $d$ below the threshold $d^\star(k):=\exp\!\bigl(256/(A'+\tfrac12)+O(1)\bigr)$, which far exceeds any dimension of practical interest. Because $d\log d$ ultimately overtakes the linear term, we do \emph{not} claim the inequality for literally all $(k,d)$: the competition between the super-exponential gain $e^{-256(d-1)/\rho^2}$ at $u_{\mathrm{SM}}$ and the $e^{O(d\log d)}$ prefactor growth is genuine. This non-uniformity is immaterial for the rest of the paper: the master failure probability $\delta_{\min}=2\,\delta_{\mathrm{IMF}}$ of Theorem~\ref{thm:main} is unconditional (it is the IMF bound of Theorem~\ref{thm:imf_tail} combined with the two-sided symmetry), and for any $(k,d)$ outside the certified range one invokes the unconditional fallback $\delta_{\min}\le\min(2\delta_{\mathrm{IMF}},2\delta_{\mathrm{SM}})$ on $[u_{\mathrm{SM}},\infty)$ recorded in Corollary~\ref{cor:delta_min_imf}.
\end{remark}

\begin{corollary}[Master bound]
\label{cor:delta_min_imf}
\begin{subequations}
For every $k\ge 3$ and $d\ge 3$, the master failure probability~\eqref{eq:delta_min} equals the IMF branch on its entire validity range,
\begin{equation}\label{eq:delta_min_imf}
    \forall u\ge u_{\mathrm{IMF}}\,,\qquad \delta_{\min}(u)\;=\;2\,\delta_{\mathrm{IMF}}(u)\,,
\end{equation}
and this is the pointwise smaller of the two Mehta--Fyodorov candidates: $2\,\delta_{\mathrm{IMF}}(u)\le 2\,\delta_{\mathrm{SMF}}(u)$ on $[u_{\mathrm{SMF}},\infty)$ by Theorem~\ref{thm:uniform_domination}. On the SM range $[u_{\mathrm{SM}},\infty)$ one further has $2\,\delta_{\mathrm{IMF}}(u)\le 2\,\delta_{\mathrm{SM}}(u)$ for every $(k,d)$ of practical interest (Remark~\ref{rem:sm_domination}); in all cases the conservative bound
\begin{equation}\label{eq:delta_min_conservative}
    \delta_{\min}(u)\;\le\;\min\!\bigl(2\,\delta_{\mathrm{IMF}}(u),\,2\,\delta_{\mathrm{SM}}(u)\bigr)\qquad\text{for }u\ge u_{\mathrm{SM}}
\end{equation}
remains unconditional. In particular,
\begin{equation*}
    \forall u\ge u_{\mathrm{IMF}}\,,\qquad
    \P\bigl\{\Gamma_{1,1}>u\bigr\}\;\le\;2\,\delta_{\mathrm{IMF}}(u)\,,
\end{equation*}
which is the bound used in Theorem~\ref{thm:main}.
\end{subequations}
\end{corollary}

\begin{proof}
The bound $\P\{\Gamma_{1,1}>u\}\le 2\,\delta_{\mathrm{IMF}}(u)$ is Theorem~\ref{thm:imf_tail} combined with the two-sided symmetry~\eqref{eq:two_sided_intro}, hence unconditional; this is~\eqref{eq:delta_min_imf}, as $\delta_{\min}$ is defined by $\delta_{\min}=2\,\delta_{\mathrm{IMF}}$ in~\eqref{eq:delta_min}. Theorem~\ref{thm:uniform_domination} gives $\delta_{\mathrm{IMF}}\le\delta_{\mathrm{SMF}}$ on $[u_{\mathrm{SMF}},\infty)$ for every $k\ge 3,d\ge 3$, so $2\,\delta_{\mathrm{IMF}}$ is the smaller of the two MF candidates there; the comparison with the SM branch on $[u_{\mathrm{SM}},\infty)$ is Remark~\ref{rem:sm_domination}, and where it is not invoked the piecewise minimum~\eqref{eq:delta_min_conservative} is unconditional.
\end{proof}

\section{Statistical analysis of Tensor PCA}
\label{sec:statistical}

\subsection{Geometric domination and rank reduction}
\label{sub:geometric}

Non-asymptotic statistical control of the profile MLE uses two deterministic ingredients: a geometric (Tube Method) inequality bounding the estimation error by the noise supremum on the feasible manifold, and a rank-reduction inequality reducing this constrained noise supremum to the unconstrained rank-one Shub--Smale supremum.

\begin{proposition}[MLE formulation]
\label{prop:mle_derivation}
Under the model~\eqref{eq:generic_tensor_regression} and for any closed subset $\mathcal{C}\subseteq\mathfrak{S}_R$, the Maximum Likelihood Estimator of $(\lambda,\meanTensor)$ over $\R\times\mathcal{C}$ satisfies
\[
    \hat{\putativemeanTensor} \in \operatorname*{arg\,max}_{\putativemeanTensor \in \mathcal{C}} \langle \obs, \putativemeanTensor \rangle_\tensors\,,
    \qquad \hat{\lambda} = \langle \obs, \hat{\putativemeanTensor} \rangle_\tensors\,.
\]
The profile MLE~\eqref{eq:MLE_definition} corresponds to $\mathcal{C}=\{\putativemeanTensor\in\mathfrak{S}_R:\kappa(\putativemeanTensor)\ge\kappa\}$.
\end{proposition}

\begin{proof}[Proof of Proposition~\ref{prop:mle_derivation}]
The negative log-likelihood of $(\lambda,\putativemeanTensor)$ given $\obs$ is, up to additive constants,
\[
    \mathcal{L}(\lambda,\putativemeanTensor)\;=\;\|\obs-\lambda\putativemeanTensor\|_F^2
    \;=\;\|\obs\|_F^2-2\lambda\langle\obs,\putativemeanTensor\rangle_\tensors+\lambda^2\,,
\]
where we used $\|\putativemeanTensor\|_F^2=1$. Minimizing in $\lambda$ for fixed $\putativemeanTensor$ gives $\hat\lambda(\putativemeanTensor)=\langle\obs,\putativemeanTensor\rangle_\tensors$. Substituting back yields the profile cost $\|\obs\|_F^2-\langle\obs,\putativemeanTensor\rangle_\tensors^2$, whose minimization over $\mathcal C$ is equivalent to maximizing $|\langle\obs,\putativemeanTensor\rangle_\tensors|$ over $\mathcal C$. Since $\mathfrak{S}_R$ is closed under sign change $\putativemeanTensor\mapsto-\putativemeanTensor$ (a sign flip preserves rank, Frobenius norm and coherence), one may pick a maximiser $\hat{\putativemeanTensor}$ with $\langle\obs,\hat{\putativemeanTensor}\rangle_\tensors\ge 0$, reducing the joint MLE to $\hat{\putativemeanTensor}\in\operatorname*{arg\,max}_{\putativemeanTensor\in\mathcal C}\langle\obs,\putativemeanTensor\rangle_\tensors$ with $\hat\lambda=\langle\obs,\hat{\putativemeanTensor}\rangle_\tensors\ge 0$.
\end{proof}

\begin{lemma}[Tube Method]
\label{lem:geometric_bound}
Let $\Gamma_{R,\kappa}:=\sup\{|\langle\noise,\putativemeanTensor\rangle_\tensors|:\putativemeanTensor\in\mathfrak{S}_R,\,\kappa(\putativemeanTensor)\ge\kappa\}$. Assume $\meanTensor\in\mathfrak{S}_R$ satisfies $\kappa(\meanTensor)\ge\kappa$. Then any estimator $\hat{\putativemeanTensor}\in\mathfrak{S}_R$ with $\kappa(\hat{\putativemeanTensor})\ge\kappa$ satisfying $\langle\obs,\hat{\putativemeanTensor}\rangle_\tensors\ge\langle\obs,\meanTensor\rangle_\tensors$ obeys
\begin{equation}\label{eq:tube}
    \|\hat{\putativemeanTensor} - \meanTensor\|_F^2 \;\leq\; \frac{4\,\Gamma_{R,\kappa}}{\lambda}\,,\qquad\text{almost surely.}
\end{equation}
\end{lemma}

\begin{proof}[Proof of Lemma~\ref{lem:geometric_bound}]
Substituting $\obs=\lambda\meanTensor+\noise$ into $\langle\obs,\hat{\putativemeanTensor}\rangle_\tensors\ge\langle\obs,\meanTensor\rangle_\tensors$ gives
\[
    \lambda\langle\meanTensor,\hat{\putativemeanTensor}\rangle_\tensors+\langle\noise,\hat{\putativemeanTensor}\rangle_\tensors
    \;\ge\;\lambda\|\meanTensor\|_F^2+\langle\noise,\meanTensor\rangle_\tensors\,.
\]
Using $\|\meanTensor\|_F=\|\hat{\putativemeanTensor}\|_F=1$ together with $\|\hat{\putativemeanTensor}-\meanTensor\|_F^2=2(1-\langle\hat{\putativemeanTensor},\meanTensor\rangle_\tensors)$,
\[
    \tfrac\lambda 2\,\|\hat{\putativemeanTensor}-\meanTensor\|_F^2
    \;\le\;\langle\noise,\hat{\putativemeanTensor}-\meanTensor\rangle_\tensors
    \;\le\;|\langle\noise,\hat{\putativemeanTensor}\rangle_\tensors|+|\langle\noise,\meanTensor\rangle_\tensors|
    \;\le\;2\,\Gamma_{R,\kappa}\,,
\]
where the last inequality uses $\kappa(\hat{\putativemeanTensor}),\kappa(\meanTensor)\ge\kappa$, so that both inner products are bounded by the supremum defining $\Gamma_{R,\kappa}$. Multiplying by $2/\lambda$ yields~\eqref{eq:tube}.
\end{proof}

\begin{remark}[Role of the coherence constraint]
\label{rem:coherence}
The restriction $\kappa(\hat{\putativemeanTensor})\ge\kappa$ is essential. Without it, the noise supremum $\sup\{|\langle\noise,\putativemeanTensor\rangle_\tensors|:\putativemeanTensor\in\mathfrak{S}_R\}$ diverges as the rank-$R$ components $t_j$ become near-collinear: in that regime the coefficients $a_j$ in~\eqref{eq:mean_tensor_rank_r} can be arbitrarily large in magnitude while keeping $\|\putativemeanTensor\|_F^2=1$, since the Gram matrix $\bm G$ becomes singular. Restricting to $\kappa(\hat{\putativemeanTensor})\ge\kappa$ confines the analysis to configurations where $\lambda_{\min}(\bm G)\ge\kappa^2$, ensuring $\|\ba\|_2\le 1/\kappa$ (Lemma~\ref{lem:reduce_rank_one}, Step~1) and a non-trivial tail bound.
\end{remark}

\begin{lemma}[Rank reduction]
\label{lem:reduce_rank_one}
For any tensor $\unnormalizedTensor\in\Tensors$ and any normalized rank-$R$ tensor $\putativemeanTensor\in\mathfrak{S}_R$,
\begin{equation}\label{eq:rank_red}
    |\langle\unnormalizedTensor,\putativemeanTensor\rangle_\tensors|
    \;\le\;\frac{\sqrt R}{\kappa(\putativemeanTensor)}\;\sup_{\bm \theta\in\mathbb S^{d-1}}|\langle\unnormalizedTensor,\bm \theta^{\otimes k}\rangle_\tensors|\,.
\end{equation}
In particular, $\Gamma_{R,\kappa}\le(\sqrt R/\kappa)\,\Gamma_{1,1}$ with $\Gamma_{1,1}$ defined by~\eqref{eq:gamma_11}.
\end{lemma}
\begin{proof}[Proof of Lemma~\ref{lem:reduce_rank_one}]
Let $\putativemeanTensor=\sum_{j=1}^R a_j t_j^{\otimes k}$ be a decomposition attaining the maximum in~\eqref{def:coherence}, and write~$\bm G$ for the Gram matrix $G_{ij}=\langle t_i,t_j\rangle^k$, whose smallest eigenvalue is $\kappa^2(\putativemeanTensor)$ by definition.

\noindent
$\bullet$
The normalization $\|\putativemeanTensor\|_F^2=1$ expands as
\[
    1\;=\;\Big\langle\sum_{i=1}^R a_i t_i^{\otimes k},\,\sum_{j=1}^R a_j t_j^{\otimes k}\Big\rangle_\tensors
    \;=\;\sum_{i,j=1}^R a_i a_j\langle t_i,t_j\rangle^k
    \;=\;\ba^\top\bm G\,\ba\,.
\]
Combined with the Rayleigh quotient inequality $\ba^\top\bm G\,\ba\ge\lambda_{\min}(\bm G)\,\|\ba\|_2^2=\kappa^2(\putativemeanTensor)\,\|\ba\|_2^2$,
\begin{equation}
\notag
    \|\ba\|_2 \;\le\; 1/\kappa(\putativemeanTensor)\,.
\end{equation}

\noindent
$\bullet$ By Cauchy-Schwarz
\[
    |\langle\unnormalizedTensor,\putativemeanTensor\rangle_\tensors|
    \;=\;\Big|\sum_{j=1}^R a_j\,\langle\unnormalizedTensor,t_j^{\otimes k}\rangle_\tensors\Big|
    \;\le\;\|\ba\|_2\,\Big(\sum_{j=1}^R\langle\unnormalizedTensor,t_j^{\otimes k}\rangle_\tensors^2\Big)^{1/2}\,.
\]

\noindent
$\bullet$
Each $t_j\in\mathbb S^{d-1}$, so
\[
    \Big(\sum_{j=1}^R\langle\unnormalizedTensor,t_j^{\otimes k}\rangle_\tensors^2\Big)^{1/2}
    \;\le\;\sqrt R\,\sup_{\bm \theta\in\mathbb S^{d-1}}|\langle\unnormalizedTensor,\bm \theta^{\otimes k}\rangle_\tensors|\,.
\]
Combining the three steps yields~\eqref{eq:rank_red}.
\end{proof}

\subsection{Proof of the main Theorem}
\label{sub:master_proof}

\begin{proof}[Proof of Theorem~\ref{thm:main}]
By construction of the profile MLE~\eqref{eq:MLE_definition}, $\hat{\putativemeanTensor}$ satisfies $\langle\obs,\hat{\putativemeanTensor}\rangle_\tensors\ge\langle\obs,\meanTensor\rangle_\tensors$ and $\kappa(\hat{\putativemeanTensor})\ge\kappa$, so Lemma~\ref{lem:geometric_bound} applies and yields $\|\hat{\putativemeanTensor}-\meanTensor\|_F^2\le 4\Gamma_{R,\kappa}/\lambda$. Lemma~\ref{lem:reduce_rank_one} then gives $\Gamma_{R,\kappa}\le(\sqrt R/\kappa)\,\Gamma_{1,1}$, hence
\begin{equation}\label{eq:proof_master_step1}
    \|\hat{\putativemeanTensor}-\meanTensor\|_F^2 \;\le\; \frac{4\sqrt R}{\kappa\,\lambda}\,\Gamma_{1,1}\,.
\end{equation}
On the complementary event $\{\Gamma_{1,1}\le u\}$, substituting~\eqref{eq:proof_master_step1} immediately yields $\|\hat{\putativemeanTensor}-\meanTensor\|_F^2\le 4\sqrt R\,u/(\kappa\lambda)$, which is~\eqref{eq:main_bound}.

It remains to bound the probability of the bad event $\{\Gamma_{1,1}>u\}$. By the symmetry $X\stackrel{d}{=}-X$ of the Gaussian field, the two-sided supremum satisfies
\begin{equation}\label{eq:two_sided}
    \P\{\Gamma_{1,1}>u\}\;\le\;2\,\P\Big\{\sup_{\bm \theta\in\mathbb S^{d-1}}X(\bm \theta)>u\Big\}\,.
\end{equation}
For $u\ge u_{\mathrm{IMF}}$, Theorem~\ref{thm:imf_tail} bounds the right-hand side by $\delta_{\mathrm{IMF}}(u)$. Multiplying by~$2$ per the symmetry reduction~\eqref{eq:two_sided} and the definition~\eqref{eq:delta_min} of $\delta_{\min}(u):=2\,\delta_{\mathrm{IMF}}(u)$ gives $\P\{\Gamma_{1,1}>u\}\le\delta_{\min}(u)$, which completes the proof. The looser bounds $2\,\delta_{\mathrm{SMF}}(u)$ on $[u_{\mathrm{SMF}},\infty)$ and $2\,\delta_{\mathrm{SM}}(u)$ on $[u_{\mathrm{SM}},\infty)$ (from Theorems~\ref{thm:smf_tail} and~\ref{thm:sm_tail}) are not needed for this proof, but appear in the comparison of Section~\ref{sub:uniform_domination}.
\end{proof}

\subsection{Information-theoretic baselines and the likelihood ratio test}
\label{sub:info_theory}

The geometric error of the profile MLE can be benchmarked against the information-theoretic distances between the planted-signal and pure-noise distributions.

\begin{proposition}[KL divergence, $\chi^2$ divergence, and log-likelihood ratio]
\label{prop:kl_div}
Let $\unnormalizedTensor=\lambda\meanTensor\in\Tensors$ with $\lambda>0$ and $\meanTensor\in\sphereTensors$. Then
\begin{equation*}
    \mathrm{KL}(\P_0\,\|\,\P_{\unnormalizedTensor}) \;=\; \frac{\lambda^2}{2}\,,\qquad
    \chi^2(\P_{\unnormalizedTensor}\,\|\,\P_0) \;=\; e^{\lambda^2}-1\,,
\end{equation*}
and almost surely
\begin{equation}\label{eq:llr}
    \log\frac{\mathrm{d}\P_{\unnormalizedTensor}}{\mathrm{d}\P_0}(\obs) \;=\; \lambda\,\langle\obs,\meanTensor\rangle_\tensors - \frac{\lambda^2}{2}\,.
\end{equation}
\end{proposition}

\begin{proof}[Proof of Proposition~\ref{prop:kl_div}]
The log-likelihood ratio reads
\[
    \log\frac{\mathrm{d}\P_{\unnormalizedTensor}}{\mathrm{d}\P_0}(\obs)
    \;=\;\tfrac12\|\obs\|_F^2-\tfrac12\|\obs-\unnormalizedTensor\|_F^2
    \;=\;\langle\obs,\unnormalizedTensor\rangle_\tensors-\tfrac12\|\unnormalizedTensor\|_F^2
    \;=\;\lambda\langle\obs,\meanTensor\rangle_\tensors-\tfrac{\lambda^2}{2}\,,
\]
since $\|\unnormalizedTensor\|_F=\lambda$ and $\unnormalizedTensor=\lambda\meanTensor$. Under $\P_0$ one has $\E_{\P_0}[\langle\obs,\meanTensor\rangle_\tensors]=\E_{\P_0}[\langle\noise,\meanTensor\rangle_\tensors]=0$ (as $\E[\noise]=0$), so $\E_{\P_0}\!\big[\log\tfrac{\mathrm d\P_{\unnormalizedTensor}}{\mathrm d\P_0}\big]=-\tfrac{\lambda^2}{2}$; hence $\mathrm{KL}(\P_0\,\|\,\P_{\unnormalizedTensor})=-\E_{\P_0}\!\big[\log\tfrac{\mathrm d\P_{\unnormalizedTensor}}{\mathrm d\P_0}\big]=\tfrac{\lambda^2}{2}$. For the $\chi^2$ divergence, the Gaussian moment generating function $\E_{\P_0}\exp(\langle\noise,\bm A\rangle_\tensors)=\exp(\tfrac12\|\bm A\|_F^2)$ gives
\[
    \E_{\P_0}\Big[\Big(\tfrac{\mathrm{d}\P_{\unnormalizedTensor}}{\mathrm{d}\P_0}(\noise)\Big)^2\Big]
    \;=\;e^{-\lambda^2}\,\E_{\P_0}\exp(2\lambda\langle\noise,\meanTensor\rangle_\tensors)
    \;=\;e^{-\lambda^2}\,e^{2\lambda^2}\;=\;e^{\lambda^2}\,,
\]
hence $\chi^2(\P_{\unnormalizedTensor}\,\|\,\P_0)=e^{\lambda^2}-1$.
\end{proof}

The likelihood ratio test (LRT) for the detection problem $H_0:\lambda=0$ versus $H_1:\lambda>0$ amounts, by~\eqref{eq:llr}, to thresholding the test statistic $\hat\lambda_{\mathrm{LRT}}=\max_{\putativemeanTensor\in\mathcal C}\langle\obs,\putativemeanTensor\rangle_\tensors$. Under $H_0$ one has $\obs=\noise$, and since $\mathcal C$ is closed under $\putativemeanTensor\mapsto-\putativemeanTensor$ this statistic equals the (two-sided) noise supremum $\Gamma_{R,\kappa}$. The available tail control is on the \emph{rank-one} field: $\delta_{\min}(u)=2\delta_{\mathrm{IMF}}(u)$ bounds $\P\{\Gamma_{1,1}>u\}$, while the rank-reduction inequality (Lemma~\ref{lem:reduce_rank_one}) only gives the deterministic, one-sided bound $\Gamma_{R,\kappa}\le(\sqrt R/\kappa)\,\Gamma_{1,1}$, with $\sqrt R/\kappa\ge 1$. Fix a target false-positive rate $\alpha\in(0,1)$ and let $u_\alpha$ be the smallest level $u\ge u_{\mathrm{IMF}}$ with
\begin{equation}\label{eq:type_one}
    \delta_{\min}(u_\alpha)\;\le\;\alpha\,.
\end{equation}
The induced LRT rejects $H_0$ when $\hat\lambda_{\mathrm{LRT}}>(\sqrt R/\kappa)\,u_\alpha$, equivalently when the rescaled statistic $(\kappa/\sqrt R)\,\hat\lambda_{\mathrm{LRT}}$ exceeds $u_\alpha$; its Type~I error is then controlled at level $\alpha$,
\[
    \P_{H_0}\!\Big\{\hat\lambda_{\mathrm{LRT}}>\tfrac{\sqrt R}{\kappa}\,u_\alpha\Big\}
    =\P\Big\{\Gamma_{R,\kappa}>\tfrac{\sqrt R}{\kappa}\,u_\alpha\Big\}
    \le\P\{\Gamma_{1,1}>u_\alpha\}\le\delta_{\min}(u_\alpha)\le\alpha\,.
\]
The inflation factor $\sqrt R/\kappa$ is exactly the one carried by the estimation bound~\eqref{eq:main_bound}; for rank-one detection ($R=1,\kappa=1$) it equals $1$. The non-asymptotic rate of Remark~\ref{rem:choice_u} gives $u_\alpha=O(\sqrt{d\log k+\log(1/\alpha)})$, hence a critical value $(\sqrt R/\kappa)\,O(\sqrt{d\log k+\log(1/\alpha)})$. Power against the alternative is governed by the recovery threshold $\lambda\gg\sqrt d$.

\begin{remark}[Numerical inversion of the tail bound]\label{rem:numerical_inversion}
The master bound $\delta_{\min}(u)=2\,\delta_{\mathrm{IMF}}(u)$ of Theorem~\ref{thm:main} is, throughout its validity range $[u_{\mathrm{IMF}},\infty)$, exactly twice the improved Mehta--Fyodorov bound $\delta_{\mathrm{IMF}}(u)$ of Theorem~\ref{thm:imf_tail}. This function admits a closed-form expression involving the Hermite partial sum $\Phi_d(\rho,u)$ and the rational function $\Psi_d(\rho,u)$ and is readily evaluated numerically; it is moreover strictly decreasing on $[u_{\mathrm{IMF}},\infty)$, as is immediate from the Gaussian factors $e^{-u^2/2}$ and $e^{-(1+\rho^2)u^2/2}$ that dominate the polynomial prefactors. Consequently, for any prescribed significance level $\alpha\in(0,1)$ one may apply a standard bisection method to~\eqref{eq:type_one} and obtain, in a few iterations, a numerically sharp estimate of the critical threshold $u_\alpha$ satisfying $\delta_{\min}(u_\alpha)=\alpha$.

This estimate is quantitatively sharp because the IMF bound itself is close to the true Kac--Rice probability: the Monte Carlo experiments in the companion repository \verifrepo{} (\emph{cf.}~Figure~\ref{fig:delta_min}) report a relative discrepancy between $\delta_{\mathrm{IMF}}(u)$ and the Monte Carlo estimate of the Kac--Rice integral $C_{k,d}\int_u^\infty\E[|\det(G_{d-1}-\rho x\,I)|]\,\varphi(x)\,\mathrm{d}x$ of at most $\sim 10^{-2}$ over the tested range of $(k,d)$ and $u$, so that the threshold obtained by inverting $\delta_{\mathrm{IMF}}$ differs from the exact Kac--Rice threshold by at most a fraction of a decimal place at the practical confidence levels $\alpha\in\{10^{-3},5\times 10^{-2}\}$ documented in Figure~\ref{fig:threshold_inversion}.

When ultimate precision is needed, one may alternatively invert the exact Kac--Rice expression~\eqref{eq:KR_combined} itself: the integrand $\E[|\det(G_{d-1}-\rho x\,I)|]\,\varphi(x)$ is estimated by Monte Carlo sampling of $\mathrm{GOE}(d-1)$ in the Mehta convention, quadrature yields a strictly decreasing function of $u$, and a second bisection returns the exact threshold $u_\alpha^{\mathrm{exact}}$. The super-exponential accuracy of the Kac--Rice representation guarantees that $u_\alpha^{\mathrm{exact}}$ is the statistically correct threshold for the likelihood ratio test.
We implemented both routines and confirm that the discrepancy $|u_\alpha-u_\alpha^{\mathrm{exact}}|$ is negligible for practical significance levels.
\end{remark}

\begin{figure}[!t]
\centering
\includegraphics[width=0.95\linewidth]{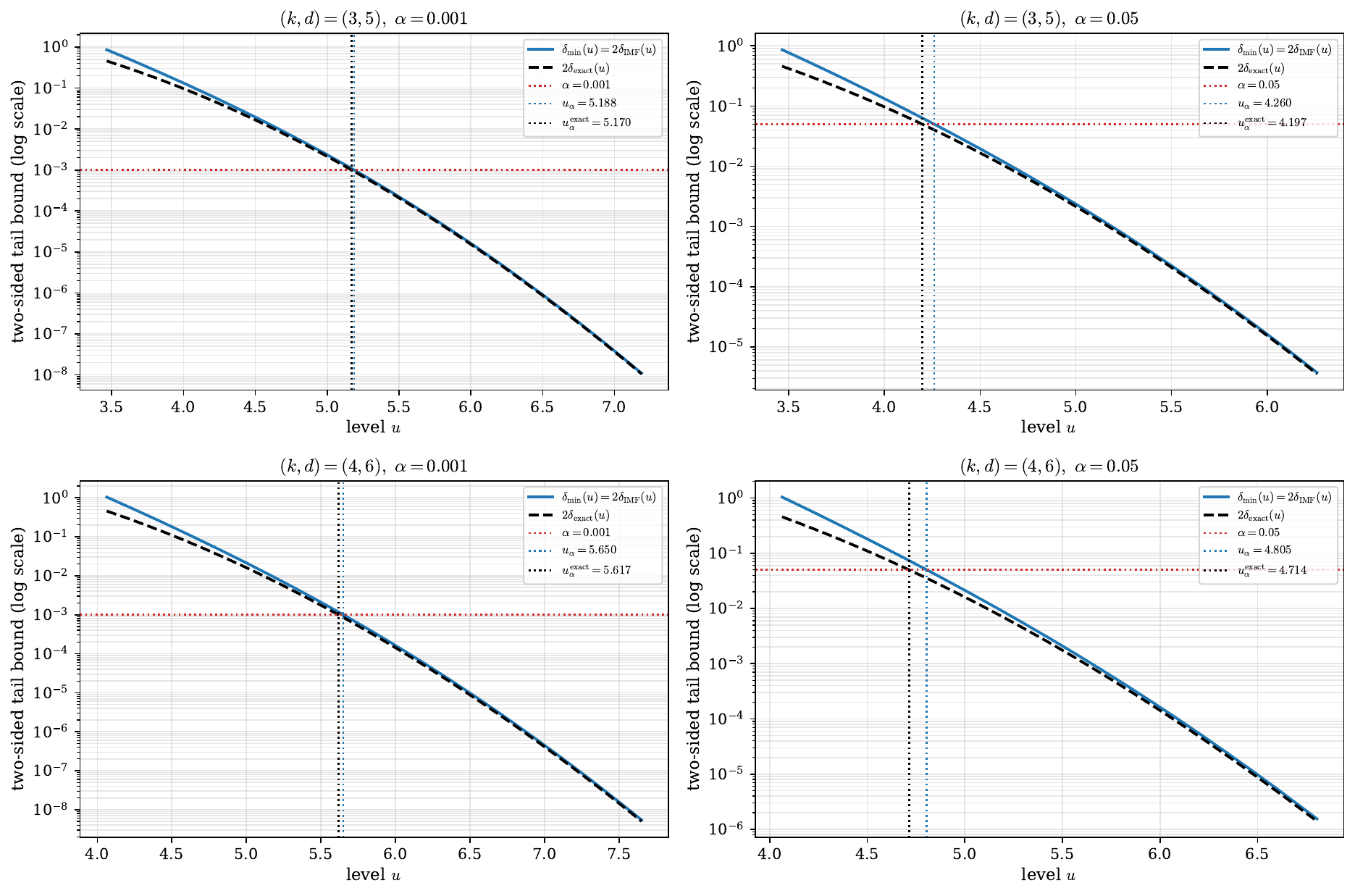}
\caption{Numerical inversion of the \emph{two-sided} master tail bound at significance levels $\alpha\in\{10^{-3},5\times 10^{-2}\}$ for $(k,d)\in\{(3,5),(4,6)\}$. In each panel: the two-sided IMF bound $\delta_{\min}(u)=2\delta_{\mathrm{IMF}}(u)$ (blue), the two-sided exact bound $2\delta_{\mathrm{exact}}(u)$ of Theorem~\ref{thm:imf_tail_exact} (black dashed), the horizontal target level $\alpha$ (red dotted), the threshold $u_\alpha$ obtained by bisecting $\delta_{\min}(u)-\alpha$ (blue dotted vertical), and the exact threshold $u_\alpha^{\mathrm{exact}}$ obtained by bisecting $2\delta_{\mathrm{exact}}(u)-\alpha$ (black dotted vertical). The two-sided $\delta_{\min}$ is the master bound of~\eqref{eq:delta_min} that governs the Type~I error of the LRT, whose statistic is the two-sided supremum $\Gamma_{R,\kappa}$. Since $\delta_{\mathrm{exact}}\le\delta_{\mathrm{IMF}}$ one has $u_\alpha^{\mathrm{exact}}\le u_\alpha$, so the IMF threshold is a conservative surrogate: here $u_\alpha=5.188,\,4.260$ for $(3,5)$ and $5.650,\,4.805$ for $(4,6)$, against $u_\alpha^{\mathrm{exact}}=5.170,\,4.197$ and $5.617,\,4.714$ respectively (the gap widens mildly as $\alpha$ grows and as $d$ increases). The induced LRT rejects $H_0$ when $\hat\lambda_{\mathrm{LRT}}>(\sqrt R/\kappa)\,u_\alpha$ (the rank-reduction inflation of Lemma~\ref{lem:reduce_rank_one}; for rank-one detection $R=\kappa=1$ the factor is~$1$), guaranteeing Type~I error at most~$\alpha$.}
\label{fig:threshold_inversion}
\end{figure}

Figure~\ref{fig:threshold_inversion} illustrates the inversion routine of Remark~\ref{rem:numerical_inversion} on the pairs $(k,d)\in\{(3,5),(4,6)\}$ at the two levels $\alpha\in\{10^{-3},5\times 10^{-2}\}$. In each panel, the two-sided curves $u\mapsto\delta_{\min}(u)=2\delta_{\mathrm{IMF}}(u)$ and $u\mapsto 2\delta_{\mathrm{exact}}(u)$ are plotted on logarithmic scale and intersected with the horizontal line at height~$\alpha$; the two vertical dotted lines report the thresholds $u_\alpha$ (bisection on $\delta_{\min}=2\delta_{\mathrm{IMF}}$) and $u_\alpha^{\mathrm{exact}}$ (bisection on $2\delta_{\mathrm{exact}}$). Three points stand out. \emph{First}, at the stringent level $\alpha=10^{-3}$ the two thresholds agree to within $|u_\alpha-u_\alpha^{\mathrm{exact}}|\le 3.4\times 10^{-2}$ in every tested configuration, so the IMF bound is tight enough to set the LRT threshold at low false-positive rates. \emph{Second}, at the loose level $\alpha=5\times 10^{-2}$ the agreement degrades to $|u_\alpha-u_\alpha^{\mathrm{exact}}|\approx 9\times 10^{-2}$, reflecting the fact that the Hermite and Szeg\H{o} bounds of Lemma~\ref{lem:dominant_term} are increasingly slack as $u$ descends toward $u_{\mathrm{IMF}}$, yet the theoretical threshold $u_\alpha$ remains a strict upper bound on $u_\alpha^{\mathrm{exact}}$, producing a test with actual Type~I error at most~$\alpha$. \emph{Third}, the gap widens with $d$: the super-exponential prefactor $\alpha_d\sim(e/(2d))^{d/2}$ shrinks faster than its numerical surrogate, so the theoretical inversion progressively over-estimates the threshold, which is exactly the \emph{safe} direction for a statistical test.

\begin{remark}[Sharpness of the validity threshold $u_{\mathrm{IMF}}$]
\label{rem:imf_empirical_validity}
Figure~\ref{fig:threshold_inversion} plots $\delta_{\mathrm{IMF}}$ and $\delta_{\mathrm{exact}}$ on $[u_{\mathrm{IMF}},\infty)$, where Theorem~\ref{thm:imf_tail} certifies $\delta_{\mathrm{exact}}\le\delta_{\mathrm{IMF}}$ via the discarding argument of Lemma~\ref{lem:Idc_positivity}. Below $u_{\mathrm{IMF}}$ the discarding argument of Lemma~\ref{lem:Idc_positivity} fails (the positivity of $\mathcal I_d^c(\rho x)$ on $[u_{\mathrm{IMF}},\infty)$ is not guaranteed), so $\delta_{\mathrm{IMF}}(u)$ ceases to be an upper bound on the Kac--Rice integral. An exploratory numerical evaluation (not displayed) suggests that $u_{\mathrm{IMF}}$ is moreover sharp in the following empirical sense: in every tested configuration, $\delta_{\mathrm{IMF}}(u)$ drops strictly below $\delta_{\mathrm{exact}}(u)$ on a window immediately to the left of $u_{\mathrm{IMF}}$, before rising again at very small~$u$ via the rational $1/u$ contribution in $\Phi_d$ for $d$ odd.
\end{remark}

\section{A non-asymptotic annealed complexity}
\label{sec:refinements}

The Kostlan--Shub--Smale field $X(\bm \theta)=\langle\noise,\bm \theta^{\otimes k}\rangle_\tensors$ on $\mathbb S^{d-1}$ coincides with the Hamiltonian of the (centred, isotropic) spherical $k$-spin model from spin-glass physics \citep{auffinger2013,arous2019landscape}. Its critical-point structure, the count of local maxima above an energy level $E$, is a central object of study. \citet{auffinger2013} write $\Theta_p$ for the \emph{total} critical-point growth rate of the order-$p$ model and $\Theta_{0,p}$ for the local-maximum (index-$0$) rate, with $\Theta_{0,p}(u)=\Theta_p(u)$ above the spectral edge (that is, for $u\le-E_\infty$; their Eqs.~(2.14),(2.16)). We write $\Theta_k:=\Theta_{0,p}\big|_{p=k}$ for the local-maximum rate of the $k$-spin field, so that $\lim_{d\to\infty}d^{-1}\log\E[N_{[E,\infty)}^{\mathrm{lm}}]=\Theta_k(-u_{\mathrm{ABC}})$, where $u_{\mathrm{ABC}}=E/\sqrt{d}$ in our normalisation (the ABC energy scale $H/N$ with ambient dimension $N=d$; see Step~3(b) of the proof) and $N_{[E,\infty)}^{\mathrm{lm}}$ counts local maxima with $X(\bm \theta)\ge E$. Here $E_\infty:=2\sqrt{(k-1)/k}$ is the ABC spectral-edge energy \citep[Eq.~(2.13)]{auffinger2013} (above which critical points are local maxima with overwhelming probability; equivalently the right edge of the limiting GOE spectral support, with $\rho E_\infty=\sqrt2$), and $E_0>E_\infty$ is the ground-state energy above which no critical points survive \citep[Theorem~2.5]{auffinger2013}.
Our four-tier hierarchy turns this asymptotic statement into a non-asymptotic two-sided bracketing at every finite $(k,d)$.

\begin{subequations}
\begin{theorem}[Non-asymptotic annealed complexity]
\label{thm:complexity}
Let $k,d\ge 3$, let $X$ be the canonical Kostlan--Shub--Smale field on $\mathbb S^{d-1}$, and let $N_{[E,\infty)}^{\mathrm{lm}}$ denote the number of local maxima of $X$ with value at least $E$. Set
\begin{equation}\label{eq:complexity_thresholds}
    E_{\mathrm{BDG}}\;:=\;\frac{8\sqrt{2(d-1)}}{\rho}\,,\qquad C_{\mathrm{amp}}\;:=\;8\sqrt 2\,,\qquad \delta_{\mathrm{BDG}}(\rho E)\;:=\;e^{-2(\rho E)^2/9}\,.
\end{equation}
\begin{enumerate}
    \item[\textup{(i)}] \emph{(Upper bound.)} For every $E\in\R$,
    \begin{equation}\label{eq:complexity_upper}
        \E\big[N_{[E,\infty)}^{\mathrm{lm}}\big]\;\le\;\E\big[N_{[E,\infty)}^{\mathrm{cp}}\big]\;=\;\delta_{\mathrm{exact}}(E)\,,
    \end{equation}
    where $N_{[E,\infty)}^{\mathrm{cp}}$ counts \emph{all} critical points (not only local maxima) with $X(\bm\theta)\ge E$. 
    \item[\textup{(ii)}] \emph{(Two-sided bracketing, high-energy regime.)} For every $E\ge E_{\mathrm{BDG}}$,
    \begin{equation}\label{eq:complexity_bracket}
        \big(1-C_{\mathrm{amp}}^{d-1}\,\delta_{\mathrm{BDG}}^{1/2}(\rho E)\big)\,\delta_{\mathrm{exact}}(E)
        \;\le\;\E\big[N_{[E,\infty)}^{\mathrm{lm}}\big]
        \;\le\;\delta_{\mathrm{exact}}(E)\,.
    \end{equation}
\end{enumerate}
\end{theorem}
\end{subequations}

\begin{subequations}
\begin{corollary}[Annealed complexity of very deep local minima, by sign symmetry]
\label{cor:minima_dual}
Let $N_{(-\infty,E']}^{\mathrm{lm,min}}$ denote the number of local minima of the Kostlan--Shub--Smale field $X$ on $\mathbb S^{d-1}$ with value at most $E'$. The Gaussian symmetry $X\stackrel d=-X$ gives the distributional identity $N_{(-\infty,E']}^{\mathrm{lm,min}}\stackrel d= N_{[-E',\infty)}^{\mathrm{lm}}$ for every $E'\in\R$, and Theorem~\ref{thm:complexity} transfers verbatim:
\begin{enumerate}
    \item[\textup{(i$'$)}] \emph{(Upper bound.)} For every $E'\in\R$,
    \begin{equation}\label{eq:complexity_min_upper}
        \E\big[N_{(-\infty,E']}^{\mathrm{lm,min}}\big]\;\le\;\delta_{\mathrm{exact}}(-E')\,.
    \end{equation}
    \item[\textup{(ii$'$)}] \emph{(Two-sided bracketing, very-deep-minima regime.)} For every $E'\le -E_{\mathrm{BDG}}$ (equivalently, $|E'|\ge 8\sqrt{2(d-1)}/\rho$),
    \begin{equation}\label{eq:complexity_min_bracket}
        \big(1-C_{\mathrm{amp}}^{d-1}\,\delta_{\mathrm{BDG}}^{1/2}(-\rho E')\big)\,\delta_{\mathrm{exact}}(-E')
        \;\le\;\E\big[N_{(-\infty,E']}^{\mathrm{lm,min}}\big]
        \;\le\;\delta_{\mathrm{exact}}(-E')\,.
    \end{equation}
\end{enumerate}
The validity range $E'\le -E_{\mathrm{BDG}}$ corresponds asymptotically to the reduced energy $E'/\sqrt d\le -8E_\infty$: the bracket covers \emph{very deep local minima}, far below the ground-state band $[-E_0,-E_\infty]$ identified by \citet[Theorem~2.5]{auffinger2013}. Local minima in the moderate-energy band $E'/\sqrt d\in[-8E_\infty,-E_\infty]$ (where \citet[Theorem~2.4]{auffinger2013} provides the case-1 log-rate $\Theta_{0,p}$) and in the bulk band $E'/\sqrt d\in[-E_\infty,0]$ (where ABC give the constant log-rate $\tfrac12\log(p-1)-(p-2)/p$) are outside the scope of the BDG-amplifier framework developed in this paper; the upper bound \textup{(i$'$)} remains unconditional throughout, however.
\end{corollary}
\end{subequations}

\begin{remark}[Exact Kac--Rice representation]
\label{rem:kr_representation}
The Kac--Rice formula combined with the conditional shifted-GOE law of the Hessian \citep[Lemma~3.2(b)]{auffinger2013} yields, for every $E\in\R$,
\begin{equation}\label{eq:complexity_exact}
    \E\big[N_{[E,\infty)}^{\mathrm{lm}}\big]
    \;=\;C_{k,d}\int_E^\infty\E\big[|\det(G_{d-1}-\rho x\,I_{d-1})|\,\mathds 1\{G_{d-1}-\rho x I\prec 0\}\big]\,\varphi(x)\,\mathrm{d}x\,,
\end{equation}
with $C_{k,d}=2\sqrt\pi(k-1)^{(d-1)/2}/\Gamma(d/2)$ and $\varphi$ the standard Gaussian density. The upper bound~\eqref{eq:complexity_upper} is obtained by dropping the indicator and applying Proposition~\ref{prop:ortho_exact}: $$\E[N^{\mathrm{lm}}_{[E,\infty)}]\le\E[\#\{\text{critical points of }X\text{ above }E\}]=\delta_{\mathrm{exact}}(E).$$
\end{remark}


\begin{subequations}

\begin{corollary}[Asymptotic match with \citet{auffinger2013} in the high-energy regime]
\label{cor:abc_match}
Recall $E_\infty$, $E_0$, and the local-maximum rate $\Theta_k$ from the start of this section: exponentially many local maxima persist throughout the band $E_\infty<E/\sqrt d<E_0$, where $\Theta_k(-eE_\infty)>0$. The ABC complexity function~\citep[Theorem~2.4 and Eq.~(2.15)]{auffinger2013} is given for $u\le -E_\infty$ by
\begin{equation}\label{eq:Thetap_definition}
    \Theta_k(u)\;=\;\tfrac12\log(k-1)\,-\,\tfrac{k-2}{4(k-1)}u^2\,-\,I_1(u)\,,\qquad I_1(u):=\tfrac{2}{E_\infty^2}\int_u^{-E_\infty}\sqrt{z^2-E_\infty^2}\,\mathrm{d}z\,.
\end{equation}
As $d\to\infty$ with $E=e\sqrt{2(d-1)}/\rho$ for fixed reduced energy $e\ge 1$, the level correspondence $E/\sqrt d\to eE_\infty$ holds, and the exact Kac--Rice equality $\E[N^{\mathrm{cp}}_{[E,\infty)}]=\delta_{\mathrm{exact}}(E)$ of part~\textup{(i)} together with \citet[Theorem~2.4]{auffinger2013} gives
\begin{equation}\label{eq:complexity_asymp}
    \lim_{d\to\infty}\frac{1}{d}\log\E\big[N_{[E,\infty)}^{\mathrm{lm}}\big]
    \;=\;\Theta_k(-eE_\infty)\,,
\end{equation}
which, after the substitution $z=-E_\infty w$ in $I_1$, yields the explicit closed form
\begin{equation}\label{eq:Thetap_explicit}
    \Theta_k(-eE_\infty)\;=\;\tfrac12\log(k-1)\,-\,\tfrac{(k-2)e^2}{k}\,-\,2\!\int_1^e\sqrt{w^2-1}\,\mathrm{d}w\,.
\end{equation}
\end{corollary}
\end{subequations}
\begin{proof}
Deferred to Appendix~\ref{app:deferred}.
\end{proof}
The threshold $e\ge 8$ governs only the \emph{non-asymptotic two-sided bracket}~\eqref{eq:complexity_bracket} -- the validity range of the BDG amplifier in Theorem~\ref{thm:complexity}\textup{(ii)} ($E\ge E_{\mathrm{BDG}}=8\sqrt{2(d-1)}/\rho$). The asymptotic identity~\eqref{eq:complexity_asymp} itself holds on the full ABC range $e\ge 1$, i.e.\ $u_{\mathrm{ABC}}\ge E_\infty$ for maxima (equivalently $u_{\mathrm{ABC}}\le-E_\infty$ for minima after sign symmetry), the case-1 branch of $\Theta_k$, since it follows from the exact equality of part~\textup{(i)} and \citet[Theorem~2.4]{auffinger2013} rather than from the bracket. The gap $e\in[1,8)$ -- moderate maxima above the spectral edge but below the BDG-amplifier threshold -- is therefore still covered by the asymptotic~\eqref{eq:complexity_asymp} (via the GOE semicircle confinement of \citet[Theorem~2.1.22]{anderson2010introduction}) and by the unconditional upper bound~\eqref{eq:complexity_upper}, but lies \emph{outside the reach of the non-asymptotic two-sided bracket}: the BDG amplifier collapses below $e=8$, and a matching finite-$d$ \emph{lower} bracket on $e\in[1,8)$ would require finer control of the largest eigenvalue of the GOE near the spectral edge (e.g.~the LDP of \citet[Theorem~6.2]{ben-guionnet}). We leave that moderate-energy lower bracket as an open problem; in that range the asymptotic~\eqref{eq:complexity_asymp} and the upper bound~\eqref{eq:complexity_upper} -- which holds for every $E\in\R$ -- remain in force.
\begin{remark}[Normalisation reconciliation with \citet{auffinger2013}]
\label{rem:abc_normalisation}
\citet{auffinger2013} use $H_{N,p}$ on the sphere of radius $\sqrt N$ ($\E[H^2]=N$) with energy $u=H/N$; we use $f=H(\sqrt N\,\cdot)/\sqrt N$ on the unit sphere ($\E[f^2]=1$) with energy $E$ on the scale $\sqrt N\,u$. The reduced energy satisfies $e=u_{\mathrm{ABC}}/E_\infty$ by~\eqref{eq:uABC_to_e}, identifying the two conventions without further rescaling.
\end{remark}

\begin{remark}[Four nested non-asymptotic complexity bounds]
\label{rem:complexity_hierarchy}
Substituting $\delta_{\mathrm{IMF}},\delta_{\mathrm{SMF}},\delta_{\mathrm{SM}}$ for $\delta_{\mathrm{exact}}$ in~\eqref{eq:complexity_bracket} yields four nested non-asymptotic bounds on the annealed complexity: $\delta_{\mathrm{exact}}$ tightest (Theorem~\ref{thm:imf_tail_exact}), $\delta_{\mathrm{IMF}}$ asymptotically sharp, $\delta_{\mathrm{SMF}}$ inversion-friendly for complexity-vs-energy curves (Corollary~\ref{cor:single_term}), and $\delta_{\mathrm{SM}}$ independent of the Mehta--Fyodorov algebra.
\end{remark}

\begin{remark}[Spiked landscape: connection to the BAMMN result]
\label{rem:spiked_landscape}
\citet{arous2019landscape} extend the complexity to the rank-one spiked model, where the conditional Hessian becomes the rank-one deformation $\widetilde G_{d-1}=\theta(m)\,e_1 e_1^\top+G_{d-1}-t(m,x)I_{d-1}$ with $\theta(m)=\sqrt{2k(k-1)}\,\lambda m^{k-2}(1-m^2)$, $t(m,x)=\sqrt{2k/(k-1)}\,x$ \citep[Eq.~(4.42)]{arous2019landscape}; our unspiked result~\eqref{eq:complexity_asymp} is its $\lambda=0$ marginal ($\theta(m)=0$). The \emph{asymptotic} spiked complexity is well understood: \citet{piccolo2023topological} establishes the large-$d$ annealed complexity for an arbitrary finite number $r$ of spikes of mixed degrees, replacing the rank-one large deviation of \citet{maida2007large} by the finite-rank spherical-integral asymptotics of \citet{guionnet2022spherical}. Its $\lambda=0$ specialisation coincides with~\eqref{eq:complexity_asymp} after the noise-normalisation change of variables $u=\sqrt 2\,x$ (both equal the Auffinger--Ben Arous--\v{C}ern\'y rate, as we checked numerically). Our framework leaves open the \emph{non-asymptotic, finite-$(k,d)$} spiked bracket. This would require redoing the Mehta--Fyodorov algebra for the rank-one-deformed ensemble, which has no closed Fyodorov representation \citep[Lemma~4.4]{arous2019landscape}, together with non-asymptotic counterparts of those large-deviation and spherical-integral inputs (compare \citealp[Theorem~1.1]{maida2007large} and \citealp[Theorem~6]{guionnet2005fourier}). We leave that finite-$d$ refinement as a follow-up.
\end{remark}

\bibliography{shub_tensor_bibliography}

\newpage
\appendix

\section{Deferred proofs}
\label{app:deferred}

For readability the three longest proofs are collected here.

\begin{proof}[Proof of Theorem~\ref{thm:imf_tail_exact}]
We check each $D_i$ in turn; only $D_2$ requires care.

\medskip

\noindent
$\bullet$ \emph{$D_4$ (squared Hermite).} This is exactly the content of Theorem~\ref{thm:imf_tail_sharp}: by the change of variable $y=\rho x$ and Corollary~\ref{cor:Kj_beta},
\[
    D_4(u)\;=\;T_d^{\mathrm{exact}}(u)\;=\;\rho^{-1}\sum_{j=0}^{d-1}c_j^2\,K_j^\beta(\rho u)\,,\qquad\beta=(3k-2)/k\,.
\]

\medskip

\noindent
$\bullet$ \emph{$D_1,D_3$ (linear Hermite).} The change of variable $y=\rho x$ in $\mathcal{L}_d(u)=\int_u^\infty H_{d-1}(\rho x)\,e^{-x^2/2}\,\mathrm{d}x$ gives $\mathcal{L}_d(u)=\rho^{-1}J_{d-1}^{1/\rho^2}(\rho u)$, and Lemma~\ref{lem:Jm_beta} with $\beta=1/\rho^2$, $\gamma=2\rho^2$, $\theta=2\rho^2-1=\Lambda$ supplies the closed form. For $d$ odd, $D_3$ exhibits a Mills-ratio-type remainder $\bar\Phi(u\sqrt\beta)$ (since $d-1$ is even and $J_0^{1/\rho^2}$ enters); for $d$ even, $D_1$ has no such remainder (since $d-1$ is odd).

\medskip

\noindent
$\bullet$ \emph{$D_2$, cross-integral, polynomial part.} Apply Lemma~\ref{lem:hermite_tail_explicit} to $\mathcal{I}_d^c(\rho x)$:
\[
    \mathcal{I}_d^c(\rho x)\;=\;e^{-(\rho x)^2/2}\sum_{\ell=0}^{\lfloor(d-1)/2\rfloor}\!\frac{2^{\ell+1}(d-1)!!}{(d-2\ell-1)!!}\,H_{d-2\ell-1}(\rho x)\;+\;\1_{\{d\,\mathrm{even}\}}\,2^{d/2}(d-1)!!\!\!\int_{\rho x}^\infty\!e^{-y^2/2}\,\mathrm{d}y\,.
\]
Substituting into~\eqref{eq:Cd_def}, the Hermite-series part becomes
\[
    \sum_\ell\frac{2^{\ell+1}(d-1)!!}{(d-2\ell-1)!!}\!\!\int_u^\infty\!H_{d-1}(\rho x)\,H_{d-2\ell-1}(\rho x)\,e^{-(1+\rho^2)x^2/2}\,\mathrm{d}x\,.
\]
The classical product linearization $H_a(\rho x)\,H_b(\rho x)=\sum_{p=0}^{\min(a,b)}2^p\,p!\,\binom{a}{p}\binom{b}{p}\,H_{a+b-2p}(\rho x)$, combined with $y=\rho x$ and Lemma~\ref{lem:Jm_beta} with $\beta=(3k-2)/k$, evaluates each integral to $\rho^{-1}J_{a+b-2p}^\beta(\rho u)$, producing the first line of~\eqref{eq:Cd_closed}.

\medskip

\noindent
$\bullet$ \emph{$D_2$, cross-integral, Fubini remainder ($d$ even only).} The remainder contribution to $\mathcal{C}_d(u)$ is, after substituting $\mathcal{I}_d^c$'s Gaussian-tail piece,
\[
    2^{d/2}(d-1)!!\,\mathcal{F}_d(u)\,,\qquad \mathcal{F}_d(u)\;:=\;\int_u^\infty H_{d-1}(\rho x)\,\Big[\!\!\int_{\rho x}^\infty e^{-y^2/2}\,\mathrm{d}y\Big]\,e^{-x^2/2}\,\mathrm{d}x\,.
\]
The double integral is over $\{(x,y):x\ge u,\ y\ge\rho x\}$. By Fubini, swapping the order, $y$ ranges over $[\rho u,\infty)$ and for each such $y$, $x$ ranges over $[u,y/\rho]$:
\[
    \mathcal{F}_d(u)\;=\;\int_{\rho u}^\infty\!e^{-y^2/2}\!\int_u^{y/\rho}\!H_{d-1}(\rho x)\,e^{-x^2/2}\,\mathrm{d}x\,\mathrm{d}y
    \;=\;\frac1\rho\!\int_{\rho u}^\infty\!e^{-y^2/2}\bigl[J_{d-1}^{1/\rho^2}(\rho u)-J_{d-1}^{1/\rho^2}(y)\bigr]\,\mathrm{d}y\,,
\]
where the second equality applies the change of variable $z=\rho x$ to the inner integral and writes the difference of $J^{1/\rho^2}_{d-1}$ at the two endpoints. Splitting at the bracket gives
\begin{equation}\label{eq:Fd_split}
    \mathcal{F}_d(u)\;=\;\underbrace{\frac{J_{d-1}^{1/\rho^2}(\rho u)}{\rho}\!\int_{\rho u}^\infty\!e^{-y^2/2}\,\mathrm{d}y}_{=:\,\mathcal{F}_d^{(1)}(u)}\;-\;\underbrace{\frac{1}{\rho}\!\int_{\rho u}^\infty\!e^{-y^2/2}\,J_{d-1}^{1/\rho^2}(y)\,\mathrm{d}y}_{=:\,\mathcal{F}_d^{(2)}(u)}\,.
\end{equation}

\emph{First piece $\mathcal F_d^{(1)}$.} Direct evaluation: $\int_{\rho u}^\infty e^{-y^2/2}\,\mathrm{d}y = \sqrt{2\pi}\,\bar\Phi(\rho u)$, so
\[
    \mathcal{F}_d^{(1)}(u)\;=\;\sqrt{2\pi}\,\bar\Phi(\rho u)\,\frac{J_{d-1}^{1/\rho^2}(\rho u)}{\rho}\,,
\]
which is the first term of~\eqref{eq:Fubini_F}.

\emph{Second piece $\mathcal F_d^{(2)}$.} For $d$ even, $d-1$ is odd, so the iterated form~\eqref{eq:Jm_beta_iter} of Lemma~\ref{lem:Jm_beta} applied with $\beta=1/\rho^2$, $\gamma=2/\beta=2\rho^2$, $\theta=(2-\beta)/\beta=2\rho^2-1=\Lambda$ has $R^{1/\rho^2}_{d-1}=0$ and reads
\[
    J_{d-1}^{1/\rho^2}(y)\;=\;2\rho^2\,e^{-y^2/(2\rho^2)}\!\sum_{k=0}^{(d-2)/2}(2\Lambda)^k\,\frac{(d-2)!!}{(d-2k-2)!!}\,H_{d-2k-2}(y)\,.
\]
Substituting into $\mathcal{F}_d^{(2)}(u)$ and pulling the constants $2\rho^2$ and the outer $1/\rho$ outside the sum,
\begin{align*}
    \mathcal{F}_d^{(2)}(u)
    &\;=\;\frac{1}{\rho}\!\int_{\rho u}^\infty\!e^{-y^2/2}\cdot 2\rho^2\,e^{-y^2/(2\rho^2)}\sum_{k=0}^{(d-2)/2}(2\Lambda)^k\,\frac{(d-2)!!}{(d-2k-2)!!}\,H_{d-2k-2}(y)\,\mathrm{d}y\\
    &\;=\;\underbrace{\frac{2\rho^2}{\rho}}_{=\,2\rho}\;\sum_{k=0}^{(d-2)/2}(2\Lambda)^k\,\frac{(d-2)!!}{(d-2k-2)!!}\!\int_{\rho u}^\infty\!H_{d-2k-2}(y)\,e^{-(1+1/\rho^2)y^2/2}\,\mathrm{d}y\\
    &\;=\;2\rho\,\sum_{k=0}^{(d-2)/2}(2\Lambda)^k\,\frac{(d-2)!!}{(d-2k-2)!!}\,J_{d-2k-2}^{\beta}(\rho u)\,,
\end{align*}
where the combined exponent $-(1/\rho^2+1)y^2/2=-\beta y^2/2$ identifies $\beta=1+1/\rho^2=(1+\rho^2)/\rho^2=(3k-2)/k$, the same $\beta$ appearing in $D_4$ (Corollary~\ref{cor:Kj_beta}), and the last line applies the definition $J_m^\beta(\rho u)=\int_{\rho u}^\infty H_m(y)\,e^{-\beta y^2/2}\,\mathrm{d}y$.

\emph{Assembly.} Substituting $\mathcal{F}_d^{(1)}$ and $\mathcal{F}_d^{(2)}$ into~\eqref{eq:Fd_split} and noting the explicit minus sign in front of $\mathcal{F}_d^{(2)}$, the second-piece coefficient becomes $-2\rho$, yielding the second line of~\eqref{eq:Fubini_F}.

\medskip

\noindent
$\bullet$ \emph{Pointwise tightness.} The chain $\delta_{\mathrm{exact}}\le\delta_{\mathrm{IMF}}^\star\le\delta_{\mathrm{IMF}}$ follows from the chain of relaxations: $\delta_{\mathrm{IMF}}^\star$ discards $-2\mathcal{I}_d^c\,H_{d-1}\,\le 0$ on $\{\rho x\ge\sqrt{2d-1}\}$, and $\delta_{\mathrm{IMF}}$ further replaces each $H_j(\rho x)\le(2\rho x)^j$ via Lemma~\ref{lem:hermite_bound}. Both inequalities are strict at every finite $u$.
\end{proof}

\medskip

\begin{proof}[Proof of Lemma~\ref{lem:dominant_term}]

\noindent
$\bullet$
We separate the Mehta expansion~\eqref{eq:fmehta} into three distinct pieces, $Q_d = T_1 + T_2 + T_3$, defined as follows:
\begin{align*}
    T_1(\nu) &= e^{-\nu^2/2}\sum_{j=0}^{d-1}c_j^2 H_j^2(\nu)\,, \\
    T_2(\nu) &= \tfrac{1}{2}\sqrt{d/2}\,c_{d-1}c_d\,H_{d-1}(\nu)\,\big[\mu_d-2\mathcal{I}_d^c(\nu)\big]\,, \\
    T_3(\nu) &= \1_{\{d\text{ odd}\}}\,\frac{H_{d-1}(\nu)}{\mu_{d-1}}\,.
\end{align*}

\noindent
$\bullet$
We extract the constant prefactor multiplying $H_{d-1}(\nu)$, which depends on the parity of the dimension $d$:
\begin{itemize}
    \item \textbf{For $d$ even:} $T_3 = 0$, and the constant part of the bracket in $T_2$ isolates $\alpha_d = \tfrac{1}{2}\sqrt{d/2}\,c_{d-1}c_d\,\mu_d$.
    \item \textbf{For $d$ odd:} $\mu_d = 0$ (by parity), leaving only $T_3$ to contribute. This gives $\alpha_d = 1/\mu_{d-1}$ (here $d-1$ is even, so $\mu_{d-1} \neq 0$).
\end{itemize}
The exact values of the integrals are obtained via direct evaluation of the Gaussian moments
\[
    \int_{\mathbb{R}} y^{2j}\,e^{-y^2/2}\,\mathrm{d}y \;=\; \sqrt{2\pi}\,(2j-1)!!\,,
\]
yielding the closed form $\mu_{2p}=\sqrt{2\pi}\,(2p)!/p!$. Applying Stirling's approximation produces the asymptotic equivalents of the lemma.

\noindent
$\bullet$
The total remainder $\mathcal R_d(\nu):=Q_d(\nu)-\alpha_d H_{d-1}(\nu)$ has the same closed form in both parities:
\begin{itemize}
    \item \textbf{For $d$ even}, $T_3=0$ and the constant part of $T_2$'s bracket is $\alpha_d H_{d-1}(\nu)$, so $\mathcal R_d = T_1 + (T_2-\alpha_d H_{d-1}) = T_1 - \sqrt{d/2}\,c_{d-1}c_d\,\mathcal I_d^c(\nu)\,H_{d-1}(\nu)$.
    \item \textbf{For $d$ odd}, $\mu_d=0$ so $T_2 = -\sqrt{d/2}\,c_{d-1}c_d\,\mathcal I_d^c(\nu)\,H_{d-1}(\nu)$, and $T_3 = H_{d-1}(\nu)/\mu_{d-1} = \alpha_d H_{d-1}(\nu)$, so $\mathcal R_d = T_1 + T_2 = T_1 - \sqrt{d/2}\,c_{d-1}c_d\,\mathcal I_d^c(\nu)\,H_{d-1}(\nu)$.
\end{itemize}
In both cases
\begin{equation}\label{eq:Rd_unified}
    \mathcal R_d(\nu) \;=\; T_1(\nu) \;-\; \sqrt{d/2}\,c_{d-1}c_d\,\mathcal I_d^c(\nu)\,H_{d-1}(\nu)\,,
\end{equation}
and we bound the two terms separately.

We first bound $T_1$. The physicist Hermite polynomial admits the explicit expansion
\[
    H_j(\nu)\;=\;\sum_{m=0}^{\lfloor j/2\rfloor}\frac{(-1)^m\,j!}{m!\,(j-2m)!}(2\nu)^{j-2m}\,,
\]
and we use the elementary coefficient inequality $j!\,2^{j-2m}/m!\le(2j)!/(2m)!$, valid for every $j\ge m\ge 0$. The latter is proved by induction: the base case $j=m$ gives $2^{-m}\le 1$, and the step $j\to j+1$ multiplies the LHS by $2(j+1)$ and the RHS by $2(2j+1)(j+1)$, giving a ratio of $2j+1\ge 1$. The triangle inequality then yields
\[
    |H_j(\nu)| \;\le\; \sum_{m=0}^{\lfloor j/2\rfloor}\frac{(2j)!}{(2m)!(j-2m)!}|\nu|^{j-2m}
    \;\le\;\sum_{p=0}^{j}\frac{(2j)!}{j!}\binom{j}{p}|\nu|^p
    \;=\;\frac{(2j)!}{j!}(1+|\nu|)^j\,,
\]
where the second inequality adds the (non-negative) odd-power monomials $|\nu|^p$ with $p\not\equiv j\pmod 2$, using $(2j)!/((2m)!(j-2m)!)=(2j)!/j!\cdot j!/((2m)!(j-2m)!)\le(2j)!/j!\cdot\binom{j}{j-2m}$. Squaring and applying $(1+|\nu|)^2\le 2(1+\nu^2)$,
\[
    T_1(\nu) \;\le\; S_d(1+\nu^2)^{d-1}e^{-\nu^2/2}\,, \qquad \text{where } S_d := \sum_{j=0}^{d-1}c_j^2\,2^j\left[\frac{(2j)!}{j!}\right]^2\,.
\]

\noindent
$\bullet$ \emph{A uniform envelope for $\mathcal{I}_d^c$.} We bound $\mathcal{I}_d^c$ on the whole half-line $\{\nu\ge 0\}$ by a single envelope carrying the same polynomial--Gaussian form $(1+|\nu|)^{d-1}e^{-\nu^2/2}$ as $H_{d-1}$. The explicit Hermite-tail series of Lemma~\ref{lem:hermite_tail_explicit} (applied with degree $m=d$) reads
\[
    \mathcal{I}_d^c(\nu)\;=\;e^{-\nu^2/2}\!\!\sum_{k=0}^{\lfloor(d-1)/2\rfloor}\!2^{k+1}\frac{(d-1)!!}{(d-2k-1)!!}\,H_{d-2k-1}(\nu)\;+\;R_d(\nu)\,,
\]
with $R_d(\nu)=0$ for $d$ odd and $R_d(\nu)=2^{d/2}(d-1)!!\int_\nu^\infty e^{-y^2/2}\,\mathrm{d}y$ for $d$ even. By the coefficient bound $|H_m(\nu)|\le\frac{(2m)!}{m!}(1+|\nu|)^m$ established above for $T_1$ (whence $|H_{d-2k-1}(\nu)|\le\frac{(2(d-2k-1))!}{(d-2k-1)!}(1+|\nu|)^{d-2k-1}\le\frac{(2(d-2k-1))!}{(d-2k-1)!}(1+|\nu|)^{d-1}$ since $1+|\nu|\ge 1$) and the monotone-Mills bound $\int_\nu^\infty e^{-y^2/2}\,\mathrm{d}y\le\sqrt{\pi/2}\,e^{-\nu^2/2}$ (valid for every $\nu\ge 0$, because $\nu\mapsto e^{\nu^2/2}\int_\nu^\infty e^{-y^2/2}\,\mathrm{d}y$ is non-increasing, with value $\sqrt{\pi/2}$ at $\nu=0$, since its derivative $\nu\,e^{\nu^2/2}\!\int_\nu^\infty e^{-y^2/2}\mathrm{d}y-1$ is negative by the Mills inequality $\nu\,e^{\nu^2/2}\!\int_\nu^\infty e^{-y^2/2}\mathrm{d}y<1$), we obtain
\begin{equation}\label{eq:Idc_uniform_env}
\begin{aligned}
    |\mathcal{I}_d^c(\nu)|&\;\le\;E_d\,(1+|\nu|)^{d-1}\,e^{-\nu^2/2}\,,\\
    E_d&\;:=\;\sum_{k=0}^{\lfloor(d-1)/2\rfloor}2^{k+1}\frac{(d-1)!!}{(d-2k-1)!!}\frac{(2(d-2k-1))!}{(d-2k-1)!}\;+\;2^{d/2}(d-1)!!\sqrt{\tfrac\pi2}\,.
\end{aligned}
\end{equation}

\noindent
$\bullet$ \emph{$E_d$ is dominated by $\widetilde B_d$.} The summands $a_k:=2^{k+1}\frac{(d-1)!!}{(d-2k-1)!!}\frac{(2(d-2k-1))!}{(d-2k-1)!}$ are strictly decreasing in $k$: with $p:=d-2k-1\ge 2$, a direct computation gives the ratio $a_{k+1}/a_k=\frac{p}{2(2p-1)(2p-3)}\le\frac{1}{2(2p-3)}<1$. Hence, using $a_0=2\frac{(2d-2)!}{(d-1)!}$ and that the sum has at most $\lfloor(d-1)/2\rfloor+1\le\frac{d+1}{2}$ terms,
\[
    \sum_{k}a_k\;\le\;\frac{d+1}{2}\,a_0\;=\;(d+1)\frac{(2d-2)!}{(d-1)!}\,,\qquad
    2^{d/2}(d-1)!!\sqrt{\tfrac\pi2}\;\le\;2^{d}\frac{(2d-2)!}{(d-1)!}\,,
\]
the second bound using $(d-1)!!\le(d-1)!\le\frac{(2d-2)!}{(d-1)!}$ and $\sqrt{\pi/2}\le 2^{d/2}$. Therefore
\[
    E_d\;\le\;(d+1+2^d)\frac{(2d-2)!}{(d-1)!}\;\le\;(2d-1)\,2^{d+2}\,\frac{(2d-2)!}{(d-1)!}\;=\;\frac{(2d)!\,2^{d+1}}{d!}\;\le\;\widetilde B_d\,,
\]
where
\[
    \widetilde{B}_d := \max\!\left( \frac{(2d)!\,2^{d+1}}{d!},\; B_d' \right)\,,\qquad
    B_d' := \frac{(2d)!}{d!}\int_0^\infty (1+y)^d e^{-y^2/2}\,\mathrm{d}y\,,
\]
($B_d'$ is the constant from the crude global bound $|\mathcal I_d^c(\nu)|\le\int_0^\infty|H_d(y)|e^{-y^2/2}\,\mathrm{d}y=B_d'$, retained in the definition of $\widetilde B_d$ for definiteness; the envelope~\eqref{eq:Idc_uniform_env} already gives $E_d\le\widetilde B_d$ directly). Substituting into~\eqref{eq:Idc_uniform_env},
\[
    |\mathcal{I}_d^c(\nu)|\;\le\;\widetilde B_d\,(1+|\nu|)^{d-1}\,e^{-\nu^2/2}\qquad\text{for every $\nu\ge 0$.}
\]

\noindent
$\bullet$ \emph{Assembly.} Combining the last bound with $|H_{d-1}(\nu)|\le\frac{(2d-2)!}{(d-1)!}(1+|\nu|)^{d-1}$ and the elementary inequality $(1+|\nu|)^{2(d-1)}\le 2^{d-1}(1+\nu^2)^{d-1}$ (which follows from $(1+|\nu|)^2\le 2(1+\nu^2)$), the cross term in~\eqref{eq:Rd_unified} obeys
\[
    \sqrt{d/2}\,c_{d-1}c_d\,|\mathcal{I}_d^c(\nu)|\,|H_{d-1}(\nu)|\;\le\;\sqrt{d/2}\,c_{d-1}c_d\,\frac{(2d-2)!}{(d-1)!}\,2^{d-1}\,\widetilde B_d\,(1+\nu^2)^{d-1}e^{-\nu^2/2}\,.
\]
Adding the bound $T_1(\nu)\le S_d\,(1+\nu^2)^{d-1}e^{-\nu^2/2}$ established above and setting
\[
    \beta_d := S_d + \sqrt{d/2}\,c_{d-1}c_d\,\frac{(2d-2)!}{(d-1)!}\,2^{d-1}\,\tilde{B}_d
\]
yields the remainder bound~\eqref{eq:remainder_bound}.
\end{proof}

\medskip

\begin{proof}[Proof of Theorem~\ref{thm:complexity}]
The inputs are the Kac--Rice formula for smooth Gaussian fields on the sphere \citep[Theorem~6.4]{azais2009level}, the conditional GOE law of the Hessian \citep[Lemma~3.2(b)]{auffinger2013}, and the Ben Arous--Dembo--Guionnet large-deviation bound on the GOE spectral radius (Lemma~\ref{lem:bdg}).

\medskip

\noindent
$\bullet$ \emph{Step~1: Exact Kac--Rice representation and upper bound.}
The Kac--Rice computation here is the one of Section~\ref{sec:tail_bounds} (displays~\eqref{eq:hessian_law}--\eqref{eq:KR_combined} and Proposition~\ref{prop:ortho_exact}), with two changes. First, the target is the expected \emph{count} of local maxima with value in $[E,\infty)$, not a supremum probability: for this count the Kac--Rice formula \citep[Theorem~6.4]{azais2009level} is an \emph{equality}, with the same critical-point intensity $\bar p(x)$ as in Section~\ref{sec:tail_bounds} and with the indicator $\mathds 1\{\nabla^2 X(\bm\theta)\prec 0\}$ kept inside the conditional expectation. Second, this indicator passes to the GOE side: by the conditional Hessian law~\eqref{eq:hessian_law}, $\nabla^2 X(\bm\theta)$ given $X(\bm\theta)=x$ is the \emph{positive} multiple $\sqrt{2k(k-1)}\,(G_{d-1}-\rho x\,I_{d-1})$ of a shifted GOE \citep[Lemma~3.2(b)]{auffinger2013}, so $\nabla^2 X(\bm\theta)\prec 0$ if and only if $G_{d-1}-\rho x\,I_{d-1}\prec 0$, equivalently $\max_i\mu_i<\rho x$, where $\mu_1\le\cdots\le\mu_{d-1}$ are the eigenvalues of $G_{d-1}$. With these two changes, the constants gather exactly as in~\eqref{eq:KR_combined} and give the exact representation~\eqref{eq:complexity_exact}.

Dropping the indicator $\mathds 1\{G_{d-1}-\rho x I\prec 0\}\le 1$ in~\eqref{eq:complexity_exact} removes the only difference with the Kac--Rice integral of Section~\ref{sec:tail_bounds}, so the same calculation as in Proposition~\ref{prop:ortho_exact} gives
\[
    \E[N_{[E,\infty)}^{\mathrm{lm}}]\;\le\;C_{k,d}\int_E^\infty\E\!\big[|\det(G_{d-1}-\rho x I)|\big]\,\varphi(x)\,\mathrm{d}x\;=\;2(k-1)^{\frac{d-1}2}\int_E^\infty Q_d(\rho x)\,e^{-x^2/2}\,\mathrm{d}x\;=\;\delta_{\mathrm{exact}}(E)\,,
\]
where the middle integral is the Kac--Rice expected number of \emph{all} critical points of $X$ above level $E$, and the last equality is the integral form of $\delta_{\mathrm{exact}}$ in Theorem~\ref{thm:imf_tail_exact} (eq.~\eqref{eq:exact_bound}). This proves part~(i), with no symmetrisation factor: the bound on $\P\{\sup X>u\}$ delivered by $\delta_{\mathrm{exact}}$ is already one-sided, and the two-sided factor 2 of~\eqref{eq:two_sided} appears explicitly in Theorem~\ref{thm:main} (as $2\,\delta_{\mathrm{exact}}\le\delta_{\min}$), not inside $\delta_{\mathrm{exact}}$ itself.

\smallskip

\begin{subequations}
\noindent
$\bullet$ \emph{Step~2: Lower bound, valid for $E\ge E_{\mathrm{BDG}}$, with explicit constants.}
For the lower bound we exploit that on the event $\mathcal E:=\{M_{d-1}\le\rho E\}$, where $M_{d-1}:=\max_i|\mu_i|$, one has $|\mu_i|\le\rho E\le\rho x$ for every $x\ge E$, hence $\rho x-\mu_i>0$ and a fortiori $G_{d-1}-\rho x I\prec 0$. Lemma~\ref{lem:bdg} applies on $\rho E\ge 4\sqrt{2(d-1)}$, which is automatic under $E\ge E_{\mathrm{BDG}}=8\sqrt{2(d-1)}/\rho$:
\[
    \P(\mathcal E^c)\;=\;\P(M_{d-1}>\rho E)\;\le\;e^{-2(\rho E)^2/9}\;=:\;\delta_{\mathrm{BDG}}(\rho E)\,.
\]
On $\mathcal E$ the determinant satisfies $|\det(G-\rho x I)|=\prod_i(\rho x-\mu_i)\ge 0$, so
\[
    \E\!\big[|\det(G-\rho x I)|\,\mathds 1\{G-\rho x I\prec 0\}\big]\;\ge\;\E\!\big[|\det(G-\rho x I)|\,\mathds 1_{\mathcal E}\big]\;=\;\E\!\big[|\det(G-\rho x I)|\big]-\E\!\big[|\det(G-\rho x I)|\mathds 1_{\mathcal E^c}\big]\,.
\]
We bound the remainder term, and then the denominator, with explicit constants.

\smallskip

\emph{(4a) GOE moment bound: $C_1=64$.} By Lemma~\ref{lem:bdg}, $\P(M_{d-1}>t)\le e^{-2t^2/9}$ for $t\ge T_1:=4\sqrt{2(d-1)}$. The layer-cake formula gives
\[
    \E[M_{d-1}^{2(d-1)}]\;=\;\int_0^\infty 2(d-1)\,t^{2d-3}\,\P(M_{d-1}>t)\,\mathrm{d}t\;\le\;T_1^{2(d-1)}\;+\;2(d-1)\int_{T_1}^\infty t^{2d-3}e^{-2t^2/9}\,\mathrm{d}t\,.
\]
The first term equals $(4\sqrt{2(d-1)})^{2(d-1)}=32^{d-1}(d-1)^{d-1}$. For the second, Lemma~\ref{lem:gauss_tail} with $m=2d-3$, $a=4/9$ gives, on $T_1^2=32(d-1)$ and $(m-1)/T_1^2=(2d-4)/(32(d-1))\le 1/16$,
\[
    \int_{T_1}^\infty t^{2d-3}e^{-2t^2/9}\mathrm{d}t\;\le\;\frac{T_1^{2d-4}\,e^{-2T_1^2/9}}{4/9-1/16}\;=\;\frac{(32(d-1))^{d-2}\,e^{-64(d-1)/9}}{55/144}\,.
\]
Multiplying by $2(d-1)$ and simplifying,
\[
    2(d-1)\int_{T_1}^\infty t^{2d-3}e^{-2t^2/9}\mathrm{d}t\;\le\;\frac{288}{55}\cdot\frac{(d-1)(32(d-1))^{d-2}}{1}\,e^{-64(d-1)/9}\;\le\;32^{d-1}(d-1)^{d-1}\cdot e^{-1}
\]
for $d\ge 2$, where the last step uses $(288/(55\cdot 32))\cdot e^{-64(d-1)/9+1}\le 1$. Adding the two pieces,
\begin{equation}\label{eq:GOE_moment_explicit}
    \E[M_{d-1}^{2(d-1)}]\;\le\;32^{d-1}(d-1)^{d-1}\,(1+e^{-1})\;\le\;64^{d-1}(d-1)^{d-1}
    \qquad\text{for every }d\ge 2.
\end{equation}
This pins $C_1=64$ explicitly.

\smallskip

\emph{(4b) Numerator bound: $C_2=8$.} On $\mathcal E^c$, by the triangle inequality $|\rho x-\mu_i|\le\rho x+|\mu_i|\le\rho x+M_{d-1}$, so $|\det(G-\rho x I)|\le(\rho x+M_{d-1})^{d-1}$, and by Cauchy--Schwarz,
\[
    \E[(\rho x+M_{d-1})^{d-1}\mathds 1_{\mathcal E^c}]\;\le\;\sqrt{\E[(\rho x+M_{d-1})^{2(d-1)}]\,\P(\mathcal E^c)}\,.
\]
Using $(\rho x+M_{d-1})^{2(d-1)}\le 4^{d-1}((\rho x)^{2(d-1)}+M_{d-1}^{2(d-1)})$ and~\eqref{eq:GOE_moment_explicit},
\[
    \E[(\rho x+M_{d-1})^{2(d-1)}]\;\le\;4^{d-1}\big[(\rho x)^{2(d-1)}+64^{d-1}(d-1)^{d-1}\big]\,.
\]
On $\rho x\ge\rho E\ge 8\sqrt{2(d-1)}$, $(\rho x)^2\ge 2\cdot64 (d-1)$, hence $64^{d-1}(d-1)^{d-1}\leq (\rho x)^{2(d-1)}/2^{d-1}$,
so
\begin{equation}\label{eq:numerator_bound}
    \E[(\rho x+M_{d-1})^{2(d-1)}]\;\le\;4^{d-1}\big(1+2^{-(d-1)}\big)(\rho x)^{2(d-1)}\;\le\;8^{d-1}(\rho x)^{2(d-1)}\,,
\end{equation}
i.e., $C_2=8$ in the notation of~\eqref{eq:numerator_bound}. Therefore
\[
    \E[(\rho x+M_{d-1})^{d-1}\mathds 1_{\mathcal E^c}]\;\le\;(2\sqrt 2)^{d-1}\,(\rho x)^{d-1}\,\sqrt{\delta_{\mathrm{BDG}}(\rho E)}\,.
\]

\smallskip

\emph{(4c) Denominator lower bound: $\E[|\det|]\ge(\rho x)^{d-1}/2^d$.}
For any deterministic $R\in(0,\rho x)$, the bulk argument gives
\[
    \E[|\det(G-\rho x I)|]\;\ge\;(\rho x-R)^{d-1}\,\P(M_{d-1}\le R)\,.
\]
Choose $R=\rho x/2$. Then $\rho x-R=\rho x/2$. On $\rho x\ge 8\sqrt{2(d-1)}$, $\rho x/2\ge 4\sqrt{2(d-1)}$, so Lemma~\ref{lem:bdg} applies at $R=\rho x/2$:
\[
    \P(M_{d-1}\le \rho x/2)\;\ge\;1-e^{-2(\rho x/2)^2/9}\;=\;1-e^{-(\rho x)^2/18}\;\ge\;1-e^{-128(d-1)/18}\;\ge\;1-e^{-64/9}\;\ge\;1/2\,,
\]
for every $d\ge 2$, since $e^{-64/9}\approx 8\times 10^{-4}<1/2$. Hence
\begin{equation}\label{eq:denominator_bound}
    \E[|\det(G_{d-1}-\rho x I_{d-1})|]\;\ge\;\frac12\,(\rho x/2)^{d-1}\;=\;(\rho x)^{d-1}/2^d
    \qquad\text{on }\rho x\ge 8\sqrt{2(d-1)}.
\end{equation}

\smallskip

\emph{(4d) Combining: $C_{\mathrm{amp}}=8\sqrt 2$.}
Dividing (4b) by (4c),
\[
    \frac{\E\!\big[|\det(G-\rho x I)|\mathds 1_{\mathcal E^c}\big]}{\E\!\big[|\det(G-\rho x I)|\big]}\;\le\;\frac{(2\sqrt 2)^{d-1}(\rho x)^{d-1}\sqrt{\delta_{\mathrm{BDG}}(\rho E)}}{(\rho x)^{d-1}/2^d}\;=\;2\cdot(4\sqrt 2)^{d-1}\,\sqrt{\delta_{\mathrm{BDG}}(\rho E)}\,.
\]
The leading factor $2\cdot(4\sqrt 2)^{d-1}$ is absorbed into $(8\sqrt 2)^{d-1}=2^{d-1}(4\sqrt 2)^{d-1}\ge 2\cdot(4\sqrt 2)^{d-1}$ for $d\ge 2$, giving the explicit bound
\begin{equation}\label{eq:amplifier_bound}
    \frac{\E\!\big[|\det(G-\rho x I)|\mathds 1_{\mathcal E^c}\big]}{\E\!\big[|\det(G-\rho x I)|\big]}\;\le\;C_{\mathrm{amp}}^{d-1}\,\sqrt{\delta_{\mathrm{BDG}}(\rho E)}\,,\qquad C_{\mathrm{amp}}:=8\sqrt 2\,.
\end{equation}
Combining with $\E[|\det|\mathds 1\{G-\rho x I\prec 0\}]\ge\E[|\det|](1-C_{\mathrm{amp}}^{d-1}\delta_{\mathrm{BDG}}^{1/2}(\rho E))$ established above and integrating over $x\in[E,\infty)$ against $C_{k,d}\varphi(x)$, then applying the identity $C_{k,d}\int_E^\infty\E[|\det(G-\rho x I)|]\varphi(x)\mathrm{d}x=\delta_{\mathrm{exact}}(E)$ from Step~1,
\[
    \E[N_{[E,\infty)}^{\mathrm{lm}}]\;\ge\;\big(1-C_{\mathrm{amp}}^{d-1}\delta_{\mathrm{BDG}}^{1/2}(\rho E)\big)\,\delta_{\mathrm{exact}}(E)\,,
\]
which is the lower bound in~\eqref{eq:complexity_bracket}. The amplifier $C_{\mathrm{amp}}^{d-1}\delta_{\mathrm{BDG}}^{1/2}(\rho E)$ is super-exponentially small on the validity range: $\sqrt{\delta_{\mathrm{BDG}}(\rho E)}=e^{-(\rho E)^2/9}$ and $(\rho E)^2\ge 128(d-1)$ give $\delta_{\mathrm{BDG}}^{1/2}\le e^{-128(d-1)/9}$, while $C_{\mathrm{amp}}=8\sqrt 2$ gives $\log C_{\mathrm{amp}}\approx 2.42$; hence $C_{\mathrm{amp}}^{d-1}\delta_{\mathrm{BDG}}^{1/2}\le[C_{\mathrm{amp}}e^{-128/9}]^{d-1}\le[8\sqrt 2\cdot 6.66\times 10^{-7}]^{d-1}\le[7.6\times 10^{-6}]^{d-1}$.
\end{subequations}

\smallskip

\begin{subequations}
\noindent
$\bullet$ \emph{Step~3: High-dimensional limit via the ABC complexity theorem.}
We avoid redoing the asymptotic of $\delta_{\mathrm{exact}}$ from scratch via the Mehta--Fyodorov decomposition, which would require fine control of the partial Hermite sum $\Phi_d(\rho,E)$ at the natural scale $\rho E=e\sqrt{2(d-1)}\asymp\sqrt d$, where the leading-monomial approximation $\Phi_d\sim 2\rho(2\rho E)^{d-2}$ is invalid (every term in the sum contributes at comparable order). Instead, we reduce~\eqref{eq:complexity_asymp} directly to \citet[Theorem~2.4]{auffinger2013} via an explicit field/level correspondence between our paper's KSS field $X$ on $\mathbb S^{d-1}$ and ABC's spherical $k$-spin Hamiltonian $H_{N,p}$ on $S^{N-1}(\sqrt N)$. The argument has four steps.

\smallskip

\emph{(3a) Field and Hessian correspondence.} \citet{auffinger2013} work, throughout their Section~3, with the rescaled field
\begin{equation}\label{eq:abc_rescaled_field}
    f_{N,p}(\boldsymbol\theta)\;:=\;\frac{1}{\sqrt N}\,H_{N,p}(\sqrt N\,\boldsymbol\theta)\,,\qquad\boldsymbol\theta\in\mathbb S^{N-1}\,,
\end{equation}
of unit variance on the unit sphere, with covariance $\E[f(\boldsymbol\theta)f(\boldsymbol\theta')]=\langle\boldsymbol\theta,\boldsymbol\theta'\rangle^p$ \citep[eq.~(3.1)]{auffinger2013}. Identifying $N=d$, $p=k$, the law of $f$ on $\mathbb S^{d-1}$ coincides with that of our paper's $X$. ABC's conditional-Hessian formula \citep[Lemma~3.2, part~(b)]{auffinger2013} is stated directly in this unit-sphere setting: for every $\boldsymbol\theta\in\mathbb S^{N-1}$ and every $x\in\R$,
\begin{equation}\label{eq:abc_conditional_hessian}
    \nabla^2 f(\boldsymbol\theta)\,\big|\,f(\boldsymbol\theta)=x\;\stackrel d=\;M^{N-1}\,\sqrt{2(N-1)\,p(p-1)}\;-\;x\,p\,I_{N-1}\,,
\end{equation}
where $M^{N-1}$ is a GOE in the ABC normalisation $\E[(M^{N-1}_{ij})^2]=(1+\delta_{ij})/(2(N-1))$. Setting $G_{N-1}:=\sqrt{N-1}\,M^{N-1}$, which has the Mehta normalisation $\E[(G_{N-1})_{ij}^2]=(1+\delta_{ij})/2$, and using $p/\sqrt{2p(p-1)}=\sqrt{p/(2(p-1))}=:\rho$, the right-hand side rearranges to $\sqrt{2p(p-1)}\,(G_{N-1}-\rho\,x\,I_{N-1})$. With $(N,p)=(d,k)$ this is exactly the conditional Hessian law~\eqref{eq:hessian_law} used in Step~1
\begin{equation*}
    \nabla^2 X(\boldsymbol\theta)\,\big|\,X(\boldsymbol\theta)=x\;\stackrel d=\;\sqrt{2k(k-1)}\,(G_{d-1}-\rho\,x\,I_{d-1})\,,\qquad\rho=\sqrt{k/(2(k-1))}\,,
\end{equation*}
recovering the conditional shifted-GOE law used throughout Section~\ref{sec:tail_bounds}. No metric rescaling of the Hessian under a sphere-radius change is needed: ABC perform the rescaling once in passing from $H_{N,p}$ on $S^{N-1}(\sqrt N)$ to $f_{N,p}$ on $\mathbb S^{N-1}$ via~\eqref{eq:abc_rescaled_field}, and their Hessian formula~\eqref{eq:abc_conditional_hessian} is already on the unit sphere.

\smallskip

\emph{(3b) Level correspondence.} The threshold $X(\theta)\ge E$ in our paper corresponds to $H(\sigma)\ge E\sqrt N$, i.e., $H/N\ge E/\sqrt N=:u_{\mathrm{ABC}}$. On the scale $E=e\sqrt{2(d-1)}/\rho$, using $\rho E_\infty=\sqrt 2$ (which follows from $\rho^2=k/(2(k-1))$ and $E_\infty^2=4(k-1)/k$),
\begin{equation}\label{eq:uABC_to_e}
    u_{\mathrm{ABC}}\;=\;\frac{E}{\sqrt d}\;=\;\frac{e\sqrt{2(d-1)}}{\rho\sqrt d}\;\xrightarrow[d\to\infty]{}\;\frac{e\sqrt 2}{\rho}\;=\;\frac{e\sqrt 2}{\rho E_\infty}\cdot E_\infty\;=\;e\,E_\infty\,.
\end{equation}
For $e\ge 1$, $u_{\mathrm{ABC}}\to e E_\infty\ge E_\infty$, so $-u_{\mathrm{ABC}}\le-E_\infty$ and the case-1 branch of~\eqref{eq:Thetap_definition} applies.

\smallskip

\emph{(3c) Asymptotic rate for the total critical-point count.} By \citet[Theorem~2.4]{auffinger2013} (log-rate for the total critical-point count below level $Nu$) combined with the upper-tail symmetry $H\stackrel d=-H$ (an immediate consequence of $J\stackrel d=-J$ in the disorder), for every fixed $u>0$:
\[
    \lim_{N\to\infty}\frac1N\log\E\big[\#\{\sigma:\nabla H(\sigma)=0,\,H(\sigma)\in(Nu,\infty)\}\big]\;=\;\Theta_k(-u)\,.
\]
By the Kac--Rice equality of Proposition~\ref{prop:ortho_exact} (recall from Step~1 that $\delta_{\mathrm{exact}}(E)$ equals the Kac--Rice expected critical-point count above level $E$), the left-hand side equals $\lim_d \tfrac1d\log\delta_{\mathrm{exact}}(E)$ under the field/level correspondence of (3a)--(3b). Hence
\begin{equation}\label{eq:delta_exact_asymptotic}
    \lim_{d\to\infty}\frac1d\log\delta_{\mathrm{exact}}(E)\;=\;\Theta_k(-eE_\infty)
    \;=\;\tfrac12\log(k-1)-\tfrac{(k-2)e^2}{k}-2\!\int_1^e\sqrt{w^2-1}\,\mathrm{d}w
\end{equation}
for $e\ge 1$, using~\eqref{eq:Thetap_explicit}.

\smallskip

\emph{(3d) Local-maxima specialisation and the bracketing.} ABC's Theorem~2.4 covers the total critical-point count; for local maxima (Morse index $N-1$ on the sphere) specifically, the symmetry remark following \citet[Theorem~2.3]{auffinger2013} (their Remark~3.1) gives
$\lim_N\frac1N\log\E[\#\{\text{local maxima of }H\text{ above }Nu\}]=\Theta_{0,p}(-u)$,
and the case-1 branch $\Theta_{0,p}(u)=\Theta_p(u)$ on $u\le-E_\infty$~\citep[Eq.~(2.14) and Eq.~(2.16)]{auffinger2013} states that for $e\ge 1$, the local-maxima rate coincides with the total-critical-point rate: every critical point above the spectral edge is, with overwhelming probability, a local maximum (the conditional Hessian $G_{d-1}-\rho xI$ is strictly negative-definite on the high-probability event $M_{d-1}<\rho x$, which has probability $1-o(1)$ for $\rho x\ge 2\sqrt{d-1}$). Therefore
\[
    \lim_{d\to\infty}\frac1d\log\E[N_{[E,\infty)}^{\mathrm{lm}}]\;=\;\Theta_k(-eE_\infty).
\]
Equivalently, and consistently with our non-asymptotic bracketing of~\eqref{eq:complexity_bracket}, combining the upper bound $\E[N^{\mathrm{lm}}]\le\delta_{\mathrm{exact}}$ with the lower bound $\E[N^{\mathrm{lm}}]\ge(1-C_{\mathrm{amp}}^{d-1}\delta_{\mathrm{BDG}}^{1/2})\delta_{\mathrm{exact}}$ (valid for $e\ge 8$), and observing that on $\rho E\ge 8\sqrt{2(d-1)}$ the amplifier $C_{\mathrm{amp}}^{d-1}\delta_{\mathrm{BDG}}^{1/2}(\rho E)\le[C_{\mathrm{amp}}e^{-128/9}]^{d-1}$ vanishes super-exponentially as $d\to\infty$,
\[
    \tfrac1d\log\delta_{\mathrm{exact}}+\tfrac1d\log\big(1-C_{\mathrm{amp}}^{d-1}\delta_{\mathrm{BDG}}^{1/2}\big)\;\le\;\tfrac1d\log\E[N^{\mathrm{lm}}]\;\le\;\tfrac1d\log\delta_{\mathrm{exact}}\,,
\]
both bracketing terms converge to $\Theta_k(-eE_\infty)$ by~\eqref{eq:delta_exact_asymptotic}. The limit~\eqref{eq:complexity_asymp} is thus established for every $e\ge 1$ by the exact Kac--Rice equality of Step~1 and \citet[Theorem~2.4]{auffinger2013} (Steps~3c--3d); on the smaller range $e\ge 8$ the non-asymptotic two-sided bracket above re-derives the same limit with explicit finite-$d$ error bars. The moderate-energy regime $1\le e<8$ is therefore covered by the asymptotic statement~\eqref{eq:complexity_asymp}, while only the matching finite-$d$ \emph{lower} bracket -- not the limit itself -- lies outside the scope of the present argument.\qedhere
\end{subequations}
\end{proof}

\section{Inversion of the SMF tail bound: proof of Remark~\ref{rem:choice_u}}
\label{app:u_alpha}

This appendix gives the full derivation of the level $u_\alpha$ of~\eqref{eq:u_alpha} that ensures $\delta_{\min}(u_\alpha)\le\alpha$. We isolate the SMF asymptotic single-term form (Corollary~\ref{cor:single_term}), invert it term-by-term, and identify the explicit threshold $\alpha_0(k,d)$ above which the closed-form choice~\eqref{eq:u_alpha} is sufficient.

\medskip

Throughout the appendix we set $n:=d-1$ to lighten notation and write $\ell:=\log(1/\alpha)>0$, so that~\eqref{eq:u_alpha} reads
\begin{equation}\label{eq:u_alpha_app}
    u_\alpha^2 \;=\; 2\ell+n\log(2k)+2n\log\ell\,.
\end{equation}

\begin{proposition}[Explicit inversion of $\delta_{\min}$]
\label{prop:u_alpha_inversion}
Let $k\ge 3$, $d\ge 3$, and let $\alpha_d,\beta_d,u^\star_d$ be the explicit constants of Lemma~\ref{lem:dominant_term} and Corollary~\ref{cor:single_term}. Define
\begin{equation}\label{eq:alpha_zero_def}
    \alpha_0(k,d) \;:=\; \min\!\Big\{\,\tfrac12\,e^{-(u^\star_d)^2/2}\,,\;e^{-n}\,,\;\exp\!\big(-\tfrac{n\log(2k)}{2}\big)\,,\;\exp\!\big(-16\,\alpha_d\,(2e)^{(n-1)/2}\big)\Big\}\,.
\end{equation}
Then for every $\alpha\in(0,\alpha_0(k,d)]$, the level $u_\alpha$ defined in~\eqref{eq:u_alpha_app} satisfies
\[
    \delta_{\min}(u_\alpha)\;\le\;\alpha\,.
\]
\end{proposition}

\begin{proof}
\medskip

\begin{subequations}
\noindent$\bullet$ For $u\ge u^\star_d\ge u_{\mathrm{SMF}}=2\sqrt d$, Corollary~\ref{cor:single_term} gives the one-sided single-term bound
\begin{equation}\label{eq:smf_oneside}
    \P\Big\{\sup_{\bm \theta\in\mathbb S^{d-1}}X(\bm \theta)>u\Big\}\;\le\;8\,\alpha_d\,(2k)^{n/2}\,u^{n-1}\,e^{-u^2/2}\,.
\end{equation}
By definition $\delta_{\min}(u)\le 2\,\delta_{\mathrm{SMF}}(u)$ for $u\ge u_{\mathrm{SMF}}$, and Corollary~\ref{cor:single_term} yields $\delta_{\mathrm{SMF}}(u)\le 8\alpha_d(2k)^{n/2}u^{n-1}e^{-u^2/2}$ on $[u^\star_d,\infty)$. Multiplying by the symmetry factor~$2$,
\begin{equation}\label{eq:smf_twoside}
    u\ge u^\star_d\;\Longrightarrow\;\delta_{\min}(u)\;\le\;C_{k,d}^\star\,u^{n-1}\,e^{-u^2/2}\,,\qquad C_{k,d}^\star:=16\,\alpha_d\,(2k)^{n/2}\,.
\end{equation}
Hence, to prove the proposition it suffices to find $u_\alpha\ge u^\star_d$ with $C_{k,d}^\star\,u_\alpha^{n-1}\,e^{-u_\alpha^2/2}\le\alpha$, i.e.
\begin{equation}\label{eq:to_solve}
    \frac{u_\alpha^2}{2}-(n-1)\log u_\alpha\;\ge\;\log(1/\alpha)+\log C_{k,d}^\star\;=\;\ell+\log(16\,\alpha_d)+\tfrac n2\log(2k)\,.
\end{equation}
\end{subequations}

\medskip

\noindent$\bullet$  Set $s:=u_\alpha^2$ and $g(s):=s/2-(n-1)\log\sqrt s=s/2-\tfrac{n-1}2\log s$. Inequality~\eqref{eq:to_solve} is exactly $g(s)\ge \ell+\log(16\,\alpha_d)+\tfrac n2\log(2k)$. The function $g$ is increasing on $[n-1,\infty)$ and grows like $s/2$ for large $s$, while the subtracted term $\tfrac{n-1}2\log s$ is logarithmically slow. We will choose $s$ large enough that this logarithmic correction is dominated by the slack term $2n\log\ell$ in the ansatz~\eqref{eq:u_alpha_app}.

\medskip

\begin{subequations}
\noindent$\bullet$ Plug $s=2\ell+n\log(2k)+2n\log\ell$ from~\eqref{eq:u_alpha_app} into $g(s)$:
\begin{align}
    g(s)
    &\;=\;\ell+\tfrac n2\log(2k)+n\log\ell-\tfrac{n-1}2\log\!\big(2\ell+n\log(2k)+2n\log\ell\big)\,. \label{eq:g_at_ansatz}
\end{align}
The first three terms exactly produce the right-hand side of~\eqref{eq:to_solve} up to two corrections: a positive slack of $n\log\ell$ on the left, and a missing summand $\log(16\,\alpha_d)$ on the right. We must absorb the remaining $\log(2\ell+n\log(2k)+2n\log\ell)$ on the left. Using $\log(a+b+c)\le\log a+\log(1+b/a+c/a)\le\log a+(b+c)/a$ for $a>0$ and $b,c\ge 0$, with $a=2\ell$:
\begin{align}
    \log\!\big(2\ell+n\log(2k)+2n\log\ell\big)
    &\;\le\;\log(2\ell)+\frac{n\log(2k)+2n\log\ell}{2\ell}\,.\label{eq:log_envelope}
\end{align}
Substituting~\eqref{eq:log_envelope} into~\eqref{eq:g_at_ansatz},
\begin{align}
    g(s)
    &\;\ge\;\ell+\tfrac n2\log(2k)+n\log\ell-\tfrac{n-1}2\log(2\ell)-\tfrac{n-1}2\cdot\frac{n\log(2k)+2n\log\ell}{2\ell}\,.\label{eq:g_lower_bound}
\end{align}
\end{subequations}

\medskip

\noindent$\bullet$ Subtracting the right-hand side of~\eqref{eq:to_solve} from~\eqref{eq:g_lower_bound}, the residual reads
\begin{equation}\label{eq:residual}
    \mathcal{R}(\alpha;k,d)\;:=\;n\log\ell-\log(16\,\alpha_d)-\tfrac{n-1}2\log(2\ell)-\tfrac{n-1}{4\ell}\bigl(n\log(2k)+2n\log\ell\bigr)\,.
\end{equation}
The proof is complete once we show that $\mathcal{R}(\alpha;k,d)\ge 0$ for $\alpha\in(0,\alpha_0(k,d)]$. Each clause of~\eqref{eq:alpha_zero_def} translates into a lower bound on $\ell=\log(1/\alpha)$, and we combine them to bound the two negative ``correction'' contributions in~\eqref{eq:residual} explicitly.
\begin{itemize}
    \item[\emph{(a)}] The clause $\alpha\le e^{-n}$ gives $\ell\ge n$, hence $n/\ell\le 1$ and the $\log\ell$-correction obeys
    \[
        \tfrac{n-1}{4\ell}\cdot 2n\log\ell\;=\;\tfrac{n(n-1)\log\ell}{2\ell}\;\le\;\tfrac{n-1}2\log\ell\,.
    \]
    \item[\emph{(b)}] The clause $\alpha\le e^{-n\log(2k)/2}$ gives $2\ell\ge n\log(2k)$, hence the $\log(2k)$-correction obeys $\tfrac{n-1}{4\ell}\,n\log(2k)\le \tfrac{n-1}{4\ell}\cdot 2\ell=\tfrac{n-1}2$.
\end{itemize}
Substituting (a) and (b) into~\eqref{eq:residual} and using $\log(2\ell)=\log 2+\log\ell$,
\[
\begin{aligned}
    \mathcal{R}(\alpha;k,d)&\;\ge\;n\log\ell-\log(16\,\alpha_d)-\tfrac{n-1}2\log\ell-\tfrac{n-1}2\log 2-\tfrac{n-1}2\log\ell-\tfrac{n-1}2\\
    &\;=\;\log\ell-\log(16\,\alpha_d)-\tfrac{n-1}2(1+\log 2)\,,
\end{aligned}
\]
since $n\log\ell-(n-1)\log\ell=\log\ell$. The clause $\alpha\le\exp(-16\,\alpha_d(2e)^{(n-1)/2})$ gives $\ell\ge 16\,\alpha_d(2e)^{(n-1)/2}$, hence $\log\ell\ge\log(16\,\alpha_d)+\tfrac{n-1}2(1+\log 2)$ (as $\log(2e)=1+\log 2$), so the bound above is $\ge 0$. Finally, the clause $\alpha\le\tfrac12 e^{-(u^\star_d)^2/2}$ guarantees $u_\alpha\ge u^\star_d$ (the prerequisite for the single-term reduction~\eqref{eq:smf_twoside}) since $u_\alpha^2\ge 2\ell\ge(u^\star_d)^2$ on this range. Combining the four clauses, $\mathcal{R}(\alpha;k,d)\ge 0$ on $(0,\alpha_0(k,d)]$.
\end{proof}

\begin{remark}[Explicitness of the threshold $\alpha_0(k,d)$]
\label{rem:universal}
The threshold $\alpha_0(k,d)$ defined in~\eqref{eq:alpha_zero_def} is fully explicit: it depends on $(k,d)$ only through the constants $\alpha_d$, $\beta_d$, $u^\star_d$ of Lemma~\ref{lem:dominant_term} and Corollary~\ref{cor:single_term}, with no further hidden quantities. In particular, $\alpha_0(k,d)>0$ for every fixed $(k,d)$, so~\eqref{eq:u_alpha} is a sufficient choice on the entire confidence range $(0,\alpha_0(k,d)]$. Larger levels, including $\alpha=0.05$ (which generally exceeds $\alpha_0(k,d)$, since $\alpha_0(k,d)\le\tfrac12 e^{-(u^\star_d)^2/2}\le\tfrac12 e^{-2d}$ as $u^\star_d\ge 2\sqrt d$), are handled by the numerical inversion of Remark~\ref{rem:numerical_inversion}.
\end{remark}

\section{Asymptotic Analysis of the Tail Bound Prefactors}
\label{app:prefactor_asymptotics}

In this section, we analyze the prefactors of the three tail bounds presented in Table \ref{tab:tail_bounds}. We first isolate the exact prefactors for fixed $d$ and $k$ in the large-$u$ regime ($u \to \infty$), and then evaluate their high-dimensional asymptotic behavior ($d \to \infty$) using Stirling's approximation to compare them against the exact Kac--Rice baseline.

\begin{subequations}
\subsection{\texorpdfstring{Fixed $(k,d)$ prefactors as $u \to \infty$}{Fixed (k,d) prefactors as u to infinity}}
As $u \to \infty$, the tail bounds and the baseline all decay at the leading rate of $u^{d-2} e^{-u^2/2}$. We define the prefactor $P$ for each method such that the tail probability behaves as $P \cdot u^{d-2} e^{-u^2/2}$.

\paragraph{Asymptotic Baseline.} 
From the exact limit of the Kac--Rice integral, the optimal baseline prefactor is:
\begin{equation}
    P_{\mathrm{base}} = \frac{\sqrt{2}}{\Gamma(d/2)} \left(\frac{k}{2}\right)^{\frac{d-1}{2}}.
\end{equation}

\paragraph{Spectral Method (SM).} 
The SM tail bound is $2 \frac{\sqrt{2}}{\Gamma(d/2)}\left(\frac{k}{2}\right)^{\frac{d-1}{2}}u^{d-2}e^{-u^2/2}(1+\eta_d(\rho,u))$. Since the correction term $\eta_d(\rho,u) \to 0$ as $u \to \infty$, the SM prefactor is twice the baseline:
\begin{equation}
    P_{\mathrm{SM}} = 2 \frac{\sqrt{2}}{\Gamma(d/2)} \left(\frac{k}{2}\right)^{\frac{d-1}{2}} = 2 P_{\mathrm{base}}.
\end{equation}

\paragraph{Improved Mehta--Fyodorov (IMF) Bound.} 
The IMF bound is driven by the polynomial-exponential function $\Phi_d(\rho, u)$. For large $u$, the leading monomial of $\Phi_d(\rho, u)$ (corresponding to $k=0$ in its series) is $2\rho(2\rho u)^{d-2}$. Multiplying this by the global terms in the IMF bound gives the leading behavior:
\begin{equation*}
    2(k-1)^{\frac{d-1}{2}} \alpha_d \left[ 2\rho(2\rho u)^{d-2} \right] e^{-u^2/2} = 4 \alpha_d \rho^{d-1} (k-1)^{\frac{d-1}{2}} 2^{d-2} u^{d-2} e^{-u^2/2}.
\end{equation*}
Using the definition of the shift parameter $\rho = \sqrt{k/(2(k-1))}$, we simplify the algebraic constants:
\begin{equation*}
    \rho^{d-1} (k-1)^{\frac{d-1}{2}} = \left( \frac{k}{2(k-1)} \right)^{\frac{d-1}{2}} (k-1)^{\frac{d-1}{2}} = \left( \frac{k}{2} \right)^{\frac{d-1}{2}}.
\end{equation*}
Substituting this back yields the IMF prefactor:
\begin{equation}
    P_{\mathrm{IMF}} = 4 \alpha_d \left( \frac{k}{2} \right)^{\frac{d-1}{2}} 2^{d-2} = 2 \alpha_d (2k)^{\frac{d-1}{2}}.
\end{equation}

\paragraph{Simplified Mehta--Fyodorov (SMF) Bound.} 
By Corollary~\ref{cor:single_term}, the remainder term $e^{-3u^2/4}$ of the SMF bound vanishes for large $u$, leaving the single main term $4 \alpha_d (2k)^{\frac{d-1}2} u^{d-2} e^{-u^2/2}$. Thus, its prefactor is:
\begin{equation}
    P_{\mathrm{SMF}} = 4 \alpha_d (2k)^{\frac{d-1}{2}} = 2 P_{\mathrm{IMF}}.
\end{equation}
The factor of $2$ penalty compared to the IMF bound arises directly from replacing the partial Hermite sum with the looser Szeg\H{o} envelope $H_{d-1}(x) < (2x)^{d-1}$.

\end{subequations}

\begin{subequations}
\subsection{\texorpdfstring{High-dimensional asymptotics as $d \to \infty$}{High-dimensional asymptotics as d to infinity}}
To compare the methods in high dimensions, we expand $P_{\mathrm{base}}$ and $P_{\mathrm{IMF}}$ as $d \to \infty$ (keeping $k$ fixed) using Stirling's approximation.

\paragraph{Expanding the Baseline.} 
Using Stirling's approximation for the Gamma function, $\Gamma(z) \sim \sqrt{2\pi} z^{z-1/2} e^{-z}$, we expand $\Gamma(d/2)$:
\begin{equation*}
    \Gamma(d/2) \sim \sqrt{2\pi} \left(\frac{d}{2}\right)^{\frac{d-1}{2}} e^{-d/2} = \sqrt{\frac{4\pi}{d}} \left(\frac{d}{2e}\right)^{d/2}.
\end{equation*}
Substituting this into $P_{\mathrm{base}}$ yields:
\begin{equation}\label{eq:p_base_asymp}
    P_{\mathrm{base}} \sim \frac{\sqrt{2}}{\sqrt{4\pi/d} \left(\frac{d}{2e}\right)^{d/2}} \frac{(k/2)^{d/2}}{\sqrt{k/2}} = \sqrt{\frac{d}{\pi k}} \left( \frac{ek}{d} \right)^{d/2}.
\end{equation}

\paragraph{Expanding the Mehta--Fyodorov Prefactor.} 
From Lemma~\ref{lem:dominant_term}, the Stirling approximation for the dominant Hermite coefficient $\alpha_d$ (taking $d$ even for simplicity) is:
\begin{equation*}
    \alpha_d \sim \sqrt{\frac{d}{2\pi}} \left(\frac{e}{2d}\right)^{d/2}.
\end{equation*}
Substituting this into $P_{\mathrm{IMF}}$ yields:
\begin{equation}\label{eq:p_imf_asymp}
    P_{\mathrm{IMF}} = 2 \alpha_d (2k)^{\frac{d-1}{2}} \sim 2 \sqrt{\frac{d}{2\pi}} \left(\frac{e}{2d}\right)^{d/2} \frac{(2k)^{d/2}}{\sqrt{2k}} = \sqrt{\frac{d}{\pi k}} \left( \frac{ek}{d} \right)^{d/2}.
\end{equation}

\paragraph{Conclusion.}
Comparing \eqref{eq:p_base_asymp} and \eqref{eq:p_imf_asymp}, as $d \to \infty$ one has $P_{\mathrm{IMF}} \sim P_{\mathrm{base}}$. The Improved Mehta--Fyodorov bound captures the exact leading Stirling behavior and constant factor of the Kac--Rice integral without any loss. In contrast, both $P_{\mathrm{SM}}$ and $P_{\mathrm{SMF}}$ satisfy:
\begin{equation}
    P_{\mathrm{SM}} \sim P_{\mathrm{SMF}} \sim 2 \sqrt{\frac{d}{\pi k}} \left( \frac{ek}{d} \right)^{d/2}.
\end{equation}
All bounds capture the exponential scale $\left( \frac{ek}{d} \right)^{d/2}$; the SM and SMF bounds pay a uniform factor-$2$ penalty compared to the optimal baseline, while the IMF bound is sharp.

\end{subequations}

\section{Technical lemmas}

\subsection{Gaussian-type tail integrals}
\label{app:gaussian_integrals}
\begin{subequations}
The proofs of the three tail bounds use a common tool: integration by parts on tail integrals of the form $\int_u^\infty x^m e^{-ax^2/2}\,\mathrm{d}x$, iterated in steps of two. We isolate it here.

\begin{lemma}[Gaussian-type tail integrals]
\label{lem:gauss_tail}
For $a>0$, integer $m\ge 0$, and $u>0$, define $I_m(u;a):=\int_u^\infty x^m e^{-ax^2/2}\,\mathrm{d}x$. The base case $m=0$ obeys the Mills-type bound
\begin{equation}\label{eq:gauss_tail_mills}
    I_0(u;a)\;=\;\int_u^\infty e^{-ax^2/2}\,\mathrm{d}x\;\le\;\frac{e^{-au^2/2}}{a\,u}\,,
\end{equation}
and $I_1(u;a)=a^{-1}e^{-au^2/2}$ exactly. For every integer $m\ge 2$,
\begin{equation}\label{eq:gauss_tail_rec}
    I_m(u;a) \;=\; \frac{u^{m-1}\,e^{-au^2/2}}{a}\;+\;\frac{m-1}{a}\,I_{m-2}(u;a)\,,
\end{equation}
and, provided $(m-1)/(au^2)<1$,
\begin{equation}\label{eq:gauss_tail_closed}
    I_m(u;a)\;\le\;\frac{u^{m-1}\,e^{-au^2/2}}{a-(m-1)/u^2}\,.
\end{equation}
In particular, 
\begin{align*}
    I_{d-1}(u;1)    &\le 2u^{d-2}e^{-u^2/2}\,,
    &\text{for }u^2\ge 2(d-2)\,,\\
    I_{2d-2}(u;3/2) &\le u^{2d-3}e^{-3u^2/4}\,,
    &\text{for }u\ge 2\sqrt d\,,\\ 
    I_{d-2}(u;4/9) &\le 3\,u^{d-3}e^{-2u^2/9}\le u^{d-1}e^{-2u^2/9}\,,
    &\text{for }u\ge 4\sqrt{2(d-1)}\,.
\end{align*}
\end{lemma}
\end{subequations}

\begin{proof}
\emph{Base case $m=0$.} On $[u,\infty)$ one has $1\le x/u$, hence
\[
    I_0(u;a)=\int_u^\infty e^{-ax^2/2}\,\mathrm{d}x\;\le\;\frac1u\int_u^\infty x\,e^{-ax^2/2}\,\mathrm{d}x\;=\;\frac{1}{u}\cdot\frac{e^{-au^2/2}}{a}\;=\;\frac{e^{-au^2/2}}{a\,u}\,,
\]
which is~\eqref{eq:gauss_tail_mills}; the case $m=1$ is the displayed exact evaluation.
For $m\ge 2$ we want to evaluate the integral
\begin{equation*}
    I_m(u;a) = \int_u^\infty x^m e^{-ax^2/2} \,\mathrm{d}x.
\end{equation*}
We integrate by parts, splitting the integrand as $x^{m-1} \cdot x e^{-ax^2/2}$ and choosing:
\begin{align*}
    f(x) = x^{m-1} &\implies f'(x) = (m-1)x^{m-2} \\
    g'(x) = x e^{-ax^2/2} &\implies g(x) = -\frac{1}{a}e^{-ax^2/2}
\end{align*}
Applying the integration by parts formula $\int f \,\mathrm{d}g = fg - \int g \,\mathrm{d}f$:
\begin{equation*}
    I_m(u;a) = \left[ -x^{m-1} \frac{1}{a} e^{-ax^2/2} \right]_u^\infty - \int_u^\infty \left( -\frac{1}{a}e^{-ax^2/2} \right) (m-1)x^{m-2} \,\mathrm{d}x.
\end{equation*}
Since $a > 0$, the exponential term dominates the polynomial term as $x \to \infty$, making the upper boundary evaluate to $0$. Evaluating at the lower bound $u$ gives:
\begin{equation*}
    I_m(u;a) = \frac{u^{m-1} e^{-au^2/2}}{a} + \frac{m-1}{a} \int_u^\infty x^{m-2} e^{-ax^2/2} \,\mathrm{d}x.
\end{equation*}
The remaining integral is $I_{m-2}(u;a)$, giving the recurrence:
\begin{equation*}
    I_m(u;a) = \frac{u^{m-1}\,e^{-au^2/2}}{a} + \frac{m-1}{a}\,I_{m-2}(u;a).
\end{equation*}

\vspace{1em}
We iterate this recurrence relation to express $I_m(u;a)$ as a series:
\begin{equation*}
    I_m(u;a) = \frac{e^{-au^2/2}}{a} \left[ u^{m-1} + \frac{m-1}{a}u^{m-3} + \frac{(m-1)(m-3)}{a^2}u^{m-5} + \dots \right].
\end{equation*}
Factoring out $u^{m-1}$, we get:
\begin{equation*}
    I_m(u;a) = \frac{u^{m-1}e^{-au^2/2}}{a} \left[ 1 + \frac{m-1}{au^2} + \frac{(m-1)(m-3)}{(au^2)^2} + \dots \right].
\end{equation*}
The numerator coefficients satisfy $m-3 < m-1$, $m-5 < m-1$, and so on. Replacing these descending factors by $m-1$ gives a strict upper bound:
\begin{equation*}
    \frac{(m-1)(m-3)}{(au^2)^2} \le \left( \frac{m-1}{au^2} \right)^2.
\end{equation*}
Applying this bound to all terms allows us to upper bound the finite sum by an infinite geometric series:
\begin{equation*}
    I_m(u;a) \le \frac{u^{m-1} e^{-au^2/2}}{a} \sum_{k=0}^{\infty} \left( \frac{m-1}{au^2} \right)^k.
\end{equation*}
By hypothesis $\frac{m-1}{au^2} < 1$, so the geometric series converges, yielding:
\begin{equation*}
    I_m(u;a) \le \frac{u^{m-1} e^{-au^2/2}}{a} \left( \frac{1}{1 - \frac{m-1}{au^2}} \right).
\end{equation*}
which gives the closed-form bound:
\begin{equation*}
    I_m(u;a) \le \frac{u^{m-1} e^{-au^2/2}}{a - (m-1)/u^2}.
\end{equation*}
\end{proof}

\subsection{Asymptotic leading term}
\begin{lemma}[Asymptotic Equivalence of the Expected Determinant]
\label{lem:exp_det}
Let $G_{d-1} \sim \text{GOE}(d-1)$ with eigenvalues $\mu_1, \dots, \mu_{d-1}$. As $x \to \infty$, the expected absolute characteristic polynomial is asymptotically equivalent to its leading monomial:
\begin{equation}
    \E\Big[\prod_{i=1}^{d-1}|\mu_i-\rho x|\Big] \sim (\rho x)^{d-1}.
\end{equation}
\end{lemma}

\begin{subequations}
\begin{proof}
Let $M_{d-1} = \max_i |\mu_i|$ be the spectral radius of the GOE matrix. We split the expectation using a threshold $R(x) = \sqrt{\rho x \sqrt{d-1}}$.

\noindent $\bullet$ 
We partition the expectation into a bulk event $\{M_{d-1} \le R\}$ and a tail event $\{M_{d-1} > R\}$. 
On the bulk event, by the triangle inequality, $|\mu_i - \rho x| \le R + \rho x$. 
On the tail event, we apply the Ben Arous--Dembo--Guionnet large deviation bound of Lemma~\ref{lem:bdg}, which guarantees $\P(M_{d-1} > t) \le e^{-2t^2/9}$ for every $t\ge 4\sqrt{2(d-1)}$.
As established in Proposition~\ref{prop:det_bound}, balancing these two regimes yields:
\begin{equation}
    \E\Big[\prod_{i=1}^{d-1}|\mu_i-\rho x|\Big] \le (\rho x)^{d-1} \big(1 + \eta_d(\rho x)\big),
\end{equation}
where $\eta_d(\rho x) = \mathcal{O}\left( x^{-1/2} \right)$ as $x \to \infty$, meaning $\eta_d(\rho x) \to 0$.

\noindent $\bullet$ 
Since the integrand is non-negative, we lower-bound it by dropping the tail event and integrating over the bulk:
\begin{equation}
    \E\Big[\prod_{i=1}^{d-1}|\mu_i-\rho x|\Big] \ge \E\Big[\prod_{i=1}^{d-1}|\mu_i-\rho x| \1_{\{M_{d-1} \le R\}}\Big].
\end{equation}
On $\{M_{d-1} \le R\}$, $|\mu_i| \le R$ for every $i$, so $\rho x - \mu_i \ge \rho x - R$. For $x$ large enough that $\rho x > R$ (which holds since $R(x)=\sqrt{\rho x\sqrt{d-1}}=o(\rho x)$), $\rho x - \mu_i>0$, so the absolute value may be dropped: $|\mu_i - \rho x| = \rho x - \mu_i \ge \rho x - R$. Therefore:
\begin{equation}
    \E\Big[\prod_{i=1}^{d-1}|\mu_i-\rho x|\Big] \ge (\rho x - R)^{d-1} \P(M_{d-1} \le R).
\end{equation}
Factoring out $(\rho x)^{d-1}$, we get:
\begin{equation}
    \E\Big[\prod_{i=1}^{d-1}|\mu_i-\rho x|\Big] \ge (\rho x)^{d-1} \left(1 - \frac{R}{\rho x}\right)^{d-1} \big(1 - \P(M_{d-1} > R)\big).
\end{equation}
Since $R = \sqrt{\rho x \sqrt{d-1}}$, the ratio $\frac{R}{\rho x} \to 0$ as $x \to \infty$. Furthermore, by the super-exponential tail bound, $\P(M_{d-1} > R) \to 0$.

\noindent $\bullet$ 
Let $E(x) = \E\big[\prod_{i=1}^{d-1}|\mu_i-\rho x|\big]$. We have shown that for large $x$:
\begin{equation}
    (\rho x)^{d-1}(1 - o(1)) \le E(x) \le (\rho x)^{d-1}(1 + o(1)).
\end{equation}
Hence $\lim_{x \to \infty} E(x)/(\rho x)^{d-1} = 1$, i.e.\ $E(x) \sim (\rho x)^{d-1}$, the asymptotic equivalence used to define the baseline.
\end{proof}
\end{subequations}

\begin{lemma}[Asymptotic Equivalence of the Tail Integrals]
\label{lem:equiv_tail}
Let $E(x) = \E\big[|\det(G_{d-1}-\rho x\,I_{d-1})|\big]$. As the threshold $u \to \infty$, the Kac-Rice tail integral is asymptotically equivalent to the baseline integral:
\begin{equation}
    C_{k,d}\int_u^\infty E(x)\,\varphi(x)\,\mathrm{d}x \sim C_{k,d}\int_u^\infty (\rho x)^{d-1}\,\varphi(x)\,\mathrm{d}x.
\end{equation}
\end{lemma}

\begin{subequations}
\begin{proof}
To prove the asymptotic equivalence, we must show that the limit of the ratio of the two integrals evaluates to 1 as $u \to \infty$:
\begin{equation}
    L = \lim_{u \to \infty} \frac{\int_u^\infty E(x)\,\varphi(x)\,\mathrm{d}x}{\int_u^\infty (\rho x)^{d-1}\,\varphi(x)\,\mathrm{d}x}.
\end{equation}

The constant $C_{k,d}$ cancels from the ratio, and both integrands are integrable, so numerator and denominator tend to $0$ as $u\to\infty$. By L'Hôpital's rule, using $\frac{\mathrm{d}}{\mathrm{d}u}\int_u^\infty f(x)\,\mathrm{d}x=-f(u)$, differentiating numerator and denominator yields:
\begin{equation}
    L = \lim_{u \to \infty} \frac{\frac{\mathrm{d}}{\mathrm{d}u} \int_u^\infty E(x)\,\varphi(x)\,\mathrm{d}x}{\frac{\mathrm{d}}{\mathrm{d}u} \int_u^\infty (\rho x)^{d-1}\,\varphi(x)\,\mathrm{d}x} = \lim_{u \to \infty} \frac{-E(u)\,\varphi(u)}{-(\rho u)^{d-1}\,\varphi(u)}.
\end{equation}

The negative signs cancel, and because the standard normal density $\varphi(u)$ is strictly positive for all $u$, it also cancels out of the ratio. This leaves:
\begin{equation}
    L = \lim_{u \to \infty} \frac{E(u)}{(\rho u)^{d-1}}.
\end{equation}

By Lemma~\ref{lem:exp_det}, we know that $E(u) \sim (\rho u)^{d-1}$ as $u \to \infty$. By definition of asymptotic equivalence, this means:
\begin{equation}
    \lim_{u \to \infty} \frac{E(u)}{(\rho u)^{d-1}} = 1.
\end{equation}

Since the limit of the ratio of derivatives is $1$, so is the limit of the original ratio, giving the asymptotic equivalence:
\begin{equation}
    \int_u^\infty E(x)\,\varphi(x)\,\mathrm{d}x \sim \int_u^\infty (\rho x)^{d-1}\,\varphi(x)\,\mathrm{d}x.
\end{equation}
Multiplying both sides by the global constant $C_{k,d}$ concludes the proof.
\end{proof}
\end{subequations}

\subsection{The Ben Arous--Dembo--Guionnet lemma}
We first record the Ben~Arous-Dembo--Guionnet large-deviation bound \citep[Lemma-6.3]{ben-guionnet} for the spectral radius of the GOE in the Mehta normalization adopted here.

\begin{lemma}[Spectral radius tail]
\label{lem:bdg}
Let $G_{d-1}\sim\mathrm{GOE}(d-1)$ with eigenvalues $\mu_1,\dots,\mu_{d-1}$, and denote $M_{d-1}=\max_i|\mu_i|$. Then
\begin{equation*}
    \P\{M_{d-1}>t\}\;\le\;e^{-2t^2/9}\qquad\text{for all }t\ge 4\sqrt{2(d-1)}\,.
\end{equation*}
\end{lemma}

\begin{subequations}
\begin{proof}
\noindent
$\bullet$
Let $J_N$ be the symmetric random matrix defined in the Ben Arous--Dembo--Guionnet normalization, with joint eigenvalue density 
$\sigma^N \propto \prod_{i<j}|\lambda_i-\lambda_j|\exp(-\frac{N}{4}\sum_{i=1}^N\lambda_i^2)$. 
Isolating $\lambda_1 = x$ and splitting the Gaussian weight as $e^{-\frac{N}{4}\sum_{i=2}^N \lambda_i^2} = e^{-\frac{N-1}{4}\sum_{i=2}^N \lambda_i^2} e^{-\frac{1}{4}\sum_{i=2}^N \lambda_i^2}$, the remaining interaction exactly matches the $(N-1)$-dimensional density $\sigma^{N-1}$. Thus,
\begin{equation}
    \sigma^N(\lambda_1 \ge M) = \frac{Z_{N-1}}{Z_N} \int_M^\infty e^{-\frac{N}{4}x^2} \mathbb{E}_{\sigma^{N-1}}\!\left[ \prod_{i=2}^N |x-\lambda_i| e^{-\lambda_i^2/4} \right] \mathrm{d}x\,.
\end{equation}
The partition function $Z_N=\int_{\R^N}\prod_{i<j}|\lambda_i-\lambda_j|\,e^{-\frac N4\sum_i\lambda_i^2}\,\mathrm d\lambda$ of this $\beta=1$ Gaussian ensemble is computed from the Mehta integral~\citep[Ch.~17]{mehta} (the $\gamma=\tfrac12$ case of $\int_{\R^N}\prod_{i<j}|x_i-x_j|^{2\gamma}e^{-\frac12\sum x_i^2}\mathrm dx=(2\pi)^{N/2}\prod_{j=1}^N\Gamma(1+j\gamma)/\Gamma(1+\gamma)$):
\begin{equation}\label{eq:mehta_integral}
    M_N\;:=\;\int_{\R^N}\prod_{i<j}|x_i-x_j|\,e^{-\frac12\sum_i x_i^2}\,\mathrm dx\;=\;(2\pi)^{N/2}\prod_{j=1}^N\frac{\Gamma(1+j/2)}{\Gamma(3/2)}\,.
\end{equation}
The rescaling $\lambda_i=\sqrt{2/N}\,x_i$ turns the weight $e^{-\frac N4\sum\lambda_i^2}$ into $e^{-\frac12\sum x_i^2}$, and contributes a Jacobian $(2/N)^{N/2}$ together with a Vandermonde factor $(2/N)^{N(N-1)/4}$ (one factor $\sqrt{2/N}$ per pair, $\binom N2=\tfrac{N(N-1)}2$ pairs), so that
\begin{equation}\label{eq:Z_N_closed}
    Z_N\;=\;(2/N)^{N(N+1)/4}\,(2\pi)^{N/2}\prod_{j=1}^N\frac{\Gamma(1+j/2)}{\Gamma(3/2)}\,.
\end{equation}
Taking the ratio at sizes $N-1$ and $N$: the Gamma product telescopes to $P_{N-1}/P_N=\Gamma(3/2)/\Gamma(1+N/2)$ (with $P_N:=\prod_{j=1}^N\Gamma(1+j/2)/\Gamma(3/2)$), the Gaussian prefactor contributes $(2\pi)^{-1/2}$, and the dimension factors combine through the exponent identity $\tfrac{N(N-1)}4-\tfrac{N(N+1)}4=-\tfrac N2$ into $2^{-N/2}\,(1+\tfrac1{N-1})^{N(N-1)/4}\,N^{N/2}$. Using $(2\pi)^{-1/2}\Gamma(3/2)=\tfrac1{2\sqrt2}$ and $2^{-N/2}N^{N/2}=(N/2)^{N/2}$, this yields the exact identity
\begin{equation}
    \frac{Z_{N-1}}{Z_N} = \left(1+\frac{1}{N-1}\right)^{\frac{N(N-1)}{4}} \frac{(N/2)^{N/2}}{2\sqrt{2}\,\Gamma(1+N/2)} \le e^{N/4}\, \frac{e^{N/2}}{2\sqrt{2\pi N}} \le e^{3N/4}\,.
\end{equation}
Here the first inequality bounds the two factors separately: $\bigl(1+\tfrac1{N-1}\bigr)^{N(N-1)/4}=\exp\bigl(\tfrac{N(N-1)}4\log(1+\tfrac1{N-1})\bigr)\le\exp(\tfrac N4)$ via $\log(1+x)\le x$, while $\tfrac{(N/2)^{N/2}}{2\sqrt2\,\Gamma(1+N/2)}\le\tfrac{e^{N/2}}{2\sqrt{2\pi N}}$ uses the standard lower bound $\Gamma(1+N/2) \ge (N/2)^{N/2} e^{-N/2} \sqrt{\pi N}$; the last step uses $2\sqrt{2\pi N}\ge 1$.
For the integrand, since $\sup_\lambda |\lambda|e^{-\lambda^2/4} = \sqrt{2/e} \le 1$, we have $|x-\lambda_i|e^{-\lambda_i^2/4} \le x + 1$. For $x \ge 8$, $x+1 = x(1+1/x) \le x e^{1/8}$, meaning the product is bounded by $x^{N-1} e^{N/8}$. 

\medskip

\noindent
$\bullet$
Combining the bounds, the probability is bounded by the tail integral:
\begin{equation}
    \sigma^N(\lambda_1 \ge M) \le e^{7N/8} \int_M^\infty x^{N-1} e^{-\frac{N}{4}x^2} \,\mathrm{d}x\,.
\end{equation}
For $x \ge 8$, the logarithmic bound $\ln x \le x^2/30$ holds uniformly: the function $g(x):=x^2/30-\ln x$ has $g'(x)=x/15-1/x>0$ on $x\ge\sqrt{15}$, hence on $[8,\infty)$, and $g(8)=64/30-\ln 8\approx 0.054>0$, so $g(x)\ge g(8)>0$ throughout. This yields $x^{N-1} \le e^{\frac{N-1}{30}x^2} \le e^{\frac{N}{30}x^2}$, and bounds the integral by $\int_M^\infty e^{-N x^2 (1/4 - 1/30)} \mathrm{d}x = \int_M^\infty e^{-13Nx^2/60} \mathrm{d}x \le e^{-13NM^2/60}$ (valid for $NM\ge 60/26\approx 2.31$, automatic under $N\ge 1$, $M\ge 8$).
By a union bound across the $N$ eigenvalues and both signs (using $\sigma^N(\lambda_i\le -M)=\sigma^N(\lambda_1\ge M)$ by the symmetry $\lambda\mapsto-\lambda$ of $\sigma^N$),
\begin{equation}
    \sigma^N\Big(\max_{1\le i\le N}|\lambda_i|\ge M\Big) \le 2N\, \sigma^N(\lambda_1 \ge M) \le \exp\!\Big(-\frac{13NM^2}{60} + \frac{7N}{8} + \ln(2N)\Big)\,.
\end{equation}
To achieve the target bound of $e^{-NM^2/9}$, we require $M^2(\frac{13}{60} - \frac{1}{9}) \ge \frac{7}{8} + \frac{\ln(2N)}{N}$. Since $\frac{13}{60} - \frac{1}{9} = \frac{19}{180}$, evaluated at the constraint $M \ge 8$ (so $M^2 \ge 64$), the left-hand side is at least $64 \times \frac{19}{180} \approx 6.76$. This is strictly greater than $\frac{7}{8} + \frac{\ln(2N)}{N} \le \frac{7}{8}+\log 2+\frac{1}{e}\approx 1.94$ for all $N \ge 1$. Hence the bound holds for all $M \ge 8$.

\medskip

\noindent
$\bullet$
The matrix $G_{d-1}$ drawn from the $\mathrm{GOE}(d-1)$ in the Mehta normalization satisfies the distributional identity $J_{d-1} \stackrel{d}{=} \sqrt{2/(d-1)}\,G_{d-1}$. This establishes the exact spectral relation $\lambda_i(J_{d-1}) = \sqrt{2/(d-1)}\,\mu_i(G_{d-1})$. 
Applying the established bound with $N=d-1$ and substituting $M = t\sqrt{2/(d-1)}$, the threshold condition $M \ge 8$ translates identically to:
\begin{equation}
    t\sqrt{\frac{2}{d-1}} \ge 8 \implies t \ge 4\sqrt{2(d-1)}\,.
\end{equation}
The exponential decay rate maps smoothly via:
\begin{equation}
    -\frac{N M^2}{9} = -\frac{d-1}{9}\left(t\sqrt{\frac{2}{d-1}}\right)^2 = -\frac{2t^2}{9}\,.
\end{equation}
Mapping the probability events gives the required inequality:
\begin{equation}
    \P\Big(\max_{1\le i\le d-1}|\mu_i(G_{d-1})|\ge t\Big) \le e^{-2t^2/9}\,.
\end{equation}
\end{proof}
\end{subequations}

\subsection{Explicit Hermite tail integrals}

\begin{subequations}
\begin{lemma}[Explicit Hermite tail integrals]
\label{lem:hermite_tail_explicit}
For every $u\in\R$ and integer $m\ge 1$ (the integration-by-parts identities below use only the finiteness of the boundary term $H_{m-1}(u)\,e^{-u^2/2}$ at the endpoint $u$ and the vanishing of the integrand at $+\infty$, so no positivity of $u$ is required),
\begin{equation}\label{eq:hermite_tail_series_1}
    \int_u^\infty H_m(x)\,e^{-x^2/2}\,\mathrm{d}x
    \;=\; e^{-u^2/2}\sum_{k=0}^{\lfloor (m-1)/2\rfloor}2^{k+1}\frac{(m-1)!!}{(m-2k-1)!!}H_{m-2k-1}(u)\;+\;R_m(u)\,,
\end{equation}
where $R_m(u)=0$ if $m$ is odd and $R_m(u)=2^{m/2}(m-1)!!\int_u^\infty e^{-x^2/2}\,\mathrm{d}x$ if $m$ is even. The following companion identity (the special weight $e^{-x^2}$) is recorded for completeness and is \emph{not used in the sequel}; the squared-Hermite tails entering $\delta_{\mathrm{exact}}$ are the general-weight form $K_j^\beta$ of Corollary~\ref{cor:Kj_beta}:
\begin{equation*}
    \int_u^\infty H_j(x)^2\,e^{-x^2}\,\mathrm{d}x
    \;=\; e^{-u^2}\sum_{k=0}^{j-1}2^k\frac{j!}{(j-k)!}H_{j-k-1}(u)H_{j-k}(u)+2^j j!\int_u^\infty e^{-x^2}\,\mathrm{d}x\,.
\end{equation*}
\end{lemma}
\end{subequations}

\begin{subequations}
\begin{proof}

\medskip

\noindent
$\bullet$
We evaluate $J_m = \int_u^\infty H_m(x) e^{-x^2/2} \,\mathrm{d}x$.
We use two standard identities for the physicist Hermite polynomials:
\begin{align}
    H_m'(x) &= 2m H_{m-1}(x) \label{eq:hermite_deriv} \\
    H_m(x) &= 2x H_{m-1}(x) - 2(m-1)H_{m-2}(x) \label{eq:hermite_recurrence}
\end{align}
We begin by evaluating the integral of $2x H_{m-1}(x) e^{-x^2/2}$ using integration by parts. Let $f(x) = H_{m-1}(x)$ and $g'(x) = x e^{-x^2/2}$, which implies $g(x) = -e^{-x^2/2}$.
\begin{align*}
    \int_u^\infty x H_{m-1}(x) e^{-x^2/2} \,\mathrm{d}x &= \left[ -H_{m-1}(x) e^{-x^2/2} \right]_u^\infty + \int_u^\infty H_{m-1}'(x) e^{-x^2/2} \,\mathrm{d}x \\
    &= H_{m-1}(u) e^{-u^2/2} + 2(m-1) \int_u^\infty H_{m-2}(x) e^{-x^2/2} \,\mathrm{d}x
\end{align*}
where we substituted \eqref{eq:hermite_deriv} in the second term. We then substitute $2x H_{m-1}(x)$ using a rearrangement of \eqref{eq:hermite_recurrence}, $2x H_{m-1}(x) = H_m(x) + 2(m-1)H_{m-2}(x)$, into the left-hand side:
\begin{equation*}
    \frac{1}{2} \int_u^\infty \Big( H_m(x) + 2(m-1)H_{m-2}(x) \Big) e^{-x^2/2} \,\mathrm{d}x = H_{m-1}(u) e^{-u^2/2} + 2(m-1) \int_u^\infty H_{m-2}(x) e^{-x^2/2} \,\mathrm{d}x
\end{equation*}
Distributing the integral and isolating $J_m = \int_u^\infty H_m(x) e^{-x^2/2} \,\mathrm{d}x$ yields the exact recurrence relation:
\begin{equation}\label{eq:J_recurrence}
    J_m = 2H_{m-1}(u) e^{-u^2/2} + 2(m-1) J_{m-2}
\end{equation}
By iteratively applying \eqref{eq:J_recurrence}, we can expand $J_m$:
\begin{align*}
    J_m &= 2H_{m-1}(u) e^{-u^2/2} + 2(m-1) \Big[ 2H_{m-3}(u) e^{-u^2/2} + 2(m-3) J_{m-4} \Big] \\
    &= e^{-u^2/2} \Big[ 2H_{m-1}(u) + 2^2(m-1)H_{m-3}(u) + 2^3(m-1)(m-3)H_{m-5}(u) + \dots \Big] + \text{Remainder}
\end{align*}
The $k$-th term of this expansion involves the cascading product $(m-1)(m-3)\cdots(m-2k+1)$, which evaluates precisely to the double factorial ratio $\frac{(m-1)!!}{(m-2k-1)!!}$. Incorporating the accumulated factor of $2^{k+1}$, the summation becomes:
\begin{equation*}
    e^{-u^2/2} \sum_{k=0}^{\lfloor (m-1)/2 \rfloor} 2^{k+1} \frac{(m-1)!!}{(m-2k-1)!!} H_{m-2k-1}(u)
\end{equation*}
The recursion stops differently depending on parity:
\begin{itemize}
    \item If $m$ is odd, the final evaluation step corresponds to $k = (m-1)/2$, meaning $H_0(u)$ is evaluated and no subsequent integrals are required. Thus, $R_m(u) = 0$.
    \item If $m$ is even, the final evaluated boundary term corresponds to $k = (m-2)/2$. The remaining integral falls back to $J_0 = \int_u^\infty H_0(x) e^{-x^2/2}\,\mathrm{d}x = \int_u^\infty e^{-x^2/2}\,\mathrm{d}x$. The accumulated multiplier in front of $J_0$ is $2^{m/2} (m-1)!!$, yielding exactly the prescribed $R_m(u)$.
\end{itemize}

\medskip

\noindent
$\bullet$
We evaluate $K_j = \int_u^\infty H_j(x)^2 e^{-x^2} \,\mathrm{d}x$.
Consider the derivative of the product $H_{j-1}(x) H_j(x) e^{-x^2}$:
\begin{equation*}
    \frac{\mathrm{d}}{\mathrm{d}x} \Big[ H_{j-1}(x) H_j(x) e^{-x^2} \Big] = \Big[ H_{j-1}'(x) H_j(x) + H_{j-1}(x) H_j'(x) - 2x H_{j-1}(x) H_j(x) \Big] e^{-x^2}
\end{equation*}
Substitute the derivative identity \eqref{eq:hermite_deriv} for the first two terms:
\begin{equation*}
    \Big[ 2(j-1)H_{j-2}(x)H_j(x) + 2j H_{j-1}(x)^2 - 2x H_{j-1}(x)H_j(x) \Big] e^{-x^2}
\end{equation*}
Using \eqref{eq:hermite_recurrence}, we can replace $2x H_{j-1}(x)$ with $H_j(x) + 2(j-1)H_{j-2}(x)$ in the last term. This gives:
\begin{equation*}
    2(j-1)H_{j-2}(x)H_j(x) + 2j H_{j-1}(x)^2 - \Big( H_j(x) + 2(j-1)H_{j-2}(x) \Big)H_j(x)
\end{equation*}
The cross-terms $2(j-1)H_{j-2}(x)H_j(x)$ cancel, leaving:
\begin{equation*}
    \frac{\mathrm{d}}{\mathrm{d}x} \Big[ H_{j-1}(x) H_j(x) e^{-x^2} \Big] = \Big( 2j H_{j-1}(x)^2 - H_j(x)^2 \Big) e^{-x^2}
\end{equation*}
Integrating this equation from $u$ to $\infty$ yields:
\begin{equation*}
    \left[ H_{j-1}(x) H_j(x) e^{-x^2} \right]_u^\infty = 2j \int_u^\infty H_{j-1}(x)^2 e^{-x^2} \,\mathrm{d}x - \int_u^\infty H_j(x)^2 e^{-x^2} \,\mathrm{d}x
\end{equation*}
Because the exponential $e^{-x^2}$ vanishes at $\infty$, evaluating the bounds leaves:
\begin{equation*}
    - H_{j-1}(u) H_j(u) e^{-u^2} = 2j K_{j-1} - K_j
\end{equation*}
Rearranging this isolates our desired recurrence for $K_j$:
\begin{equation}\label{eq:K_recurrence}
    K_j = H_{j-1}(u) H_j(u) e^{-u^2} + 2j K_{j-1}
\end{equation}
By unrolling \eqref{eq:K_recurrence} successively $j$ times, we extract boundary terms $H_{j-k-1}(u)H_{j-k}(u) e^{-u^2}$ at each step $k$ (where $k$ ranges from $0$ to $j-1$). 
At step $k$, the initial multiplier accumulates as $2j \cdot 2(j-1) \cdots 2(j-k+1)$, which simplifies directly to $2^k \frac{j!}{(j-k)!}$. This gives the exact finite sum provided in the lemma statement. 
After $j$ steps, the recursion completely bottoms out at $K_0 = \int_u^\infty H_0(x)^2 e^{-x^2} \,\mathrm{d}x = \int_u^\infty e^{-x^2} \,\mathrm{d}x$. The accumulated multiplier in front of $K_0$ is $2^j j!$, yielding the exact stated remainder.
\end{proof}
\end{subequations}

\medskip

The next lemma extends Lemma~\ref{lem:hermite_tail_explicit} from the canonical Gaussian weight $e^{-y^2/2}$ to an arbitrary weight $e^{-\beta y^2/2}$ with $\beta>0$. The general $\beta$-version gives the exact squared-Hermite tail $T_d^{\mathrm{exact}}(u)$ used in Theorem~\ref{thm:imf_tail_sharp}: when integrating $H_j(\rho x)^2$ against $e^{-(1+\rho^2)x^2/2}$ on $[u,\infty)$ and substituting $y=\rho x$, the relevant exponent is $\beta=(1+\rho^2)/\rho^2=(3k-2)/k>2$ for $k\ge 3$, a regime to which the existing Lemma~\ref{lem:hermite_tail_explicit} ($\beta=1$) does not apply.

\begin{subequations}
\begin{lemma}[Generalized Hermite--Gaussian tail recurrence]
\label{lem:Jm_beta}
For every $\beta>0$, every integer $m\ge 0$, and every $u\in\R$, set
\begin{equation}\label{eq:Jm_beta_def}
    J_m^\beta(u)\;:=\;\int_u^\infty H_m(y)\,e^{-\beta y^2/2}\,\mathrm{d}y\,.
\end{equation}
Write $\gamma:=2/\beta$ and $\theta:=(2-\beta)/\beta$. The base cases are
\begin{equation}\label{eq:Jm_beta_init}
    J_0^\beta(u)\;=\;\sqrt{\tfrac{2\pi}\beta}\,\bar\Phi(u\sqrt\beta)\,,\qquad
    J_1^\beta(u)\;=\;\gamma\,e^{-\beta u^2/2}\,,
\end{equation}
and for every $m\ge 2$, the family $\{J_m^\beta\}$ satisfies the three-term recurrence
\begin{equation}\label{eq:Jm_beta_rec}
    J_m^\beta(u)\;=\;\gamma\,H_{m-1}(u)\,e^{-\beta u^2/2}\;+\;2(m-1)\,\theta\;J_{m-2}^\beta(u)\,.
\end{equation}
Iterating~\eqref{eq:Jm_beta_rec} yields the explicit closed form
\begin{equation}\label{eq:Jm_beta_iter}
    J_m^\beta(u)\;=\;\gamma\,e^{-\beta u^2/2}\!\!\sum_{k=0}^{\lfloor (m-1)/2\rfloor}(2\theta)^k\,\frac{(m-1)!!}{(m-2k-1)!!}\,H_{m-2k-1}(u)\;+\;R_m^\beta(u)\,,
\end{equation}
where $R_m^\beta(u)=0$ if $m$ is odd, and $R_m^\beta(u)=(2\theta)^{m/2}(m-1)!!\,J_0^\beta(u)$ if $m$ is even.
\end{lemma}

\noindent
For $\beta=1$, $\gamma=2$ and $\theta=1$, so $\gamma(2\theta)^k=2^{k+1}$, and Lemma~\ref{lem:Jm_beta} reduces exactly to~\eqref{eq:hermite_tail_series_1}; the remainder $R_m^\beta$ specialises to $2^{m/2}(m-1)!!\int_u^\infty e^{-y^2/2}\,\mathrm{d}y$, matching the existing $R_m$. For $\beta>2$, $\theta<0$, and the iterated sum~\eqref{eq:Jm_beta_iter} carries alternating signs in $k$ but remains exact pointwise; in subsequent estimates one must not pass to absolute values of individual summands.

\begin{proof}[Proof of Lemma~\ref{lem:Jm_beta}]
\noindent
$\bullet$
Apply integration by parts with $f(y)=H_{m-1}(y)$ and $g'(y)=2y\,e^{-\beta y^2/2}$, hence $g(y)=-(2/\beta)\,e^{-\beta y^2/2}=-\gamma\,e^{-\beta y^2/2}$. The boundary term at $+\infty$ vanishes by Gaussian decay, and using $H_{m-1}'(y)=2(m-1)\,H_{m-2}(y)$,
\[
    \int_u^\infty 2y\,H_{m-1}(y)\,e^{-\beta y^2/2}\,\mathrm{d}y
    \;=\;\gamma\,H_{m-1}(u)\,e^{-\beta u^2/2}\;+\;2(m-1)\,\gamma\,J_{m-2}^\beta(u)\,.
\]
Combining with the Hermite recurrence $H_m(y)=2y\,H_{m-1}(y)-2(m-1)\,H_{m-2}(y)$,
\[
    J_m^\beta(u) \;=\; \int_u^\infty 2y\,H_{m-1}(y)\,e^{-\beta y^2/2}\,\mathrm{d}y\;-\;2(m-1)\,J_{m-2}^\beta(u)\,,
\]
and $2(m-1)\gamma-2(m-1)=2(m-1)(\gamma-1)=2(m-1)\theta$, yielding~\eqref{eq:Jm_beta_rec}.

\medskip

\noindent
$\bullet$ For the boundary cases, $J_0^\beta(u)=\int_u^\infty e^{-\beta y^2/2}\,\mathrm{d}y=\sqrt{2\pi/\beta}\,\bar\Phi(u\sqrt\beta)$ by direct change of variable, and $J_1^\beta(u)=\int_u^\infty 2y\,e^{-\beta y^2/2}\,\mathrm{d}y=(2/\beta)\,e^{-\beta u^2/2}=\gamma\,e^{-\beta u^2/2}$.

\medskip

\noindent
$\bullet$ Unrolling~\eqref{eq:Jm_beta_rec} $k$ times and using the telescoping product $(m-1)(m-3)\cdots(m-2k+1)=(m-1)!!/(m-2k-1)!!$,
\[
    J_m^\beta(u)\;=\;\gamma\,e^{-\beta u^2/2}\!\!\sum_{j=0}^{k-1}(2\theta)^j\,\frac{(m-1)!!}{(m-2j-1)!!}\,H_{m-2j-1}(u)\;+\;(2\theta)^k\,\frac{(m-1)!!}{(m-2k-1)!!}\,J_{m-2k}^\beta(u)\,.
\]
For $m$ odd, take $k=(m-1)/2$: $m-2k-1=0$, hence $(m-2k-1)!!=0!!=1$, the residual factor $J_{m-2k}^\beta(u)=J_1^\beta(u)=\gamma\,e^{-\beta u^2/2}$, and the residual term coincides with the $j=(m-1)/2$ summand of the main sum (since $H_{m-2k-1}(u)=H_0(u)=1$ and $\gamma e^{-\beta u^2/2}=J_1^\beta(u)$); absorbing it into the main sum gives $R_m^\beta=0$. For $m$ even, take $k=m/2$: $m-2k-1=-1$, $(m-2k-1)!!=(-1)!!=1$, the residual factor is $J_0^\beta(u)$, and $R_m^\beta(u)=(2\theta)^{m/2}\,(m-1)!!\,J_0^\beta(u)$.

\medskip

\noindent
$\bullet$ \emph{Boundary verification ($m=1,2$).} For $m=1$, the sum~\eqref{eq:Jm_beta_iter} reduces to $\gamma\,e^{-\beta u^2/2}\,H_0(u)=\gamma\,e^{-\beta u^2/2}$, matching $J_1^\beta(u)$. For $m=2$, the main sum is $\gamma\,e^{-\beta u^2/2}\,H_1(u)=2\gamma u\,e^{-\beta u^2/2}$ and $R_2^\beta(u)=2\theta\,J_0^\beta(u)$; this matches the direct integration $J_2^\beta(u)=\int_u^\infty(4y^2-2)e^{-\beta y^2/2}\,\mathrm dy=2\gamma u\,e^{-\beta u^2/2}+2\theta\,J_0^\beta(u)$ obtained via integration by parts.
\end{proof}
\end{subequations}

\begin{corollary}[Squared-Hermite Gaussian tail]
\label{cor:Kj_beta}
\begin{subequations}
For every $\beta>0$, every integer $j\ge 0$, and every $u\in\R$, the squared-Hermite tail
\begin{equation}\label{eq:Kj_beta_def}
    K_j^\beta(u)\;:=\;\int_u^\infty H_j(y)^2\,e^{-\beta y^2/2}\,\mathrm{d}y
\end{equation}
admits the exact closed form
\begin{equation}\label{eq:Kj_beta_form}
    K_j^\beta(u)\;=\;\sum_{p=0}^j 2^p\,p!\,\binom{j}{p}^{\!2}\,J_{2j-2p}^\beta(u)\,,
\end{equation}
where each $J_{2j-2p}^\beta$ is given in closed form by Lemma~\ref{lem:Jm_beta}.
\end{subequations}
\end{corollary}

\begin{proof}[Proof of Corollary~\ref{cor:Kj_beta}]
The classical linearization of the product of two physicist Hermite polynomials reads
\[
    H_m(y)\,H_n(y)\;=\;\sum_{p=0}^{\min(m,n)} 2^p\,p!\,\binom{m}{p}\binom{n}{p}\,H_{m+n-2p}(y)\,.
\]
Specialising to $m=n=j$ yields $H_j(y)^2=\sum_{p=0}^j 2^p\,p!\,\binom{j}{p}^{\!2}\,H_{2j-2p}(y)$. Multiplying by $e^{-\beta y^2/2}$, integrating from $u$ to $\infty$, and exchanging the finite sum with the integral (justified termwise since each $H_m\,e^{-\beta y^2/2}$ is absolutely integrable on $[u,\infty)$), one obtains~\eqref{eq:Kj_beta_form}.
\end{proof}

\section{List of notations}
\label{app:notations}

\paragraph{Tensor regression model}
\begin{center}
\renewcommand{\arraystretch}{1.2}
\begin{tabular}{p{2.5cm} p{12cm}}
    $k \ge 2$ & Order of the symmetric tensors (all theorems require $k,d\ge 3$). \\
    $d \ge 2$ & Ambient dimension of the underlying vectors (tensors have $d^k$ entries). \\
    $n = d-1$ & Effective dimension (associated to the sphere and the GOE matrices). \\
    $R \ge 1$ & Rank of the signal tensor $\meanTensor$ (and rank bound for the candidate tensors $\putativemeanTensor$). \\
    $\kappa \in (0, 1]$ & Coherence parameter. \\
    $\lambda > 0$ & Signal-to-noise ratio. \\
    $\obs, \noise, \meanTensor$ & Observed tensor, standard Gaussian noise tensor, and planted (true) normalized signal tensor $\sigma^\star$, respectively. \\
    $\putativemeanTensor,\ \hat{\putativemeanTensor}$ & Generic candidate tensor and the profile MLE~\eqref{eq:MLE_definition}. \\
    $\mathcal{C}_{R,\kappa},\ \bm G$ & Feasible set $\{\putativemeanTensor\in\mathfrak{S}_R:\kappa(\putativemeanTensor)\ge\kappa\}$; coherence $\kappa(\cdot)$ from~\eqref{def:coherence} via the Gram matrix $\bm G$, $G_{ij}=\langle t_i,t_j\rangle^k$. \\
    $\langle \cdot, \cdot \rangle_\tensors, \|\cdot\|_F$ & Tensor canonical inner product and Frobenius norm. \\
    $\Tensors$ & Space of symmetric tensors of order $k$ and dimension $d$. \\
    $\sphereTensors$ & Unit sphere in $\Tensors$. \\
    $\mathfrak{S}_R$ & Set of symmetric normalized tensors of rank at most $R$.
\end{tabular}
\end{center}

\paragraph{Geometry and Random Fields}
\begin{center}
\renewcommand{\arraystretch}{1.2}
\begin{tabular}{p{2.5cm} p{12cm}}
    $\mathbb{S}^{d-1}$ & Unit sphere in $\R^d$, which has dimension $n$. \\
    $|\mathbb{S}^{d-1}|$ & Surface area of the unit sphere, $|\mathbb{S}^{d-1}| = 2\pi^{d/2}/\Gamma(d/2)$. \\
    $T_{\bm{\theta}}\mathbb{S}^{d-1}$ & Tangent space to the sphere at point $\bm{\theta}$, also of dimension $n$. \\
    $X(\bm{\theta})$ & Centered Gaussian random field defined on the sphere by $X(\bm{\theta}) = \langle \noise, \bm{\theta}^{\otimes k} \rangle_\tensors$ (the Kostlan--Shub--Smale random field). \\
    $\nabla X, \nabla^2 X$ & Riemannian gradient and Riemannian Hessian of the random field $X$. \\
    $\Gamma_{R,\kappa}$ & Supremum of the noise over the manifold constraints defined by $\mathfrak{S}_R$ and~$\kappa$. \\
    $\Gamma_{1,1}$ & Two-sided absolute supremum of the field, $\Gamma_{1,1} = \sup_{\bm{\theta}} |X(\bm{\theta})|$. \\
    $\widetilde\Gamma_{1,1}(u)$ & One-sided excursion probability $\widetilde\Gamma_{1,1}(u) = \P\{\sup_{\bm{\theta}} X(\bm{\theta}) > u\}$, related to $\Gamma_{1,1}$ by $\P\{\Gamma_{1,1}>u\}\le 2\,\widetilde\Gamma_{1,1}(u)$.
\end{tabular}
\end{center}

\paragraph{Random Matrices and Spectra}
\begin{center}
\renewcommand{\arraystretch}{1.2}
\begin{tabular}{p{2.5cm} p{12cm}}
    $G_{d-1}$ & Random $n \times n$ matrix drawn from the Gaussian Orthogonal Ensemble (GOE). \\
    $\mu_1,\dots,\mu_{d-1}$ & Real eigenvalues of the matrix $G_{d-1}$. \\
    $M_{d-1}$ & Spectral radius of $G_{d-1}$, $M_{d-1} = \max_i |\mu_i|$. \\
    $\rho$ & Scaling constant appearing for the GOE shift, defined as $\rho = \sqrt{k/(2(k-1))}$. \\
    $I_{d-1}$ & Identity matrix of dimension $n$.
\end{tabular}
\end{center}

\paragraph{Special Functions and Analysis}
\begin{center}
\renewcommand{\arraystretch}{1.2}
\begin{tabular}{p{2.5cm} p{12cm}}
    $\bar\Phi(u)$ & Standard Gaussian tail, $\bar\Phi(u) = \P(Z>u)$. \\
    $\varphi(x)$ & Density function of the standard normal distribution, $\varphi(x) = (2\pi)^{-1/2} e^{-x^2/2}$. \\
    $\Gamma(\cdot)$ & Gamma function. \\
    $H_j(x)$ & Hermite polynomial of degree $j$ (physicists' convention, $H_n(x)\sim(2x)^n$). \\
    $q_d, Q_d$ & Weighted and unweighted characteristic polynomials appearing in the Fyodorov--Mehta orthogonal expansion. \\
    $u$ & Threshold level for bounding the exceedance probability. \\
    $I_m(u;a)$ & Incomplete variance integral, $I_m(u;a) = \int_u^\infty x^m e^{-ax^2/2} \,\mathrm{d}x$.
\end{tabular}
\end{center}

\paragraph{Kac--Rice integral, tail bounds, and thresholds}
\begin{center}
\renewcommand{\arraystretch}{1.2}
\begin{tabular}{p{2.7cm} p{11.8cm}}
    $C_{k,d}$ & Kac--Rice prefactor, $C_{k,d}=2\sqrt\pi\,(k-1)^{(d-1)/2}/\Gamma(d/2)$. \\
    $\delta_0(u)$ & Kac--Rice integral~\eqref{eq:KR_combined_intro} bounding $\widetilde\Gamma_{1,1}(u)$. \\
    $\delta_{\mathrm{bl}}(u)$ & Asymptotic baseline~\eqref{eq:asymptotic_baseline_approx}; $\delta_0\sim\delta_{\mathrm{bl}}$ as $u\to\infty$ (not a bound). \\
    $\delta_{\mathrm{exact}}(u)$ & Exact closed-form evaluation of $\delta_0$ (Theorem~\ref{thm:imf_tail_exact}). \\
    $\delta_{\mathrm{IMF}},\ \delta_{\mathrm{IMF}}^\star$ & Improved Mehta--Fyodorov bound and its sharpened variant (Theorems~\ref{thm:imf_tail},~\ref{thm:imf_tail_sharp}). \\
    $\delta_{\mathrm{SMF}},\ \delta_{\mathrm{SM}}$ & Simplified Mehta--Fyodorov and spectral-method bounds (Theorems~\ref{thm:smf_tail},~\ref{thm:sm_tail}). \\
    $\delta_{\min}(u)$ & Master failure probability $\delta_{\min}=2\,\delta_{\mathrm{IMF}}$ (Theorem~\ref{thm:main}). \\
    $u_{\mathrm{IMF}},u_{\mathrm{SMF}},u_{\mathrm{SM}}$ & Validity thresholds $\sqrt{2d-1}/\rho$, $2\sqrt d$, $32\sqrt{d-1}/\rho$. \\
    $u_\alpha,\ \alpha_0(k,d)$ & Confidence-level inversion~\eqref{eq:u_alpha} and its validity threshold (Remark~\ref{rem:choice_u}).
\end{tabular}
\end{center}

\paragraph{Mehta--Fyodorov constants and Hermite-tail integrals}
\begin{center}
\renewcommand{\arraystretch}{1.2}
\begin{tabular}{p{2.7cm} p{11.8cm}}
    $\Lambda,\ \beta$ & $\Lambda=2\rho^2-1=1/(k-1)$; $\beta=(1+\rho^2)/\rho^2=(3k-2)/k$. \\
    $c_j,\ \mu_m$ & $c_j=(2^j j!\sqrt\pi)^{-1/2}$; $\mu_m=\int_\R H_m(y)e^{-y^2/2}\,\mathrm{d}y$. \\
    $\alpha_d,\ \beta_d$ & Dominant Hermite coefficient and remainder envelope in $Q_d=\alpha_d H_{d-1}+\mathcal R_d$ (Lemma~\ref{lem:dominant_term}). \\
    $\Phi_d,\ \Psi_d$ & Polynomial--rational functions of the IMF decomposition (Proposition~\ref{prop:IMF_decomp}). \\
    $\eta_d(\rho,u)$ & Layer-cake correction of the SM bound (Proposition~\ref{prop:det_bound}). \\
    $\mathcal I_d^c(\nu)$ & Hermite tail integral $\int_\nu^\infty H_d(y)e^{-y^2/2}\,\mathrm{d}y$. \\
    $\mathcal L_d,\ \mathcal C_d$ & Linear Hermite tail and cross integral in the $\delta_{\mathrm{exact}}$ decomposition. \\
    $D_1,\dots,D_4$ & The four pieces of $\delta_{\mathrm{exact}}$ (Theorem~\ref{thm:imf_tail_exact}). \\
    $J_m^\beta,\ K_j^\beta,\ T_d^{\mathrm{exact}}$ & Generalised Hermite and squared-Hermite tail integrals (Lemma~\ref{lem:Jm_beta}, Corollary~\ref{cor:Kj_beta}).
\end{tabular}
\end{center}

\paragraph{Spherical $k$-spin complexity (Section~\ref{sec:refinements})}
\begin{center}
\renewcommand{\arraystretch}{1.2}
\begin{tabular}{p{2.7cm} p{11.8cm}}
    $N^{\mathrm{lm}}_{[E,\infty)},\ N^{\mathrm{cp}}_{[E,\infty)}$ & Number of local maxima, resp.\ critical points, of $X$ with value $\ge E$. \\
    $E,\ e$ & Energy level and reduced energy $e=\rho E/\sqrt{2(d-1)}$. \\
    $E_\infty,\ E_0$ & ABC spectral-edge energy $2\sqrt{(k-1)/k}$ and ground-state energy $E_0>E_\infty$. \\
    $E_{\mathrm{BDG}},\ u_{\mathrm{ABC}}$ & Two-sided-bracket threshold $8\sqrt{2(d-1)}/\rho$; ABC energy scale $u_{\mathrm{ABC}}=E/\sqrt d$. \\
    $\Theta_k,\ \Theta_p,\ \Theta_{0,p}$ & ABC complexity functions: local-maximum rate of the $k$-spin field, and total resp.\ index-$0$ rates of the $p$-spin model (Corollary~\ref{cor:abc_match}). \\
    $C_{\mathrm{amp}},\ \delta_{\mathrm{BDG}}$ & Amplifier constant $8\sqrt2$ and GOE spectral-radius bound $\delta_{\mathrm{BDG}}(\rho E)=e^{-2(\rho E)^2/9}$ (Lemma~\ref{lem:bdg}).
\end{tabular}
\end{center}

\end{document}